\documentclass[11pt,centertags,oneside]{article}
\usepackage{amsmath,amstext,amsthm,amscd,typearea}
\usepackage{amssymb}
\usepackage{a4wide}
\usepackage[mathscr]{eucal}
\usepackage{mathrsfs}
\usepackage{typearea}
\usepackage{charter}
\usepackage{pdfsync}
\usepackage{url}
\usepackage[T1]{fontenc}

\usepackage{hyperref}

\usepackage{xcolor}

\usepackage[a4paper,width=15.2cm,top=4cm,bottom=4cm]{geometry}

\numberwithin{equation}{section}

\newtheorem{theorem}{Theorem}[section]

\newtheorem{proposition}[theorem]{Proposition}

\newtheorem{lemma}[theorem]{Lemma}
\newtheorem{remark}[theorem]{Remark}

\newtheorem{problem}[theorem]{Problem}

\newcommand{\cali}[1]{\mathscr{#1}}

\newcommand{\supp}{{\rm Supp}}

\newcommand{\dist}{\mathop{\mathrm{dist}}\nolimits}

\newcommand{\vol}{\mathop{\mathrm{vol}}}

\newcommand{\ddc}{dd^c}
\newcommand{\dc}{d^c}

\newcommand{\PSH}{{\rm PSH}}

\newcommand{\id}{{\rm id}}

\newcommand{\Cc}{\cali{C}}

\newcommand{\capK}{\text{cap}}

\newcommand{\B}{\mathbb{B}}
\newcommand{\C}{\mathbb{C}}

\newcommand{\N}{\mathbb{N}}

\newcommand{\R}{\mathbb{R}}

\title{\bf  Quantitative stability for the complex Monge-Amp\`ere equations}
\providecommand{\keywords}[1]{\textbf{\textit{Keywords:}} #1}
\providecommand{\subject}[1]{\textbf{\textit{Mathematics Subject Classification 2010:}} #1}

\author{Hoang-Son Do and Duc-Viet Vu}

\newcommand{\Addresses}{{
		\bigskip
		\footnotesize
		\textsc{Duc-Viet Vu, University of Cologne, Division of Mathematics, Department of Mathematics and Computer Science, Weyertal 86-90, 50931, K\"oln,  Germany}
		\noindent
		\par\nopagebreak
		\noindent
		\textit{E-mail address}: \texttt{vuviet@math.uni-koeln.de}	
		
	\bigskip
\footnotesize
\textsc{Hoang-Son Do, Vietnam Academy of Science and Technology, Institute of Mathematics, 18 Hoang Quoc Viet road, Cau Giay, Hanoi, Vietnam}
\noindent
\par\nopagebreak
\noindent
\textit{E-mail address}: \texttt{dhson@math.ac.vn}	}}
\date{\today}
\begin{document}
\maketitle
\begin{abstract} We prove several quantitative stability estimates for solutions of complex Monge-Amp\`ere equations when \emph{both} the cohomology class and the  prescribed singularity vary. In a broad sense, our results  fit well into the study of degeneration of families of K\"ahler-Einstein metrics. The key mechanism in our method is the pluripotential theory in the space of potentials of finite lower energy.
\end{abstract}
\noindent
\keywords {Monge-Amp\`ere equation}, {convex weights}, {lower energy}, {non-pluripolar products}.
\\

\noindent
\subject{32U15}, {32Q15}.

\tableofcontents

\section{Introduction}

Let $(X,\omega)$ be a compact K\"ahler manifold of dimension $n$ and let $\alpha$ be a big cohomology $(1,1)$-class in $X$. Let $\theta$ be a closed smooth real $(1,1)$-form in $\alpha$ and let $\mu$ be a non-pluripolar finite measure on $X$. Consider the complex Monge-Amp\`ere equation 
\begin{align}\label{eq-MA}
\theta_u^n = \mu,
\end{align}
where $u$ is a $\theta$-psh function, and $\theta_u:= \ddc u+\theta$, and the left-hand side of (\ref{eq-MA}) denotes the non-pluripolar self-product of $\theta_u$ (see \cite{BT_fine_87,BEGZ,GZ-weighted,Viet-generalized-nonpluri}). By monotonicity of non-pluripolar products (see \cite{Lu-Darvas-DiNezza-mono,Viet-generalized-nonpluri,WittNystrom-mono}), if (\ref{eq-MA}) has a solution, then it is necessary that $\mu(X) \le \vol(\alpha)$, where $\vol(\alpha)$ denotes the volume of the  big class $\alpha$. When $\mu(X)= \vol(\alpha)$, the equation (\ref{eq-MA}) admits a unique solution by \cite{BEGZ,Cegrell,Dinew-uniqueness,Kolodziej_Acta,Yau1978}, and this solution is of minimal singularity in $\alpha$ if $\mu$ is sufficiently regular (for example, $\mu$ has a $L^p$ ($p>1$) density with respect to a smooth volume form on $X$). 

One expects that the regularity of solutions agrees well with that of the measure $\mu$. This expectation is true at least for the following two classes of extreme regularities. The first one is the class of measures which are H\"older continuous as a linear functional on the space $\PSH_0(X, \omega)$ of $\omega$-psh functions $u$ with $\int_X u \omega^n =0$ endowed with $L^1$-metric (we call such measures \emph{H\"older continuous ones}). The second one is  the class of measures of finite lower energy (\emph{i.e,} non-pluripolar measures). These two classes are important because they are two regularities governing the range of measures where (\ref{eq-MA}) is solvable (within the framework of the theory of non-pluripolar products of currents).  We refer to \cite{DemaillyHiep_etal,DKC_Holder-Sobolev,DinhVietanhMongeampere,Vu_Do-MA,KC-holder2018,Kolodziej08holder,Lu-To-Phung,NgocCuong-Holder-calc.var,NgocCuong-Holder2020,Tosatti-Weinkove,Vu_MA,Vu_MA_holder} and references therein for more informations in the setting where $\mu(X)= \vol(\alpha)$. 

 Consider now the case where the mass of $\mu$ is not necessarily equal to $\vol(\alpha)$, \emph{i.e,} where $\mu(X) \le \vol(\alpha)$. In this case one can still solve (\ref{eq-MA}) by putting it in the context of prescribed singularity. We need several notions. Denote by $\PSH(X,\theta)$ the set of $\theta$-psh functions.  Let $u_1,u_2 \in \PSH(X,\theta)$. Recall that $u_1$ is more singular than $u_2$ if $u_1 \le u_2+ O(1)$, and $u_1$ is of the same singularity type as $u_2$ if $u_1-u_2$ is bounded.

 Let $\phi \in \PSH(X, \theta)$ such that $\phi \le 0$ and $\int_X \theta_\phi^n >0$. Denote by $\PSH(X,\theta, \phi)$ the set of $\theta$-psh functions $u$ with $u \le \phi$. Note that it is slightly different from the usual definition of $\PSH(X,\theta,\phi)$ in which  $u$ is only required to be more singular than $\phi$. This difference is not essential.  We say that $\phi$ is a \emph{model $\theta$-psh function} (see \cite{Lu-Darvas-DiNezza-mono,Ross-WittNystrom}) if $\phi= P_\theta[\phi]$ and $\int_X \theta_\phi^n >0$, where 
$$P_\theta[\phi]:= \big(\sup \{ \psi \in \PSH(X,\theta): \psi \le 0, \, \psi \le \phi+ O(1)\}\big)^*.$$
The function $P_\theta[\phi]$ is called a roof-top envelope in \cite{Lu-Darvas-DiNezza-mono}.  By \cite{Lu-Darvas-DiNezza-mono}, the function $P_\theta[u]$ is a model one for every $u \in \PSH(X,\theta)$ with $\int_X \theta_u^n >0$, and for every $u \in \PSH(X,\theta,\phi)$ with $\int_X \theta_u^n= \int_X \theta_\phi^n$ we have $P_\theta[u]= P_\theta[\phi]$.  

Let $\phi$ be now a model $\theta$-psh function.  Let $\mu$ be a non-pluripolar measure with $\mu(X)= \int_X \theta_\phi^n$. We want to solve the equation 
 \begin{align}\label{eq-MAphi}
(\ddc u+ \theta)^n  = \mu,
\end{align}
for $u \in \PSH(X,\theta, \phi)$ and $\sup_X(u-\phi)=0$. We note that since $\phi$ is a model, if $u \in \PSH(X, \theta)$ such that $u \le 0$ and  $u \le \phi+ O(1)$, then $u \le \phi$, and $\sup_X(u-\phi)=\sup_X u$. Thus the normalization condition $\sup_X(u-\phi)=0$ can be rewritten as $\sup_X u=0$.   

The hypothesis that $\phi$ is model is a minimal requirement so that (\ref{eq-MAphi}) is solvable in a meaningful way; see \cite{Lu-Darvas-DiNezza-mono} for an explanation about the nature of this assumption.  Let $\mathcal{E}(X, \theta,\phi)$ be the set of $u \in \PSH(X,\theta,\phi)$ such that $\int_X \theta_u^n =\int_X \theta_\phi^n$. By \cite{Lu-Darvas-DiNezza-logconcave} (or \cite{Lu-Darvas-DiNezza-mono,Vu_Do-MA}), the equation (\ref{eq-MAphi}) admits a unique solution in $\mathcal{E}(X, \theta,\phi)$, and if $\mu$ has $L^p$ density then the solution is of the same singularity type as $\phi$. Furthermore a characterization of the class of measures $\mu$ where  (\ref{eq-MAphi}) admits a solution of finite pluricomplex energy  was given in  \cite{Vu_Do-MA}. 

A $\theta$-singularity type (in $\alpha$) is an equivalence class of $\theta$-psh functions of the same singularity type. The space of  $\theta$-singularity types is denoted by $\mathcal{S}(\theta)$(or $\mathcal{S}(\alpha)$ when $\theta$ is clear from the context).  A natural pseudo-metric $d_{\mathcal{S}(\theta)}$ in $\mathcal{S}(\theta)$ was introduced in \cite{Darvas-Lu-DiNezza-singularity-metric}. We refer to  Section \ref{sec-variedtype} for a recap of this pseudodistance.    A model $\theta$-singularity type is by definition the class of a model $\theta$-psh function. By \cite[Theorem 1.3]{Lu-Darvas-DiNezza-mono}, every model $\theta$-singularity type contains a unique model $\theta$-psh function. Hence there is a 1-1 correspondence between model $\theta$-singularity types and model $\theta$-psh functions.  For $u \in \PSH(X,\theta)$, we denote by $[u]_\theta$ (or simply $[u]$ when $\theta$ is clear) the $\theta$-singularity type of $u$. To ease the notation we will denote by $d_{\mathcal{S}(\theta)}(u,v)$ the distance $d_{\mathcal{S}(\theta)}\big([u]_\theta, [v]_\theta\big)$.

In Proposition \ref{pro-equivdistance} (in Section \ref{sec-variedtype}), we push further the study of metrics on the space of singularity types by observing that if we embed $\mathcal{S}(\theta)$ into a bigger space $\mathcal{S}(\theta')$ for $\theta' \ge \theta$ (notice that $\theta'$ is not necessarily in the cohomology class of $\theta$), then the pseudodistance $d_{\mathcal{S}(\theta)}$ is actually comparable with that induced by $d_{\mathcal{S}(\theta')}$. This allows us to compare singularity types in different cohomology classes without changing the nature of the distance $d_{\mathcal{S}(\theta)}$. By this we will sometimes ignore $\theta$ and  only write $d_\mathcal{S}$.  In view of the resolution of (\ref{eq-MAphi}), we are  led to the following natural stability question. We fix a $\Cc^0$-norm on the space of smooth $(1,1)$-forms on $X$. 

\begin{problem} \label{pro-stability} Let $\theta_1,\theta_2$ be closed smooth real $(1,1)$-forms on $X$.  Let $\phi_j$ be model $\theta_j$-psh functions and $\mu_j$ be a non-pluripolar measure of mass equal to $\int_X \theta_{\phi_j}^n$ for $j=1,2$. Let  $u_j$ be  the solution of  (\ref{eq-MAphi}) for $\mu_j,\phi_j$ for $j=1,2$. Compare $u_1$ with $u_2$ in terms of $d_{\mathcal{S}}(\phi_1, \phi_2)$, $\|\theta_1- \theta_2\|_{\Cc^0}$, and  a suitable distance between $\mu_1,\mu_2$? 
\end{problem}

Here by $d_{\mathcal{S}}(\phi_1, \phi_2)$, we mean $d_{\mathcal{S}(A \omega)}(\phi_1, \phi_2)$, where $A$ is a big constant so that $\theta_j \le A \omega$ for $j=1,2$. As discussed above, the condition that $d_{\mathcal{S}(A \omega)}(\phi_1, \phi_2)$ converges to $0$ is independent of the choice of $A$. To get motivated about the above problem, let's consider the following simple situation. Let $(\alpha_j)_j$ be  a sequence of cohomology K\"ahler $(1,1)$-classes converging to a big class $\alpha_\infty$ as $j \to \infty$. We know that there exists a unique closed positive $(1,1)$-current $T_j \in \alpha_j$ such that $T_j^n = \vol(\alpha_j) \omega^n/ \int_X \omega^n$. One thus asks further: what can we say about the convergence of the sequence $(T_j)_j$? Even when $\alpha_\infty$ is also K\"ahler, it seems that known methods are not sufficient to deal with such a question.

We will develop in this paper  a quite satisfactory method to treat the above stability problem. The emphasis of our approach is the quantitative point of view. As it will be clear later, even when one is only interested in obtaining qualitative stability as in the above simplified situation (with varied cohomology classes), it is still essential in proofs to obtain beforehand quantitative stability estimates. To be more precise, \emph{one of the main protagonists in our work is a quantitative stability for solutions to (\ref{eq-MAphi}) of finite lower energy in the setting where the cohomology class and the prescribed singularity are fixed, \emph{i.e,} where $\theta$ and $\phi$ are fixed.} To our best knowledge, no such estimate was established in the literature. Consider in the simplest setting when $\theta= \omega$ and $\phi=0$, and let $\mu_j= (\ddc u_j+\omega)^n$ be a non-pluripolar measure of mass equal to $\int_X \omega^n$  and $\sup_X u_j=0$ for $j=1,2$. If $\mu_1= \mu_2$, then it is well-known that $u_1= u_2$ by \cite{Dinew-uniqueness}. However there has been no available result comparing $u_1,u_2$ when $\mu_1,\mu_2$ are close to each other.  This is due to the fact that arguments in  \cite{Dinew-uniqueness} (and in other known proofs of this uniqueness property, see  (\cite{BEGZ,Lu-Darvas-DiNezza-mono,Dinew-uniqueness,Lu-To-Phung}) are non-quantitative.

In another aspect, the stability of solutions when the cohomology class varies is closely related to the question of degenerations of special K\"ahler metrics on manifolds, or more generally, families of K\"ahler-Einstein metrics. There is  rich literature on this topic; see for example   \cite{DDG-family,Hein-Tosatti,Hein-Tosatti-higherorder,Tosatti-collapsing,Tosatti-noncollapse,Tosatti-Weinkove-Yang}. We would like to stress that although in some typical model of degenerations of Ricci flat K\"ahler metric an optimal local $\Cc^\infty$  convergence of potentials (\emph{i.e}, solutions) on some Zariski open subset of the ambient manifold  was obtained in \cite{Hein-Tosatti,Hein-Tosatti-higherorder,Tosatti-collapsing,Tosatti-noncollapse}, it seems that the global convergence of potentials (solutions) has not been  well-studied. Our work fits well into this research direction. 

The first stability result for varied prescribed singularities, which is not quantitative, was given in  \cite[Theorem 1.4]{Darvas-Lu-DiNezza-singularity-metric}. Previously there were several stability results in the fixed prescribed singularity setting in the literature: some are quantitative and some are not. We refer to \cite{BBGZ-variational,Blocki_stability,Dinew_Zhang_stability,Kolodziej05,KC_hermitian,Guedj-Zeriahi-big-stability,Lu-To-Phung,Vu_MA_holder} and references therein for more details. Key technical tools to obtain quantitative stability has been so far  (variants of) Ko{\l}odziej's capacity method (\cite{Kolodziej05}) and an integration by parts arguments originally in \cite{Blocki_stability}.  All of these cited results require the measures in the right-hand side of the Monge-Amp\`ere equations to be sufficiently regular (to be more precise, measures must be at least the Monge-Amp\`ere of  $\theta$-psh functions in $\mathcal{E}^1(X, \theta)$). 

Finally, we underline that our interest in the stability of solutions also comes from complex dynamics because equilibrium measures associated to holomorphic dynamical systems are, in many important cases, natural Monge-Amp\`ere measures; see  \cite{DNT_equi,DS_book}. Stability of solutions of (\ref{eq-MA}) is hence relevant to the bifurcation theory of these holomorphic dynamical systems (see \cite{Berteloot_cours}).  We  refer  \cite{Tosatti-survey-dynamics} for a recent application of Monge-Amp\`ere equations to dynamical systems and vice versa. 

We  would like to inform that in a forthcoming paper the second-named author will apply the method in this paper to attack the problem of degeneration of conic K\"ahler-Einstein metrics proposed in \cite{Biquard-Guenancia}.

\subsection{Statement of main results}

In order to motivate readers about our first main result, let us consider the following simple setting. Let $\omega$ be a K\"ahler form on $X$. Let $u,v$ bounded negative $\omega$-psh functions. For measures $\mu,\nu$ on $X$, we denote by $\|\mu-\nu\|$ the mass norm of $\mu-\nu$.  It was proved by  B{\l}ocki \cite{Blocki_stability} that 
\begin{align}\label{ine-gradientchibangt}
\int_X d(u-v) \wedge \dc (u-v) \wedge \omega^{n-1} \le C_{M,n}\bigg(\int_X (u-v)(\omega_v^n- \omega_u^n) \bigg)^{2^{1-n}},
\end{align}
where $M$ is an upper bound for $\|u\|_{L^\infty}+ \|v\|_{L^\infty}$, and $C_{M,n}$ is a constant depending only on $n,M$. This combined with Poincar\'e's inequality gives
\begin{align}\label{blocki-stab-ine}
\|u-v\|_{L^1} \le C_{\omega,M,n}\|\omega_u^n- \omega_v^n\|^{2^{-n}},
\end{align}
where $C_{\omega,M,n}$ is a constant depending on $\omega,M,n$. One can indeed make $C_{\omega,M,n}$ independent of $\omega$ (but dependent on $\int_X \omega^n$) but the exponent in the right-hand side of (\ref{blocki-stab-ine}) must be changed a bit. We will ignore, for the moment, this detail. The stability estimate (\ref{blocki-stab-ine}) is the starting point of our journey in this paper. Observe that the constant $C_{\omega,M,n}$ in the  right-hand side of (\ref{blocki-stab-ine}) depends \emph{both} on $u$ and $v$, or more generally, on the regularity of \emph{both} $\omega_u^n$ and $\omega_v^n$ (because for example if the latter measures have $L^p$ density for some $p>1$, then  $u,v$ are bounded by constants depending only on the $L^p$-norm of the densities). An instance of our first main result, Theorem \ref{main1} below, is that \emph{an estimate similar to (\ref{blocki-stab-ine}) still holds if only the information on the regularity of $\omega_u^n$ is available.} Concretely a direct consequence of Theorem \ref{main1} (applied to $\tilde{\chi}(t)=t$) in the setting in consideration is that  there exist an explicit constant $0<\gamma <1$ such that if $\omega_u^n= g \omega^n$ for some $g \in L^p$, then there is a constant $C_{g,\omega,n}$ depending on $\omega$ (and $n$), and an upper bound of $\|g\|_{L^p}$ so that
\begin{align}\label{ine-uvomegacunglop}
\|u-v\|_{L^1} \le C_{g,\omega,n}\|\omega_u^n- \omega_v^n\|^\gamma,
\end{align} 
for every bounded $\omega$-psh function $v$ (notice $C_{g,\omega,n}$ is independent of $v$). The key ingredient of the proof of the last inequality is the following generalization of (\ref{ine-gradientchibangt}): for $u \le v$ and for every bounded or not $\cali{C}^1$ convex functions $\chi, \tilde{\chi}$ with $\chi(0)=\tilde{\chi}(0)=0$ and $\tilde{\chi} \le \chi$, there holds
\begin{align} \label{ine-chikahlerexam} 
\int_X \chi'(u-v) d(u-v) \wedge \dc (u-v) \wedge \omega^{n-1} \le C_{N} f_{ \chi,\tilde{\chi}} \bigg(\int_X -\chi (u-v)(\omega_v^n- \omega_u^n) \bigg)
\end{align}
where $C_N$ is an explicit constant depending on $\tilde{\chi}$-energies of $u,v$, and  $f_{\chi,\tilde{\chi}}$ is an explicit continuous function depending on $\chi, \tilde{\chi}$ and satisfying $f_{\chi,\tilde{\chi}}(0)=0$; see Proposition \ref{pro-mainstabilitylowenergyconvex} in Section 3 or the next subsection for details. In order to prove (\ref{ine-uvomegacunglop}), we apply (\ref{ine-chikahlerexam}) essentially to $\chi(t):= \max\{t,-1\}$, and $\tilde{\chi}(t)= \max\{t,-A\}$ for some suitable big constant $A>0$. 

 The proof of (\ref{ine-chikahlerexam}) reveals the core of our method and is not a direct generalization of the arguments in the proof of (\ref{ine-gradientchibangt}) in \cite{Blocki_stability} which treats the case where $\chi(t)= \tilde{\chi}(t)=t$. We will explain more in the next subsection \ref{subsec-key-component} about this point. With the above preparation, we are now in position to state our first main result in full generality. To this end, we need some more notations.

For every Borel set $E$ in $X$,  recall that the capacity of $E$ is given by 
 $$\capK(E)=\capK_\omega(E):= \sup_{\{w \in \PSH(X,\omega): 0 \le w  \le 1\}} \int_{E}\omega_w^n.$$
 We usually remove the subscript $\omega$ from $\capK_\omega$ if $\omega$ is clear from the context. There are generalizations of capacity in big cohomology classes, many of them are comparable; see  Theorem \ref{the comparisoncap} below and \cite{Lu-comparison-capacity}. Recall that a sequence of Borel functions $(u_j)_j$ is said to \emph{converge to  a  Borel function $u$ in capacity} if  for every constant $\epsilon>0$, we have that $\capK(\{|u_j-u| \ge \epsilon\})$  converges to $0$ as $j \to \infty$.  The convergence in capacity is of great importance in pluripotential theory in part because it implies the convergence of Monge-Amp\`ere operators under reasonable circumstances. To study quantitatively the convergence in capacity, it is convenient to introduce the following distance function on $\PSH(X,\omega)$: 
$$d_{\capK}(u, v):=\sup_{w \in \PSH(X,\omega): 0\le w \le 1}\int_X |u-v|^{1/2} \omega_w^n$$
for every $u, v \in \PSH(X,\omega)$ (note that $d_{\capK}(u, v)<\infty$ thanks to the Chern-Levine-Nirenberg inequality). The number $\frac{1}{2}$ in the definition of $d_\capK$ can be replaced by any constant in $(0,1)$.  One can see that for $u_j,u \in \PSH(X,\omega)$ for $j \in \N$,   $d_{\capK}(u_j, u) \to 0$ if and only if $|u_j-u| \to 0$ in capacity.
 

For $\theta$-psh functions $u, v$, we put
$$d_\theta(u,v):= 2\int_X\theta_{\max\{u, v\}}^n-\int_X \theta_u^n - \int_X \theta_v^n .$$
The function $d_\theta$ is comparable to $d_{\mathcal{S}(\theta)}$ (see Proposition \ref{pro-equivdistance}). For quantitative estimates, it is more convenient to use $d_\theta$ than $d_{\mathcal{S}(\theta)}$. It is perhaps worth noting that our method to prove the stability results below also implies that $d_\capK$ is bounded from above by a power of $d_\theta$ for model $\theta$-potentials (see Proposition \ref{cormodel} for details).    

Let $\mathcal{W}^-$ be the set of convex increasing functions $\chi: \R_{\le 0} \to\R_{\le 0}$ so that $\chi(0)=0$ and $\chi(-\infty)= -\infty$. It follows from \cite[Proposition 3.2]{BEGZ} that for every non-positive $\theta$-psh function $u$, there exists
 $\chi\in\mathcal{W}^-$ and $C>0$ such that
  $$-\int_X\psi\,\theta_u^n\leq C,$$
  for every $\psi\in\PSH(X, \omega)$ with $\sup_X\psi=0$.

\begin{theorem}\label{main1}
	Let $\theta$ be a closed smooth real $(1,1)$-form such that $\theta\leq A\omega$ for a given constant $A\geq 1$. Let $u\in\PSH(X, \theta)$ such that  $\sup_Xu=0$
	and $\int_X\theta_{u}^n:=\delta>0$. 
	Let	$B\geq A$ and $\tilde{\chi} \in \mathcal{W}^-$ with $\tilde{\chi}(-1)=-1$ such that
	$$\int_X-\tilde{\chi}(\psi)\theta_u^n\leq B\delta,$$
	for every $\psi\in\PSH(X, (A+1)\omega)$ with $\sup_X\psi=0$. Let $h(s):=(-\tilde{\chi}(-s))^{1/2}$ for $s \le 0$. 
	Then, for every  constant $0<\gamma<1$,
	there exists a constant $C>0$ depending only on $n, X, \omega$ and $\gamma$  such that 
	\begin{equation}\label{eq0main1}
		d_{\capK}(u, v)^2\leq C (A\, B)^2 
		\left(h^{\circ(n)}\left(\dfrac{\delta}{\|\theta_{u}^n-\eta_{v}^n\|+A^n\|\theta-\eta\|_{\Cc^0}
			+d_{(A+1)\omega}(u, v)}\right)\right)^{-\gamma},
	\end{equation}
	for every closed smooth real $(1, 1)$-form $\eta\leq A\omega$ and for each $v\in\PSH(X, \eta)$ with $\sup_Xv=0$.
\end{theorem}

Here, we denote by $\|\mu-\mu'\|$ the mass norm of $\mu- \mu'$. 
The condition that $\tilde{\chi}(-1)=-1$ is merely a normalization one. For an arbitrary $\tilde{\chi} \in \mathcal{W}^-$, we can consider $\tilde{\chi}/|\tilde{\chi}(-1)|$ which satisfies the last requirement.   Theorem \ref{main1} says that under a very weak assumption on the regularity of the Monge-Amp\`ere of $u$, one can bound from above the distance $d_\capK$ of $u$ with \emph{any other quasi-psh function} $v$. 

We now turn our attention to the class of H\"older continuous measures whose definition is recalled below. 
Let $\PSH_0(X,\omega)$ be the set of $\omega$-psh functions $u$ with $\int_X u \omega^n =0$. We endow $\PSH_0(X,\omega)$ with the $L^1(\omega^n)$ distance.  Let $\mu$ be a measure on $X$ such that quasi-psh functions are $\mu$-integrable. We say that \emph{$\mu$ is H\"older continuous with H\"older constant $A$ and H\"older exponent $\gamma$ if it is so as a functional on $\PSH_0(X,\omega)$,} in other words, for every $u_1, u_2 \in \PSH_0(X,\omega)$, we have
\begin{align}\label{ine-holdermeasuredinhnghia}
	\int_X |u_1- u_2| d \mu \le A \|u_1- u_2\|_{L^1(\omega^n)}^{\gamma}.
\end{align}
This notion was introduced in \cite{DinhVietanhMongeampere}. By expressing every $\omega$-psh function $u$ as $u= u- \int_X u \omega^n + \int_X u \omega^n$, we deduce from (\ref{ine-holdermeasuredinhnghia})   that
\begin{align}\label{ine-holdermeasuredinhnghia2}
	\int_X |u_1- u_2| d \mu \le \big(A+\mu(X) \big)\max\{\|u_1- u_2\|_{L^1(\omega^n)}^{\gamma}, \|u_1- u_2\|_{L^1(\omega^n)}\}
\end{align}
for every $\omega$-psh function $u_1, u_2$. Clearly the last inequality also implies that $\mu$ is H\"older with H\"older exponent $\gamma$ and with H\"older constant $\lambda\big(A+\mu(X)\big)$, for some constant $\lambda$ depending only on $(X,\omega)$.   Recall that a measure is  H\"older continuous if and only if it can be written as $(\ddc u+\omega)^n$ for some H\"older continuous $\omega$-psh function $u$ on $X$; see  \cite{DemaillyHiep_etal,DinhVietanhMongeampere} and also \cite{Kolodziej08holder}. We refer to these papers and \cite{Hiep_holder,Kolodziej-Nguyen-continuous,NgocCuong-Holder2020,Vu_MA} for examples of H\"older continuous measures. Most basic examples are measures with $L^p$ density or smooth volume forms of (immersed) generic (real) Cauchy-Riemann submanifolds on $X$.  

Recall that the set of Radon measures on $X$ endowed with the weak topology is a metric space with the distance $\dist_{-\delta}$ for $\delta \in [0,\infty)$ defined as follows: for measures  $\mu,\mu',$
\begin{align}\label{def-KWdistance}
	\dist_{-\delta}(\mu,\mu'):= \sup_{\|v\|_{\Cc^{\delta}} \le 1} \big| \langle \mu-\mu', v \rangle \big|,
\end{align}
where $v$ is a smooth real-valued function on $X$ (see \cite[Theorem 6.9]{Villani}). Note that $\dist_{-\delta}$ induces the same weak topology when $\delta>0$. When $\delta=0$, it is the mass norm of $\mu_1- \mu_2$. We also have the following interpolation inequality: for $0 \le \beta_0 < \beta_1 < \beta_2$,
\begin{align} \label{ine-distdelta}
	\dist_{-\beta_1} \le \dist_{-\beta_0}^{\frac{\beta_2- \beta_1}{\beta_2- \beta_0}} \dist_{-\beta_2}^{\frac{\beta_1- \beta_0}{ \beta_2- \beta_0}}.
\end{align}
We refer to \cite{Lunardi-book-interpolation,Triebel} for a proof (see also \cite{Vu_MA}). This kind of estimate is very important in complex dynamics since the appearance of \cite{DS_acta} where a more general version of  (\ref{ine-distdelta}) for currents  was introduced.


Our next main result is as follows:
\begin{theorem}\label{main3}
	Let $\theta_1, \theta_2$ be closed smooth real $(1,1)$-forms and $A$ be positive constant at least $1$ such that $\theta_j \le A \omega$ for $j=1,2$. Let $0<\delta \le 1$ and $M \ge 1$ be   constants and
	$u_j\in\PSH(X, \theta_j)$ ($j=1, 2$) such that  
	$$\sup_Xu_j=0, \quad \int_X\theta_{u_j}^n\geq \delta,$$
	and  $\mu_j:= (\theta_j+dd^cu_j)^n$ ($j=1, 2$) are H\"older continuous measures on $X$ with H\"older exponent $\beta$ and with H\"older constant $M\delta$.
	Then, for every $0<\gamma<1$,
	there exists a constant $C>0$ depending only on $n, X, \omega, A, M, \beta$ and $\gamma$  such that 
	$$\big(d_{\capK}(u_1, u_2)\big)^2\leq C
	\left(\dfrac{\tau^{\gamma\beta/(\beta+1)}+\|\theta_1-\theta_2\|_{\Cc^0}
		+d_{(A+1)\omega}(u_1, u_2)}{\delta}\right)^{2^{-n}\gamma},$$
	where  $\tau:=\dist_{-1}(\mu_1,\mu_2)$.
\end{theorem}

By interpolation inequality  (\ref{ine-distdelta}), an analogous inequality also holds for $\dist_{-\beta}$ in place of  $\dist_{-1}$ for any constant $\beta>0$.  Our last main result is a generalization of Cegrell-Ko{\l}odziej-Xing stability theorem (\cite{Cegrell-Kolodziej,Xing-stability}) which treated the case where $\theta = \omega$ and $\phi=0$ (and only for the class of potentials of full Monge-Amp\`ere mass). We also underline that the original result in  \cite{Cegrell-Kolodziej,Xing-stability} is non-quantitative and Theorem \ref{main2} already strengthens their results in their setting.  

\begin{theorem}\label{main2}
	Let $\theta_1, \theta_2$ be closed smooth real $(1,1)$-forms and $A$ be positive constant at least $1$ such that $\theta_j \le A \omega$ for $j=1,2$. Let $0<\delta \le 1$ and
	$u_j\in\PSH(X, \theta_j)$ ($j=1, 2$) such that  $\sup_Xu_j=0$
	and $\int_X\theta_{u_j}^n\geq \delta$. Assume that there exists a
Radon measure $\mu$ on $X$ such that $\mu$ vanishes on pluripolar sets and $(\theta_j+dd^cu_j)^n\leq\mu$ for
$j=1, 2$.
	Then, there exists a continuous increasing function $f_{\mu}:\R_{\ge 0}\rightarrow\R_{\ge 0}$ depending only on $n, X, \omega, A, \delta$ and $\mu$  such that $f(0)=0$ and
	$$d_{\capK}(u_1, u_2)^2\leq f_{\mu}\left(\dist_{-1}(\mu_1, \mu_2)+\|\theta_1-\theta_2\|_{\Cc^0}
		+d_{(A+1)\omega}(u_1, u_2)\right),$$
		where $\mu_j:=(\theta_j+dd^cu_j)^n$ for $j=1, 2$.
\end{theorem}

Theorem \ref{main2} implies particularly that for every model $\theta$-psh function $\phi$,  the convergence in capacity or in $L^1$ and the weak convergence of Monge-Amp\`ere measures are equivalent in the class of potentials in $\mathcal{E}(X, \theta, \phi)$ whose Monge-Amp\`ere measures are bounded from above by a fixed non-pluripolar measure. This is more or less the original motivation of Cegrell-Ko{\l}odziej in \cite{Cegrell-Kolodziej}. 

Finally we note that as an application of Theorem \ref{main2} or \ref{main3}, one can recover  a main result in \cite{Darvas-Lu-DiNezza-singularity-metric} that the pseudometric space of singularity types of volume bounded from below by a fixed positive constant is complete, we refer to Remark \ref{re-suyratinhday} in the end of the paper and Subsection \ref{subsec-appli-singu} for details.

\subsection{Key components in our method} \label{subsec-key-component}

As mentioned above the core of the method developed in this  paper is a solution to the quantitative stability for measures of lower energy in the setting where the cohomology class and the prescribed singularity are fixed. We underline  that in what follows  by convex weights we mean also bounded convex functions, although such functions were not usually considered as weights.  This point of view is the key allowing us to treat the general setting when both the cohomology class and the prescribed singularity of solution vary. 

Let $\widetilde{\mathcal{W}}^-$ be the set of convex, non-decreasing functions $\chi: \R_{\le 0} \to \R_{ \le 0 }$ such that $\chi(0)=0$ and $\chi \not = 0$. Note that $\chi$ can be bounded. 
Obviously $\mathcal{W}^-$ is contained in $\widetilde{\mathcal{W}}^-$. It is crucial in our method that we consider also $\chi \in \widetilde{\mathcal{W}}^-$ which is bounded.       Let $M \ge 1$ be a constant and  $\mathcal{W}^+_M$ the usual space of  increasing concave functions $\chi: \R_{\le 0} \to \R_{ \le 0}$ such that $\chi(0)=0$, $\chi \not \equiv 0$, and $|t \chi'(t)| \le M|\chi(t)|$ for every $t \le 0$.

Let $\varrho:= \int_X \theta_\phi^n$. For  $\chi \in \widetilde{\mathcal{W}}^- \cup  \mathcal{W}^+_M$ and   $u \in \PSH(X,\theta,\phi)$, let 
 $$E^0_{\chi, \theta, \phi}(u):= - \varrho^{-1} \int_X \chi(u- \phi) \theta_u^n$$
 which is called \emph{the (normalized) $\chi$-energy} of $u$ (with respect to $\theta, \phi$).   We denote
$$\mathcal{E}_\chi(X, \theta,\phi):=\big\{u \in \mathcal{E} (X, \theta, \phi): E_{\chi, \theta, \phi}(u)<\infty  \big\},$$ where $\mathcal{E}(X, \theta,\phi)$ is the space of $\theta$-psh functions $u \le \phi$ with $\int_X \theta_u^n = \int_X \theta_\phi^n$. Certainly if $\chi$ is bounded, then $\mathcal{E}_\chi(X, \theta,\phi)= \mathcal{E}(X, \theta,\phi)$. We would like to point out however that our method is not about the finiteness of $E^0_{\chi, \theta, \phi}(u)$ but estimating the size of that quantity. Thus whether $\chi$ is bounded or not does not make much difference for our later arguments.  
  Put
$$\quad I^0_\chi(u,v):=  \varrho^{-1}\int_{\{u<v\}} \chi(u-v) (\theta_v^n - \theta_u^n)+\varrho^{-1}\int_{\{u>v\}} \chi(v-u) (\theta_u^n - \theta_v^n)$$
 for $u,v \in \mathcal{E}_{\chi}(X, \theta, \phi)$. The factor $\varrho^{-1}$ in the defining formulae for $E^0_{\chi,\theta, \phi}(u)$ and  $I^0_\chi(u,v)$ plays the role of a normalizing constant. In geometric applications it is important to treat the case where $\varrho \to 0$, \emph{i.e,} to obtain estimates uniformly  as $\varrho \to 0$ (here we allow $\theta$ or its cohomology class to vary).   
 
Clearly if $\theta_u^n= \theta_v^n$, then $I^0_\chi(u,v)=0$. We will see later that each term in the sum defining $I^0_\chi(u,v)$ is nonnegative. 
We recall that there is a natural (quasi-)metric on the space $\mathcal{E}_\chi(X,\theta, \phi)$ constructed in \cite{Darvas_book,Darvas-lower-energy, Gupta}, and see \cite{DDL-L1metric,Lu-DiNezza-Lpmetric,Trusiani-energy,Xia-energy} as well.  The functional $I^0_\chi(u,v)$ has an intimate relation with these quasi-metrics. We refer to the end of Section \ref{sec-fixedtype} for details on this connection.  Here is the first key ingredient in our proof of main results.

\begin{theorem} \label{th-lowerenergy-phanintro}  Let $\theta$ be a closed smooth real $(1,1)$-form  and $\phi$ be a  negative $\theta$-psh function such that $\varrho:= \int_X \theta_\phi^n>0$.  Let $\chi, \tilde{\chi} \in \widetilde{\mathcal{W}}^-\cup\mathcal{W}^+_M$ ($M\geq 1$) such that $\tilde{\chi} \le \chi$. Let $B \ge 1$ be a constant and let $u_j, \psi_j \in \mathcal{E}(X, \theta, \phi)$ satisfy $u_1 \le u_2$ and 
$$E^0_{\tilde{\chi}, \theta, \phi}(u_j)+E^0_{\tilde{\chi}, \theta, \phi}(\psi_j) \le B,$$
for $j=1,2$. Then there exists a constant $C_n>0$ depending only on $n$ and $M$,
 and a continuous increasing function $f: \R_{\ge 0} \to \R_{ \ge 0}$ depending only on $\chi, \tilde{\chi}$ such that $f(0)=0$ and 
\begin{align*}
\int_X  -\chi (u_1-u_2)  (\theta_{\psi_1}^n- \theta_{\psi_2}^n) \le 
C_n\varrho B^2 f\big(I^0_\chi(u_1,u_2)\big).
\end{align*}
\end{theorem}

The following result is the second key which is a consequence of the first one. 

\begin{theorem} \label{th-lowerenergy-kocochuanhoa-phanintro}
	 Let $\theta$ be a closed smooth real $(1,1)$-form, and let $A \ge 1$ be a constant such that $\theta \le A \omega$.   Let $\phi$ be a  model $\theta$-psh function.  Let 
	 $\chi, \tilde{\chi} \in \widetilde{\mathcal{W}}^-\cup\mathcal{W}^+_M$ ($M\geq 1$)
	 such that $\tilde{\chi} \le \chi$. Let $B \ge 1$ be a constant and $u_1, u_2, \psi \in \mathcal{E}(X, \theta, \phi)$  with $\sup_X u_1= \sup_X u_2$ satisfy
	$$  
	E^0_{\tilde{\chi}, \theta, \phi}(u_1)+
	E^0_{\tilde{\chi}, \theta, \phi}(u_2)+E^0_{\tilde{\chi}, \theta, \phi}(\psi) \le B.$$
Then, for every constant $m>0$ and $0<\gamma<1$, there exist  a constant  $C>0$ depending on $n, M, X, \omega, m$ and $\gamma$, and  a function $f$ as in Theorem \ref{th-lowerenergy-phanintro}  such that 
\begin{align*}
\int_X  -\chi\big(-|u_1- u_2|\big) \theta_\psi^n \le -\varrho \chi\left(-\lambda^m\right)
+C\varrho B_{\gamma, m}^2 \lambda^\gamma,
\end{align*}
where $\lambda:= f \big(I^0_\chi(u_1,u_2)\big)$ and 
$B_{\gamma, m}=A^{(1-\gamma)/(2m)}(B-\tilde{\chi}(-A))(1-\tilde{\chi}(-1))$.
\end{theorem}

The condition $\sup_X u_1= \sup_X u_2$ is simply a normalization one. By Theorem \ref{th-lowerenergy-kocochuanhoa-phanintro}, one sees in particular that if $I^0_\chi(u_1,u_2) \to 0$, then $|u_1-u_2| \to 0$ in $L^p$ for every $p>0$. 
The function $f$ can be made explicitly; see Theorems \ref{th-lowerenergy} and \ref{th-lowerenergy-kocochuanhoa} below for more elaborated versions of these above results. 

We note that \emph{the single theorem \ref{th-lowerenergy-phanintro} contains the following three important results in pluripotential theory: uniqueness of solutions of complex Monge-Amp\`ere equations, domination principle, and comparison of capacities}. We obtain indeed quantitative (hence stronger) versions of these results for which we refer to Subsection  \ref{subsec-appli}. Readers also find there a quantitative version of the fact that the convergence in Darvas's metric in $\mathcal{E}_\chi(X, \theta, \phi)$ implies the convergence in capacity. Notice that such an estimate seems to be not reachable by using the usual plurisubharmonic envelope method.

The main novelty of Theorem \ref{th-lowerenergy-phanintro} is that it deals with \emph{arbitrary} weights. Similar statements was already known for $\chi(t)=t$ (see \cite{Blocki_stability,Guedj-Zeriahi-big-stability}). However the proof there only work \emph{exclusively} for this case. One should notice that the weight $\chi(t)=t$ is very special: it is linear and lies in the middle between higher energy weights and lower energy weights. As to the proof of Theorem \ref{th-lowerenergy-phanintro},  going up to the space of higher energy weights or going down to the space of lower energy weights are equally difficult.   We will explain this point in more details in the paragraph after Theorem \ref{the-mainstabilitylowenergyconvexintro} below.

The key in the proof of Theorem \ref{th-lowerenergy-phanintro} is Proposition \ref{pro-mainstabilitylowenergyconvex} in Section 3 a simplified version of which we state here for readers' convenience.

\begin{theorem} \label{the-mainstabilitylowenergyconvexintro} Let $\chi, \tilde{\chi} \in\widetilde{\mathcal{W}}^-\cup\mathcal{W}^{+}_M$	such that $\tilde{\chi}\leq\chi$ and $\chi\in \Cc^1(\R)$.  Let $u_1, u_2, u_3 \in \mathcal{E}(X, \theta, \phi)$ such that $u_1 \le u_2$ and
	$u_j-\phi$ is bounded ($j=1, 2, 3$), where $\phi$ is a negative $\theta$-psh function satisfying 
	$\varrho:=\vol(\theta_\phi)>0$.  Then there exist a constant $C_n>0$ depending only on $n$ and $M$,
	 and a function $f$ as in Theorem \ref{th-lowerenergy-phanintro}  such that 
	\begin{align*} 
		\int_X  \chi'(u_1-u_2) d(u_1- u_2) \wedge \dc (u_1- u_2) \wedge \theta_{u_3}^{n-1}  \le C_n\varrho B^2 f\big(I^0_\chi(u_1,u_2)\big),		
	\end{align*}
	where $B:=\sum_{j=1}^3 \max\{E^0_{\tilde{\chi},\theta, \phi}(u_j),1 \}$.
\end{theorem}

As far as we know, all of previous works related to Theorem \ref{the-mainstabilitylowenergyconvexintro} only concern with $\chi(t)=t$.  In this case, Theorem \ref{the-mainstabilitylowenergyconvexintro} was known with an explicit $f$ and without $\tilde{\chi}$ if   $\phi$ is of minimal singularity in the cohomology class of $\theta$, by \cite{Blocki_stability,Guedj-Zeriahi-big-stability}. 

The key ingredients in previous versions of  Theorem \ref{the-mainstabilitylowenergyconvexintro} for $\chi(t)=t$ are integration by parts arguments. Direct generalization of such reasoning immediately break down if $\chi  \not= \id$:  in  a more precise but technical level,  the integration by parts arguments give terms like $\chi'(u_1- u_2)d (u_1-u_3) \wedge \dc (u_1- u_3)$, such quantity is easy to bound if $\chi = \id$ (hence $\chi' \equiv 1$), but it is no longer the case if $\chi \not = \id$.  

In order to prove Theorem \ref{the-mainstabilitylowenergyconvexintro}, we still use this strategy but need to use a so-called ``monotonicity argument" from \cite{Vu_Do-MA,Viet-generalized-nonpluri,Viet-convexity-weightedclass} to deal with  general $\chi$. In a nutshell it is about using intensively the pluri-locality of Monge-Amp\`ere operators together with the monotonicity of pluricomplex energy which allow one to bound from above ``Monge-Amp\`ere quantities'' of bad potentials by that of nicer potentials.  This method is a flexible tool to deal with ``low regularity'', and was a key in the proof of the convexity of the class of potentials of finite $\chi$-energy  in \cite{Viet-convexity-weightedclass}, as well as, giving a characterization of the class of Monge-Amp\`ere measures with potentials of finite $\chi$-energy in \cite{Vu_Do-MA}. Moreover in order to deduce Theorem  \ref{th-lowerenergy-kocochuanhoa-phanintro} from Theorem \ref{th-lowerenergy-phanintro}, we use, among other things, an idea from \cite{Guedj-Zeriahi-big-stability} together with a very simple but crucial \emph{lower bound} of the sublevel sets of $\omega$-psh functions; see Lemma \ref{lem vol estimate} below. Such an estimate is of independent interest.   
 \\

\noindent
\textbf{Organization of the paper.} In Section \ref{sec-preli}, we recall the crucial integration by parts formula from \cite{Viet-convexity-weightedclass}, auxiliary facts about weights are also collected there. Theorems \ref{th-lowerenergy-phanintro}, \ref{th-lowerenergy-kocochuanhoa-phanintro}, and \ref{the-mainstabilitylowenergyconvexintro} are proved  in Section \ref{sec-fixedtype}.  We prove Theorems \ref{the-hoitucapacitymassnorm},  \ref{main1}, \ref{main3} and \ref{main2} in Subsection \ref{subsec-stability-estimates-varied}. Proposition \ref{cor-DDLcomplete} is proved at the end of the paper.\\

\noindent
\textbf{Acknowledgments.} We would like to thank Tam\'as Darvas, Vincent Guedj, Henri Guenancia,  Prakhar Gupta,  Hoang Chinh Lu, Tat Dat T\^o, Valentino Tosatti, and Ahmed Zeriahi for fruitful discussions. The research of D.-V. Vu is partly funded by the Deutsche Forschungsgemeinschaft (DFG, German Research Foundation)-Projektnummer 500055552. 
  \\

\section{Preliminaries} \label{sec-preli}

\subsection{Integration by parts} \label{subsec-intebyparts}

In this subsection, we recall the integration by parts formula obtained in \cite[Theorem 2.6]{Viet-convexity-weightedclass}. This formula will play a key role in our proof of main results later.

Let $X$ be a compact K\"ahler manifold.  Let $T_1, \ldots, T_m$ be closed positive $(1,1)$-currents on $X$. Let $T$ be  a closed positive current of bi-degree $(p,p)$ on $X$. The \emph{$T$-relative non-pluripolar product} $\langle \wedge_{j=1}^m T_j \dot{\wedge} T\rangle$ is defined  in a way similar to that of  the usual non-pluripolar product (see \cite{Viet-generalized-nonpluri}). The product $ \langle  \wedge_{j=1}^m T_j \dot{\wedge} T\rangle $ is a closed positive current of bi-degree $(m+p,m+p)$; and the wedge product $ \langle  \wedge_{j=1}^m T_j \dot{\wedge} T\rangle $ as an operator on currents  is  symmetric with respect to $T_1, \ldots, T_m$ and is homogeneous. In latter applications, we will only use the case where $T$ is the non-pluripolar product of some closed positive $(1,1)$-currents, say, $T= \langle T_{m+1} \wedge \cdots \wedge T_{m+l} \rangle$, where $T_j$ is $(1,1)$-currents for $m+1 \le j \le m+l$. In this case, $\langle T_1 \wedge \cdots \wedge T_m \dot{\wedge}T \rangle $ is simply equal to $\langle \wedge_{j=1}^{m+l} T_j \rangle$. We usually remove the bracket $\langle \quad \rangle$ in the non-pluripolar product to ease the notation. 

Recall that a \emph{dsh} function on $X$ is the difference of two quasi-plurisubharmonic (quasi-psh for short) functions on $X$ (see \cite{DS_tm}). These functions are well-defined outside pluripolar sets. Let $v$ be a dsh function on $X$.  Let $T$ be a closed positive current on $X$. We say that $v$ is \emph{$T$-admissible} if  there exist  quasi-psh functions $\varphi_1, \varphi_2$ such that $v= \varphi_1- \varphi_2$  and $T$ has no mass on $\{\varphi_j=-\infty\}$ 
for $j=1,2$. In particular, if $T$ has no mass on pluripolar sets, then every dsh function is $T$-admissible.  

Assume now that $v$ is $T$-admissible.    Let $\varphi_{1}, \varphi_{2}$ be quasi-psh functions such that $v= \varphi_{1}- \varphi_{2}$ and $T$ has no mass on $\{\varphi_{j}=-\infty\}$ for $j=1,2$. Let 
$$\varphi_{j,k}:= \max\{\varphi_{j}, -k \}$$
for every $j=1,2$ and $k \in \N$. Put $v_k:= \varphi_{1,k}- \varphi_{2,k}$. Put
$$Q_k:= d v_k \wedge \dc v_k \wedge T=\ddc v_k^2  \wedge T - v_k\ddc v_k \wedge T.$$
By the plurifine locality with respect to $T$ (\cite[Theorem 2.9]{Viet-generalized-nonpluri}) applied to the right-hand side of the last equality, we have 
\begin{align}\label{eq-localplurifineddc}
\bold{1}_{\cap_{j=1}^2 \{\varphi_{j}> -k\}} Q_k =\bold{1}_{\cap_{j=1}^2\{\varphi_{j}> -k\}} Q_{s}
\end{align}
for every $s\ge k$. We say that $\langle d v \wedge \dc v \dot{\wedge} T \rangle$ is  \emph{well-defined} if the mass of $\bold{1}_{\cap_{j=1}^2 \{\varphi_{j}> -k\}} Q_k$ is uniformly bounded on $k$. In this case, using  (\ref{eq-localplurifineddc}) implies that there exists a positive current $Q$ on $X$ such that for every bounded Borel form $\Phi$ with compact support on $X$ such that 
$$\langle Q,\Phi \rangle  = \lim_{k\to \infty} \langle \bold{1}_{\cap_{j=1}^2 \{\varphi_{j}> -k\}} Q_k, \Phi\rangle,$$
and we define $\langle d v \wedge \dc v \dot{\wedge} T \rangle$ to be the current $Q$.  This agrees with the classical definition if $v$ is the difference of two  bounded quasi-psh functions. One can check  that this definition is independent of the choice of $\varphi_1, \varphi_2$.   By \cite[Lemma 2.5]{Viet-convexity-weightedclass}, if $v$ is bounded, then  $\langle d v \wedge \dc v \dot{\wedge} T \rangle$ is well-defined.

Let $w$ be another $T$-admissible dsh function.  If $T$ is of bi-degree $(n-1,n-1)$, we can also define the current $\langle  dv \wedge \dc w \dot{\wedge} T\rangle$ by a similar procedure as above. More precisely, we say $\langle dv \wedge \dc w \dot{\wedge} T \rangle$ is \emph{well-defined} if  $\langle dv \wedge \dc v \dot{\wedge} T \rangle$, $\langle dw \wedge \dc w \dot{\wedge} T \rangle$, and $\langle d(v+w) \wedge \dc (v+w) \dot{\wedge} T \rangle$ are well-defined. In this case, as in the classical case of bounded potentials, the defining formula for $\langle dv \wedge \dc w \dot{\wedge} T \rangle$ is obvious: 
$$2 \langle dv \wedge \dc w \dot{\wedge} T \rangle= \langle d(v+w) \wedge \dc (v+w) \dot{\wedge} T \rangle- \langle d v \wedge \dc v \dot{\wedge} T \rangle- \langle d w \wedge \dc w \dot{\wedge} T \rangle.$$
As above, if $v,w$ are bounded $T$-admissible, then $\langle dv \wedge \dc w \dot{\wedge} T \rangle$ is well-defined and given by the above formula. The following Cauchy-Schwarz inequality is clear from definition. 

\begin{lemma}\label{le-CSine} Assume that $\langle dv \wedge \dc w \dot{\wedge} T \rangle$ is well-defined. Then for every positive Borel function $\chi$,  we have
$$\int_X \chi \langle dv \wedge \dc w \dot{\wedge} T \rangle  \le \bigg(\int_X \chi \langle dv \wedge \dc v \dot{\wedge} T \rangle \bigg)^{1/2}\bigg(\int_X \chi \langle dw \wedge \dc w \dot{\wedge} T \rangle \bigg)^{1/2}.$$
\end{lemma}
We put
$$\langle \ddc v \dot{\wedge} T \rangle:= \langle \ddc \varphi_1 \dot{\wedge} T\rangle- \langle \ddc \varphi_2 \dot{\wedge} T\rangle$$
which is independent of the choice of $\varphi_1,\varphi_2$.  The following integration by parts formula is crucial for us later.

\begin{theorem} \label{th-integrabypart}  (\cite[Theorem 2.6]{Viet-convexity-weightedclass} or \cite[Theorem 3.1]{Vu_Do-MA}) Let  $T$ be a closed positive current of bi-degree $(n-1,n-1)$ on $X$. Let $v,w$ be bounded $T$-admissible dsh functions on $X$. If $\chi: \R \to \R$ is a $\cali{C}^3$ function then we have  
\begin{multline}\label{eq-intebypartschi}
\int_X \chi(w) \langle  \ddc v \dot{\wedge} T\rangle=\int_X v \chi''(w) \langle dw \wedge \dc w \dot{\wedge} T\rangle+\int_X v \chi'(w) \langle \ddc w \dot{\wedge} T\rangle\\
= - \int_X \chi'(w) \langle d w \wedge  \dc v \dot{\wedge} T\rangle.
\end{multline}
\end{theorem}

Since the case where $T$ is a non-pluripolar product of $(1,1)$-currents plays an important role in the study of the complex Monge-Amp\`ere equation, we present below an equivalent natural way to define the current $\langle d \varphi \wedge \dc \varphi \dot{\wedge} T \rangle$ in this setting. It is just for the purpose of clarification. 

\begin{lemma}\label{le-bangnhauTnonpluripolar} Let $u_1,\ldots, u_m$ be  negative psh functions on an open subset $U$ in $\C^n$ such that $T:=\langle \ddc u_1 \wedge \cdots \wedge \ddc u_m \rangle$ is well-defined. Let $v$ be the difference of two bounded psh functions on $U$.  For $k \in \N$, put $u_{j,k}:= \max\{u_j, -k\}$ and 
$$T_k:= \ddc u_{1,k} \wedge \cdots \wedge \ddc u_{m,k}.$$
 Then we have 
$$d v \wedge \dc v \wedge T= d v \wedge \dc v \wedge T_k$$
on $\cap_{j=1}^m \{u_j >-k\}$. 
\end{lemma}

\proof Put 
$$\psi_k:= k^{-1}\max \{u_1+ \cdots+u_m, -k\}+1.$$
Observe $\psi_k T_k= \psi_k T$. Now regularizing $v$ and using the continuity of Monge-Amp\`ere operators of bounded potentials, we obtain
$$\psi_k d v \wedge \dc v \wedge T= \psi_k d v \wedge \dc v \wedge T_k.$$
Hence 
\begin{align*}
d v \wedge \dc v \wedge T =  d v \wedge \dc v \wedge T_k
\end{align*}
on $U:=\cap_{j=1}^m \{u_j >-k/(2m)\}$ (for $\psi_k \ge 1/2$ on $U$). Note that $ d v \wedge \dc v \wedge T_k=  d v \wedge \dc v \wedge T_{k/(2m)}$ on $U$ by the plurifine locality. Thus the desired assertion follows.    This finishes the proof. 
\endproof

Let $T_1,\ldots, T_m$ be closed positive $(1,1)$-currents on $X$.  Let $n:= \dim X$.  \emph{Consider now} 
$$T:= \langle T_1 \wedge \cdots \wedge T_m \rangle.$$
Note that $T$ has no mass on pluripolar sets. Let $\varphi_1,\varphi_2$ be negative quasi-psh function on $X$.  Let $\varphi_{j,k}$ ($j=1,2$) be as before and $v:=\varphi_1- \varphi_2$.  In the moment, we work locally. Let $U$ be an open  small enough local chart (biholomorphic to a polydisk in $\C^n$) in $X$ such that  $T_j= \ddc u_j$ for $j=1,\ldots, m$, where $u_j$ is negative psh functions on $U$. Put $u_{j,k}:= \max\{u_j, -k\}$ for $k \in \N$, and 
$$T_k:= \ddc u_{1,k} \wedge \cdots \wedge \ddc u_{m,k}, \quad Q'_k:= d v_k \wedge \dc v_k \wedge T_k.$$
Put $A_k:= \cap_{j=1}^2 \{\varphi_{j}> -k\} \cap \cap_{j=1}^m \{u_j >-k\}$. By plurifine properties of Monge-Amp\`ere operators, we have 
$$\bold{1}_{A_k}Q'_k= \bold{1}_{A_{k}} Q'_{s}$$
for every $s \ge k$. One can check that the condition that  $(\bold{1}_{A_k} Q'_k)_k$ is of mass bounded uniformly (on compact subsets in $U$) in $k$ is independent of the choice of potentials. 

\begin{proposition}\label{pro-haidinhnghiaddctuongduong}  The current $\bold{1}_{A_k} Q'_k$ is of mass bounded uniformly in $k$ on compact subsets in $U$ for every $U$ (small enough biholomorphic to a polydisk in $\C^n$) if and only if the current $\langle d v \wedge \dc v \dot{\wedge} T \rangle$ is well-defined. In this case we have 
\begin{align}\label{eq-haidinhnghiatuongduong}
\langle d v \wedge \dc v \dot{\wedge} T \rangle= \lim_{k \to \infty} \bold{1}_{A_k} Q'_k.
\end{align}
\end{proposition}

\proof   By writing a smooth form of bi-degree $(n-m-1,n-m-1)$ as the difference of two smooth positive forms, we can assume without loss of generality that $T$ is of bi-degree $(n-1,n-1)$ (hence $m=n-1$).  Assume that $\langle d v \wedge \dc v \dot{\wedge} T \rangle$ is well-defined. We will check that $\bold{1}_{A_k} Q'_k$ is of mass bounded uniformly in $k$ on compact subsets in $U$. Let $\chi$ be a nonnegative smooth function compactly supported on $U$. Put
 $$\psi:= \varphi_1+\varphi_2+ u_1 + \cdots + u_m, \quad \psi_k:= k^{-1}\max\{\psi, -k\}+1.$$
 and $\varphi_{jk}:= \max\{\varphi_j, -k\}$ for $1 \le j \le 2$. Observe that $0 \le \psi_k \le 1$ and  if $\psi_k >0$, then $\varphi_j>-k$ for $1 \le j \le 2$; and 
\begin{align}\label{limit-psiknon}
\psi_k(x) \ge 1- s/k
\end{align}
for  every $x \in A_{s/(m+2)}$ and $1 \le s \le k$.   Recall $v_k:= \varphi_{1k}- \varphi_{2k}$ which is bounded (but not necessarily uniformly in $k$). 
 Observe that $\langle d v \wedge \dc v \dot{\wedge} T \rangle$ has no mass on pluripolar sets because $T$ is so (see for example \cite[Lemma 2.1]{Viet-generalized-nonpluri}). Put $Q''_k:= \psi_k Q_k= \psi_k \bold{1}_{A_k} Q'_k$ By (\ref{limit-psiknon}) and Lemma \ref{le-bangnhauTnonpluripolar}, we have
\begin{align}\label{ine-dvdcvTlimitk} 
\langle d v \wedge \dc v \dot{\wedge} T \rangle &= \lim_{k\to \infty} \psi_k d v_k \wedge \dc v_k \wedge T\\
\nonumber
&= \lim_{k\to \infty} \psi_k d v_k \wedge \dc v_k \wedge T_k =\lim_{k\to \infty}Q''_k
\end{align}
on $U$. On the other hand, by (\ref{limit-psiknon}) again,  we see that the claim that $Q''_k$ is of mass uniformly bounded on compact subsets in $U$ is equivalent to that $\bold{1}_{A_k}Q'_k$ is so. This together with (\ref{ine-dvdcvTlimitk}) yields the desired assertion. 

Conversely, suppose now that  $\bold{1}_{A_k} Q'_k$ is of mass bounded uniformly in $k$ on compact subsets in $U$ for every $U$. Thus there exists a positive current $R$ on $U$ such that $\bold{1}_{A_k} R= \bold{1}_{A_k} Q'_k$ for every $k$ and $U$. Set 
$$\tilde{\psi} := \varphi_1+ \varphi_2, \quad \tilde{\psi}_k:= k^{-1} \max\{\tilde{\psi}, -k\}+1.$$
Let $s \in \N$ with $s \ge k$. Observe
\begin{align*}
\psi_{s} R=\tilde{\psi}_k  \psi_{s} R+ (1- \tilde{\psi}_k) \psi_{s} R.  
\end{align*}
The second term in the right-hand side of the last inequality tends to $0$ (uniformly in $s$) because $\tilde{\psi}_k$ converges pointwise to $1$ outside a pluripolar set and $R$ has no mass on pluripolar sets. Using Lemma \ref{le-bangnhauTnonpluripolar}, we have 
\begin{align*}
\tilde{\psi}_k  \psi_{s} R &= \tilde{\psi}_k  \psi_{s}  d v_{s} \wedge \dc v_{s} \wedge T_{s} \\
&= \tilde{\psi}_k  \psi_{s}  d v_{s} \wedge \dc v_{s} \wedge T= \tilde{\psi}_k  \psi_{s}  d v_{k} \wedge \dc v_{k} \wedge T,
\end{align*} 
here we used the plurifine topology properties with respect to $T$ (see \cite[Theorem 2.9]{Viet-generalized-nonpluri}), thanks to the fact that   $\varphi_{j,k} = \varphi_{j, s}$ on $\{\tilde{\psi}_k \not = 0\}$ for $j=1,2$ (recall $s \ge  k$), and they are bounded psh functions. Letting $s \to \infty$ gives
$$\tilde{\psi}_k R = \tilde{\psi}_k  \bold{1}_{\cup_{j=1}^m \{u_j > - \infty\}}  d v_{k} \wedge \dc v_{k} \wedge T=\tilde{\psi}_k  d v_{k} \wedge \dc v_{k} \wedge T$$
because the current $d v_{k} \wedge \dc v_{k} \wedge T$ has no mass on pluripolar sets. Now letting $k \to \infty$ gives the desired assertion. This finishes the proof.
\endproof

Thanks to Proposition \ref{pro-haidinhnghiaddctuongduong}, we can use the right-hand side of (\ref{eq-haidinhnghiatuongduong}) to define $\langle d v \wedge \dc v \dot{\wedge} T \rangle$ in the case where $T$ is the non-pluripolar product of some closed positive $(1,1)$-currents. By the same reason, in this case, we will use the expression $dv \wedge \dc w \wedge T_1 \wedge \ldots \wedge T_{n-1}$ to denote $\big\langle d v \wedge \dc w \dot{\wedge} \langle T_1 \wedge \cdots \wedge T_{n-1}\rangle \big\rangle$ whenever it is well-defined.

\subsection{Auxiliary facts on weights} \label{subsec-auxi}

In this subsection, we present some facts about weights needed for the proofs of main results.

Recall that  $\widetilde{\mathcal{W}}^-$ is the set of all convex, non-decreasing functions 
$\chi: \R_{\le 0}\rightarrow\R_{\le 0}$ such that $\chi(0)=0$ and $\chi \not \equiv 0$.  Let $M \ge 1$ be a constant and  $\mathcal{W}^+_M$ the usual space of  increasing concave functions $\chi: \R_{\le 0} \to \R_{ \le 0}$ such that $\chi(0)=0$, $\chi \not \equiv 0$, and $|t \chi'(t)| \le M|\chi(t)|$ for every $t \le 0$. 
 We have the following lemmas. 

\begin{lemma} \label{lem approchi} Let $c>0$, $0<\delta<1$ and $\chi: \R\rightarrow\R$ such that 
	$\chi(t)=ct$ for every $t\geq-\delta$ and $\chi|_{(-\infty, 0]}\in \widetilde{\mathcal{W}}^-\cup \mathcal{W}^+_M$ ($M\geq 1$).  Let $g$ be a smooth radial cut-off function supported in $[-1,1]$ on $\R$, \emph{i.e,} $g(t)= g(-t)$ for $t \in \R$, $0 \le g \le 1$ and $\int_\R g(t) dt =1$. Put $g_\epsilon(t):=  \epsilon^{-1}g(\epsilon t)$ for every constant $\epsilon >0$ and $\chi_\epsilon:= \chi * g_\epsilon$ (the convolution of $\chi$ with $g_\epsilon$). Then  the following assertions are true:
	
	$(i)$ if $\chi \in \widetilde{\mathcal{W}}^-$, then   $\chi_\epsilon|_{(-\infty, 0]}\in \widetilde{\mathcal{W}}^-$
	for every $0<\epsilon<\delta$,  $\chi_\epsilon\searrow \chi$ as $\epsilon\searrow 0$
	and $\sup (\chi_{\epsilon}-\chi)\leq c\epsilon$;
	
	$(ii)$ if $\chi \in \mathcal{W}^+_M$ and $0<\epsilon<\delta^2/2$ then  
	$\chi_\epsilon|_{(-\infty, 0]}\in \mathcal{W}^+_{M/(1-\delta)}$. Moreover, if $0<\epsilon<\delta^2/8$ then
	$$\overline{\chi}_{\epsilon}:=\chi_{\epsilon}( \cdot +\epsilon)-c\epsilon\in \mathcal{W}^+_{M/(1-\delta)^2}, \quad 
	\overline{\chi}_{\epsilon}\geq\chi-c\epsilon,$$
	 and $\overline{\chi}_{\epsilon}$ converges uniformly to $\chi$ as $\epsilon \to 0$ on compact subsets in $\R$.
\end{lemma}
\begin{proof}
	The part $(i)$ follows from \cite[Lemma 2.1]{Vu_Do-MA}. The part $(ii)$ can be obtained more or less by similar arguments as in the last reference. We provide details for readers' convenience. 	
	It is clear that $\chi_{\epsilon}$ is a concave, increasing function with $\chi_{\epsilon}(0)=0$.
	 We will show that
	 \begin{equation}\label{eq1 lem2.7}
	 	\chi_{\epsilon}'(t)\leq\dfrac{M}{1-\delta}\dfrac{\chi_{\epsilon}(t)}{t},
	 \end{equation}
	for every $t<0$ and $0<\epsilon<\delta^2/2$.
	
	If $t<-\dfrac{\delta}{2}$ then we have
	\begin{align*}
	\chi_{\epsilon}'(t)=\int_{-\epsilon}^{\epsilon}\chi'(t-s)g_{\epsilon}(s)ds
		\leq\int_{-\epsilon}^{\epsilon}\dfrac{M \chi (t-s)}{t-s}g_{\epsilon}(s)ds
		&\leq \int_{-\epsilon}^{\epsilon}\dfrac{M \chi (t-s)}{t+\epsilon}g_{\epsilon}(s)ds\\
		&=\dfrac{M\chi_{\epsilon}(t)}{t+\epsilon}\\
		&=\dfrac{M t}{t+\epsilon}\dfrac{\chi_{\epsilon}(t)}{t}\\
		&\leq \dfrac{M}{1-\delta}\dfrac{\chi_{\epsilon}(t)}{t},
	\end{align*}
	for every $0<\epsilon<\delta^2/2$.
	
On the other hand, if $t\geq -\dfrac{\delta}{2}$, then  $\chi_{\epsilon}(t)=\chi(t)=ct$ for every  $0<\epsilon<\delta^2/2$. As a consequence, we have
$$\chi_{\epsilon}'(t)=\chi'(t)\leq \dfrac{M\chi(t)}{t}=M\dfrac{\chi_{\epsilon}(t)}{t}.$$
Thus, \eqref{eq1 lem2.7} follows.	Hence, $\chi_\epsilon|_{(-\infty, 0]}\in \mathcal{W}^+_{M/(1-\delta)}$.

Now, we consider $\overline{\chi}_{\epsilon}$. Since $\chi$ is increasing, one sees that 	$\overline{\chi}_{\epsilon}\geq\chi-c\epsilon$ and $\overline{\chi}_{\epsilon}$ converges uniformly to $\chi$ as $\epsilon \to 0$ on compact subsets in $\R$. It remains
to show that
$\overline{\chi}_{\epsilon}\in \mathcal{W}^+_{M(1+\delta)/(1-\delta)}$ for every $0<\epsilon<\delta^2/8$.
Note that
$$\overline{\chi}_{\epsilon}=h_{\epsilon}* g_{\epsilon},$$
where $h_{\epsilon}(t)=\chi(t+\epsilon)-c\epsilon$. The function $\overline{\chi}_\epsilon(t)$ is concave, increasing
and $\overline{\chi}+\epsilon(0)=0$.

 If $-\delta/2\leq t<0$ then  $h_{\epsilon}(t)=\chi(t)=ct$ for every $0<\epsilon<\delta^2/2$. Therefore
 $$h_{\epsilon}'(t)=\chi'(t)\leq \dfrac{M\chi(t)}{t}=M\dfrac{h_{\epsilon}(t)}{t}.$$
 If $t<-\delta/2$ then 
 \begin{align*}
 	h_{\epsilon}'(t)=\chi'(t+\epsilon)\leq M\dfrac{\chi (t+\epsilon)}{t+\epsilon}
 	\leq M\dfrac{\chi(t+\epsilon)-c\epsilon}{t+\epsilon}
 	=M\dfrac{h_{\epsilon}(t)}{t+\epsilon}
 	&=\dfrac{M t}{t+\epsilon}\dfrac{h_{\epsilon}(t)}{t}\\ 	
 	&\leq \dfrac{M }{1-\delta}\dfrac{h_{\epsilon}(t)}{t},
 \end{align*}
 for every $0<\epsilon<\delta^2/2$.
 
 Then, for every $0<\epsilon<\delta^2/2$, we have $h_{\epsilon}\in \mathcal{W}^+_{M/(1-\delta)}$ 
 and $h_{\epsilon}=ct$ for every $t\geq-\delta/2$. Hence, for every $0<\epsilon<\delta^2/8$, we have
 $$\overline{\chi}_{\epsilon}=h_{\epsilon}* g_{\epsilon}\in 
  \mathcal{W}^+_{\frac{M}{(1-\delta)(1-\delta/2)}}
  \subset  \mathcal{W}^+_{\frac{M}{(1-\delta)^2}}.$$
 The proof is completed.
\end{proof}

\begin{lemma}\label{le-regularizedchi}
	Let $\chi, \tilde{\chi}\in \widetilde{\mathcal{W}}^-\cup \mathcal{W}^+_M$ ($M \ge 1$) such that $\tilde{\chi} \le \chi$.	Then, there exist sequences of functions
	$\chi_j, \tilde{\chi}_j\in\widetilde{\mathcal{W}}^-\cup \mathcal{W}^+_{M_j}$
(with $M_j\searrow M$ as $j\to\infty$ )	satisfying the following conditions:
	\begin{itemize}
		\item 	$\chi_j\in \Cc^{\infty}(\R)$ for every $j$;
		\item  $\chi_j\geq\tilde{\chi}_j$ and $\chi_j\geq\chi-2^{-j}$ for every $j$ big enough;
		\item  $\tilde{\chi}-2^{-j}\leq\tilde{\chi}_j\leq\tilde{\chi}$ on $(-\infty, -1]$ for every $j$ big enough;
		\item  	$\chi_j$ converges uniformly to $\chi$ on compact subsets in $\R_{\le 0}$.
	\end{itemize}
\end{lemma}
\begin{proof}
 We  split the proof into two cases.\\
 
 \noindent
{\bf Case 1}:  $\chi\in \widetilde{\mathcal{W}}^-$.\\
For every $j\geq 1$, we denote 
$$\overline{\chi}_j(t)=\begin{cases}
	\max\{\chi(t), c_jt\}\qquad\mbox{if}\quad t<0,\\
	c_jt\qquad\mbox{if}\quad t\geq 0,
\end{cases}$$
where $$c_j:=\dfrac{-\chi(-2^{-j})}{2^{-j}} \cdot$$
Then $\overline{\chi}_j$ satisfies the hypothesis of Lemma \ref{lem approchi} for $\delta:= 2^{-j}$.
Let $g$ be a smooth radial cut-off function supported in $[-1,1]$ on $\R$, \emph{i.e,} $g(t)= g(-t)$ for $t \in \R$, $0 \le g \le 1$ and $\int_\R g(t) dt =1$. For every $j\geq 1$, we define
\begin{center}
	$\chi_j=\overline{\chi}_j * g_{4^{-j-1}}\quad$ and $\quad \tilde{\chi}_j=\tilde{\chi}$.
\end{center}
By Lemma \ref{lem approchi}, we have $\chi_j$ and $\tilde{\chi}_j$ satisfy the desired conditions.\\

\noindent
{\bf Case 2}:  $\chi\in\mathcal{W}^+_M$.\\
Since $\chi \ge \tilde{\chi}$, we also have $\tilde{\chi}\in\mathcal{W}^+_M$.
Assume that $g$ and $c_j$ are as in Case 1.
 For every $j\geq 1$, we define
 $$\overline{\chi}_j(t)=\begin{cases}
 	\min\{\chi(t), c_jt\}\qquad\mbox{if}\quad t<0,\\
 	c_jt\qquad\mbox{if}\quad t\geq 0,
 \end{cases}$$
and
	$$\chi_j(t)=(\overline{\chi}_j(\cdot+4^{-j-1}) * g_{4^{-j-1}})(t)-c_j4^{-j-1}.$$
We also denote
 $\tilde{\chi}_j(t)=\min\{\tilde{\chi}(t), \chi_j(t)\}$.
  By  Lemma \ref{lem approchi}, we have $\chi_j$ and $\tilde{\chi}_j$ satisfy the desired conditions. The proof is completed.
\end{proof}

Let $\phi$ be a negative $\theta$-psh function. We denote by $\PSH(X, \theta, \phi)$ the set of $\theta$-psh functions $u \le \phi$. Recall that by monotonicity, we always have $\int_X \theta^n_u \le \int_X \theta^n_\phi$, where for every $\theta$-psh function $v$, we put $\theta_v:= \ddc v +\theta$.  We also define by $\mathcal{E}(X,\theta, \phi)$ the set of $u \in \PSH(X, \theta, \phi)$ of full Monge-Amp\`ere mass with respect to $\phi$, \emph{i.e,} $\int_X\theta_u^n = \int_X \theta_\phi^n.$ 

Let $\chi \in \widetilde{\mathcal{W}}^- \cup \mathcal{W}^+_M$, and $u \in \PSH(X, \theta, \phi)$. We put
$$E_{\chi, \theta,\phi}(u):= \int_X - \chi(u- \phi) \theta_u^n.$$
We also define by $\mathcal{E}_\chi(X, \theta,\phi)$ the set of $u \in \mathcal{E} (X, \theta, \phi)$ with $E_{\chi, \theta, \phi}(u)<\infty$.

\begin{lemma} \label{le-sosanhnangnluongintegrabig} Let $\chi \in \widetilde{\mathcal{W}}^-\cup \mathcal{W}^+_M$ and $u_1,u_2 \in \mathcal{E}_\chi(X, \theta, \phi)$. Then there exists a constant $C_1>0$ depending only on $n$ and $M$
	 such that
$$- \int_X \chi(u_1- \phi) \theta_{u_2}^n \le C_1 \sum_{j=1}^2 E_{\chi, \phi, \theta}(u_j),$$
and 
$$E_{\chi,\theta,\phi}\big(au_1+(1-a)\big)u_2 \le C_1 \sum_{j=1}^2 E_{\chi, \phi, \theta}(u_j),$$
for every $0<a<1$.
Furthermore if $u_1 \ge u_2$, then 
$$E_{\chi, \phi, \theta}(u_1) \le C_2 E_{\chi, \phi,\theta} (u_2),$$
for some constant $C_2$ depending only on $n$ and $M$.
\end{lemma}

\proof The first and third inequalities are  from \cite[Lemma 3.2]{Vu_Do-MA} (see also \cite[Propositions 2.3, 2.5]{GZ-weighted} for the case where $\phi=0$ and $\theta$ is a K\"ahler form). The second desired inequality was  implicitly proved in  the proof of convexity of finite energy classes in \cite[Proposition 3.3]{Viet-convexity-weightedclass} (in a much broader context). Alternatively one can use properties of envelopes in \cite{Lu-Darvas-DiNezza-mono} to get the same conclusion.  We prove here the second desired inequality using ideas from \cite{Viet-convexity-weightedclass} for readers' convenience. 

 Considering $u_j - \epsilon$ for $\epsilon>0$ instead of $u_j$, and taking $\epsilon \to 0$ later, without loss of generality, we can assume that $u_j <\phi \le 0$ for $j=1.2$. By replacing $u_j,\theta$ by $u_j- \phi$, $\theta_\phi$ respectively, we can assume that $\phi =0$, but $\theta$ is no longer a smooth form but a closed positive $(1,1)$-current. This change causes no trouble for us.  Let $v:= au_1+(1-a)u_2$.  Observe that $X \subset \{u_1 < u_2\} \cup \{u_1 > 2 u_2\}$. 
Hence 
\begin{align*}
	E_{\chi,\theta}(v)
 &\le \int_{\{u_1 < u_2\}} -\chi(v) \theta_v^n +\int_{\{u_1 > 2u_2\}} -\chi(v) \theta_v^n\\
& \le \int_{\{u_1 < u_2\}} -\chi(v) \theta_v^n +\int_{\{u_1 > 2u_2\}}-\chi(v) \theta_v^n\\
& \le  \sum_{k=0}^{n}\left(\int_{\{u_1 < u_2\}} -\chi(u_1) \theta_{u_1}^k \wedge \theta_{u_2}^{n-k} +\int_{\{u_1 > 2u_2\}}-\chi((2-a/2)u_2) \theta_{u_1}^k \wedge \theta_{u_2}^{n-k}\right)\\
& \le  \sum_{k=0}^{n}\int_{\{u_1 < u_2\}} -\chi(u_1) \theta_{u_1}^k \wedge \theta_{\max\{u_1,u_2\}}^{n-k} +\\
&\quad + \sum_{k=0}^n \int_{\{u_1 > 2u_2\}}-2^{k+1} \chi(u_2) \theta_{\max\{u_1/2, u_2\}}^k \wedge \theta_{u_2}^{n-k}\\
& \le  \sum_{k=0}^{n}\left(\int_X -\chi(u_1) \theta_{u_1}^k \wedge \theta_{\max\{u_1,u_2\}}^{n-k} +2^{k+1} \int_X- \chi(u_2) \theta_{\max\{u_1/2, u_2\}}^k \wedge \theta_{u_2}^{n-k}\right)\\
& \lesssim  E_{\chi,\theta}(u_1)+E_{\chi,\theta}(\max\{u_1, (u_1+u_2)/2\})+ E_{\chi,\theta}(u_2)+  E_{\chi, \theta}(\max\{u_1/4+u_2/2, u_2\}) \\
& \lesssim  E_{\chi,\theta}(u_1)+ E_{\chi,\theta}(u_2),
\end{align*}
where the two last estimates hold due to  the first and third inequalities of the lemma.
 This finishes the proof. 
\endproof

\begin{lemma}\label{le-uocluongchiepsilon} 
Let $\chi, \tilde{\chi} \in \widetilde{\mathcal{W}}^-\cup \mathcal{W}^+_M$ such that $\tilde{\chi} \le \chi$ and 
let $u_1, u_2,..., u_{n+1} \in \mathcal{E}(X, \theta,\phi)$. Denote $\varrho:= \vol(\theta_\phi)$. Then there exists a constant $C>0$ depending only on $n$ and $M$ such that  
$$- \int_X \chi(\epsilon (u_1-\phi))\theta_{u_2}\wedge...\wedge\theta_{u_{n+1}} \le C\, B\varrho  (1-\tilde{\chi}(-1))Q_{0}(\epsilon),$$
for every $0<\epsilon\leq 1$, where 
\begin{center}
	$B=1+\max_{1\leq j\leq n+1}E_{\tilde{\chi}, \theta, \phi}(u_j)/\varrho\quad$
	and $\quad Q_{0}(\epsilon):= \sup_{\{t \le -1\}}\dfrac{\chi(\epsilon t)}{\tilde{\chi}(t)}\cdot$
\end{center}

\end{lemma}

\proof  Let $L$ be the left-hand side of the desired inequality. 
We have
\begin{align*}
L  & \le - \int_{\{u_1 \ge\phi-1\}} \chi(\epsilon (u_1-\phi))  \theta_{u_2}\wedge...\wedge\theta_{u_{n+1}} 
- \int_{\{u_1 < \phi-1\}} \chi(\epsilon (u_1-\phi))  \theta_{u_2}\wedge...\wedge\theta_{u_{n+1}}\\
& \le  -\chi(-\epsilon)\varrho- Q_{0}(\epsilon)\int_{\{u_1 <\phi -1\}} \tilde{\chi}(u_1-\phi) \theta_{u_2}\wedge...\wedge\theta_{u_{n+1}}\\
&\le -\varrho Q_{0}(\epsilon)\tilde{\chi}(-1) - Q_{0}(\epsilon)\int_{X} \tilde{\chi}(u_1-\phi) \theta_{u_2}\wedge...\wedge\theta_{u_{n+1}}\\
&\le  -\varrho Q_{0}(\epsilon)\tilde{\chi}(-1)+C Q_{0}(\epsilon)\max_{1\leq j\leq n+1}E_{\tilde{\chi}, \theta, \phi}(u_j),\\
\end{align*}
where $C>0$ depends only on $n$ and $M$. The last estimate holds due to Lemma \ref{le-sosanhnangnluongintegrabig}.
Thus the desired inequality follows. 
\endproof

By the convexity/concavity and by the assumption $\tilde{\chi}\leq\chi$, we have
\begin{equation}
	\begin{cases}
		Q_{0}(\epsilon)\geq\epsilon Q_{0}(1)\quad\mbox{if}\quad\chi\in\widetilde{\mathcal{W}}^{-},\\
		Q_{0}(\epsilon)\leq\epsilon Q_{0}(1)\quad\mbox{if}\quad\chi\in\mathcal{W}_M^{+},
	\end{cases}
\end{equation}
for every $0<\epsilon \le 1$.
  Moreover, if $\chi\in\widetilde{\mathcal{W}}^{-}$ and $\chi(t)/\tilde{\chi}(t)\rightarrow 0$ as $t\rightarrow-\infty$, then by the definition of $Q_{0}$, we also have
\begin{equation}\label{eq chantrenQ2}
	Q_{0}(\epsilon)\leq \dfrac{\chi (-\epsilon^{1/2})}{\tilde{\chi}(-1)}
	+\sup_{\{t \le -\epsilon^{-1/2}\}}\dfrac{\chi(t)}{\tilde{\chi}(t)}\stackrel{\epsilon\to 0^+}{\longrightarrow} 0.
\end{equation}

Let $u_1,u_2 \in \mathcal{E}_\chi(X, \theta, \phi)$, and   
  $v:= \max\{u_1, u_2\}$. Put
$$\nu(u_1,u_2):= \chi(-|u_1- u_2|) (\theta_{u_2}^n- \theta_{u_1}^n),$$
and 
\begin{align} \label{eq-dinhgnhiaIchi}
I_\chi(u_1,u_2) &:= \int_{\{u_1< u_2\}} \nu(u_1,u_2)+\int_{\{u_1> u_2\}} \nu(u_2,u_1) = \int_X \nu(u_1,v) + \int_X \nu(u_2,v).
\end{align}

\begin{proposition}\label{prop Ichi positive}
	Let $\chi \in \widetilde{\mathcal{W}}^- \cup \mathcal{W}^+_M$.
Let $\phi$ is a negative $\theta$-psh function and $u_1, u_2\in \mathcal{E}_\chi(X, \theta,\phi)$.
	Then
	$$I_\chi(u_1,u_2)\geq 0.$$
\end{proposition}

\begin{proof}
	Denote $\mu=\theta_{u_2}^n- \theta_{u_1}^n$.
Since $\chi$ is absolutely continuous, we have $\chi$ is differentiable almost everywhere and 
$-\chi(t)=\int_t^0\chi'(s)ds$ for every $t<0$.
Hence
	\begin{align*}
		\int_{\{u_1< u_2\}} \nu(u_1,u_2)&=-\int_{\{u_1< u_2\}}\left(\int_{u_1-u_2}^0\chi'(t)dt\right)d\mu\\
		&=-\int_{\{u_1< u_2\}}\left(\int_{-\infty}^0\chi'(t)\mathbf{1}_{\{u_1<u_2+t\}}dt\right)d\mu\\
		&=-\int_{-\infty}^0\chi'(t)\mu\{u_1<u_2+t\}dt.
\end{align*}
Moreover, it follows from \cite[Lemma 2.3]{Lu-Darvas-DiNezza-logconcave} that $\mu\{u_1<u_2+t\}\leq 0$
for every $t\leq 0$. Hence
$$\int_{\{u_1< u_2\}} \nu(u_1,u_2)=-\int_{-\infty}^0\chi'(t)\mu\{u_1<u_2+t\}dt\geq 0.$$
Similarly, we have
$$\int_{\{u_2< u_1\}} \nu(u_2, u_1)\geq 0.$$
Thus $$I_\chi(u_1,u_2)= \int_{\{u_1< u_2\}} \nu(u_1,u_2)+\int_{\{u_2< u_1\}} \nu(u_2, u_1) \geq 0.$$
\end{proof}

\section{Stability estimates for fixed singularity type} \label{sec-fixedtype}

\subsection{Main results} 

Let $\chi, \tilde{\chi} \in \widetilde{\mathcal{W}}^-\cup\mathcal{W}_M^+$ ($M\geq 1$) such that $\tilde{\chi} \le \chi$.   For each constant $t\geq 0$,
 we denote
\begin{equation}\label{eq QB}
		Q(t)= Q_{\chi,\tilde{\chi}}(t):=
	\begin{cases}
		1 \quad\mbox{if}\quad t\ge 1,\\
			\left(Q_{0}(t)/Q_{0}(1)\right)^{1/2}\quad\mbox{if}\quad 0<t<1\quad\mbox{and}\quad\chi\in\widetilde{\mathcal{W}}^{-},\\
		t^{1/2}\quad\mbox{if}\quad 0<t<1\quad\mbox{and}\quad\chi\in\mathcal{W}_M^{+},\\
		\lim_{s\to 0^+}Q(s)\quad\mbox{if}\quad t=0.
	\end{cases}
\end{equation}
where $Q_{0}$ is defined as in Lemma \ref{le-uocluongchiepsilon}. We remove the subscript $\chi,\tilde{\chi}$ from $Q_{\chi,\tilde{\chi}}$ if $\chi, \tilde{\chi}$ are clear from the context.  Note that $Q$ is increasing continuous function in $t$ and 
\begin{equation}
	Q(0)=0\quad\mbox{if either }\quad \chi, \tilde{\chi} \in \mathcal{W}_M^+\quad\mbox{or}\quad
	\lim\limits_{t\to -\infty}\dfrac{\chi(t)}{\tilde{\chi}(t)}=0.
\end{equation}

  Now, we state the main results of this section. For the convenience, we normalize energies with respect to $\varrho:= \int_X \theta_\phi^n$ as follows
  $$E^0_{\tilde{\chi}, \theta, \phi}:=\varrho^{-1}E_{\tilde{\chi}, \theta, \phi}, \quad I^0_\chi(u_1,u_2)=\varrho^{-1}I_\chi(u_1,u_2).$$

\begin{theorem} \label{th-lowerenergy}  Let $\theta$ be a closed smooth real $(1,1)$-form  and $\phi$ be a  negative $\theta$-psh function such that $\varrho:=\vol(\theta_\phi)>0$.  Let $\chi, \tilde{\chi} \in \widetilde{\mathcal{W}}^-\cup\mathcal{W}^+_M$ ($M\geq 1$) such that $\tilde{\chi} \le \chi$. Let $B \ge 1$ be a constant and let $u_j, \psi_j \in \mathcal{E}(X, \theta, \phi)$ satisfy $u_1 \le u_2$ and 
$$E^0_{\tilde{\chi}, \theta, \phi}(u_j)+E^0_{\tilde{\chi}, \theta, \phi}(\psi_j) \le B,$$
for $j=1,2$. Then there exists a constant $C_n>0$ depending only on $n$ and $M$ such that 
\begin{align}\label{ine-thlowerenergy}
\int_X  -\chi (u_1-u_2)  (\theta_{\psi_1}^n- \theta_{\psi_2}^n) \le 
C_n\varrho B^2(1-\tilde{\chi}(-1))^2 Q^{\circ n}\big(I^0_\chi(u_1,u_2)\big),
\end{align}
where   $Q$ is defined by \eqref{eq QB}, and $Q^{\circ n}:= Q \circ Q \circ \cdots \circ Q$ ($n$-iterate of $Q$).
\end{theorem}

Since the measure $\theta_{\psi_1}^n- \theta_{\psi_2}^n$ is not positive, we need the following consequence of the above theorem for later applications on stability estimates.

\begin{theorem} \label{th-lowerenergy-kocochuanhoa}
	 Let $\theta$ be a closed smooth real $(1,1)$-form  and $\phi$ be a  negative $\theta$-psh function such that  $\phi=P_{\theta}[\phi]$, $\varrho:=\vol(\theta_\phi)>0$
	 and $\theta\leq A\omega$ for some constant $A \ge 1$.  Let 
	 $\chi, \tilde{\chi} \in \widetilde{\mathcal{W}}^-\cup\mathcal{W}^+_M$ ($M\geq 1$)
	 such that $\tilde{\chi} \le \chi$. Let $B \ge 1$ be a constant and $u_1, u_2, \psi \in \mathcal{E}(X, \theta, \phi)$ satisfying
	$$
	E^0_{\tilde{\chi}, \theta, \phi}(u_1)+
	E^0_{\tilde{\chi}, \theta, \phi}(u_2)+E^0_{\tilde{\chi}, \theta, \phi}(\psi) \le B,$$
	for $j=1,2$. Then, for every constant $m>0$ and $0<\gamma<1$, there exists a constant $C>0$ depending on $n, M, X, \omega, m$ and $\gamma$  such that 
\begin{align*}
\int_X  -\chi\big(-|u_1- u_2|\big) \theta_\psi^n \le -\varrho \chi\left(-|a_1-a_2|-\lambda^m\right)
+C\varrho A^{(1-\gamma)/m}(B-\tilde{\chi}(-A))^2(1-\tilde{\chi}(-1))^2\lambda^{\gamma},
\end{align*}
where $a_j:= \sup_Xu_j$ and $\lambda=Q^{\circ n}\big(I^0_\chi(u_1,u_2)\big)$.
\end{theorem}



\subsection{Proof of Theorem \ref{th-lowerenergy}}  \label{subsec-key}

Here is the first step in the proof of Theorem \ref{th-lowerenergy}.
\begin{lemma} \label{le-lientuccungkididwedgedc} If  Theorem \ref{th-lowerenergy} holds for $u_j, \psi_j$ of the same singularity type as $\phi$, then it holds for the general case.
\end{lemma}

\proof Let $u_j, \psi_j$ ($j=1,2$) be as in the statement of Theorem \ref{th-lowerenergy}.  For every
$k>0$, we denote $u_{j,k}:= \max\{u_j,  \phi -k\}$
 and $\psi_{j, k}=\max\{\psi_j, \phi-k\}$. By Lemma \ref{le-sosanhnangnluongintegrabig}, there exists a constant $C_1>0$ depending
only on $n$ and $M$ such that
$$ E^0_{\tilde{\chi},\theta, \phi}(u_{j,k})+ E^0_{\tilde{\chi},\theta, \phi}(\psi_{j,k})\leq C_1B,$$
for $j=1, 2$ and for every $k>0$. Therefore, by the assumption, there exists a constant $C_2>0$
depending only on $n$ and $M$ such that
$$\int_X  -\chi (u_{1, k}-u_{2, k})  (\theta_{\psi_{1, l}}^n- \theta_{\psi_{2, l}}^n) \le 
C_2\varrho B^2(1-\tilde{\chi}(-1))^2 Q^{\circ (n)}(I^0_\chi(u_{1, k},u_{2, k})),$$
for every $k, l>0$. Letting $l\rightarrow\infty$ and using \cite[Theorem 2.2]{Darvas-Lu-DiNezza-singularity-metric}, we get
\begin{equation}\label{eq0 le-lientuccungkididwedgedc}
	\int_X  -\chi (u_{1, k}-u_{2, k})  (\theta_{\psi_{1}}^n- \theta_{\psi_{2}}^n) \le 
	C_2\varrho B^2(1-\tilde{\chi}(-1))^2 Q^{\circ (n)}(I^0_\chi(u_{1, k},u_{2, k}))
\end{equation}
for every $k>0$.
 We will show that
\begin{equation}\label{eq1 le-lientuccungkididwedgedc}
	I_\chi(u_1,u_2)=\lim_{k \to \infty}I_{\chi}(u_{1,k}, u_{2,k}).
\end{equation}
Denote
$$f:= \chi(u_{1}- u_{2}) (\theta_{u_{2}}^n - \theta_{u_{1}}^n), \quad f_k:=  \chi(u_{1,k}- u_{2,k}) (\theta_{u_{2,k}}^n - \theta_{u_{1,k}}^n).$$
We have 
\begin{align*}
	I_{\chi}(u_{1,k}, u_{2,k})  = \int_{X} f_k 
	&= \int_{\{u_1> \phi-k\}}f_k+ \int_{\{u_1 \le \phi-k\}}  f_k\\
	& =\int_{\{u_1> \phi-k\}} f+ \int_{\{u_1 \le \phi  -k\}}  f_k\\
	&= I_\chi(u_1,u_2)- \int_{\{u_1\le \phi -k\}}f + \int_{\{u_1 \le  \phi -k\}}  f_k.
\end{align*}
Then
\begin{align*}
	|I_{\chi}(u_{1,k}, u_{2,k})-I_\chi(u_1,u_2)|&= 
	\bigg|\int_{\{u_1\le \phi -k\}}f + \int_{\{u_1 \le  \phi -k\}}  f_k\bigg|\\
	&\leq\int_{\{u_1\le \phi -k\}}\mu+ \int_{\{u_1 \le  \phi -k\}} -\chi(u_{1,k}- u_{2,k}) (\theta_{u_{2,k}}^n + \theta_{u_{1,k}}^n)\\
	&\leq \int_{\{u_1\le \phi -k\}}\mu+ \int_{\{u_1 \le  \phi -k\}} -\chi(-k) (\theta_{u_{2,k}}^n + \theta_{u_{1,k}}^n),
\end{align*}
where $\mu=-\chi(u_1-\phi)(\theta_{u_1}^n+\theta_{u_2}^n)$.
By Lemma \ref{le-sosanhnangnluongintegrabig}, we have  $\int_X\mu<\infty$. Then it follows from Lebesgue's
dominated convergence theorem that $\lim_{k\to\infty}\int_{\{u_1\le \phi -k\}}\mu=0$. Therefore, 
\begin{equation}\label{eq2 le-lientuccungkididwedgedc}
	\limsup\limits_{k\to\infty}	|I_{\chi}(u_{1,k}, u_{2,k})-I_\chi(u_1,u_2)|
	\leq \limsup\limits_{k\to\infty}\int_{\{u_1 \le  \phi -k\}} -\chi(-k) (\theta_{u_{1,k}}^n + \theta_{u_{2,k}}^n).
\end{equation} 
By the fact that 
 $$\int_X\theta_{u_{1,k}}^n=\int_X\theta_{u_{1,k}}^n=\int_X\theta_{\phi}^n, \quad \mathbf{1}_{\{u_1 \le  \phi -k\}}\theta_{u_{j,k}}^n=\mathbf{1}_{\{u_1 \le  \phi -k\}}\theta_{u_{j}}^n \quad (j=1,2),$$
  we have
\begin{equation}\label{eq3 le-lientuccungkididwedgedc}
	-\chi(-k)\int_{\{u_1 \le  \phi -k\}} (\theta_{u_{1,k}}^n + \theta_{u_{2,k}}^n)
	= -\chi(-k)\int_{\{u_1 \le  \phi -k\}} (\theta_{u_{1}}^n + \theta_{u_{2}}^n)
	\leq \int_{\{u_1 \le  \phi -k\}}\mu.
\end{equation}
By using \eqref{eq2 le-lientuccungkididwedgedc}, \eqref{eq3 le-lientuccungkididwedgedc} and 
the fact $\lim_{k\to\infty}\int_{\{u_1\le \phi -k\}}\mu=0$, we get \eqref{eq1 le-lientuccungkididwedgedc}.
Now, combining \eqref{eq0 le-lientuccungkididwedgedc} and \eqref{eq1 le-lientuccungkididwedgedc}, we obtain 
$$	\int_X  -\chi (u_{1}-u_{2})  (\theta_{\psi_{1}}^n- \theta_{\psi_{2}}^n) \le 
C_2\varrho B^2(1-\tilde{\chi}(-1))^2 Q^{\circ (n)}(I^0_\chi(u_{1}, u_{2})).$$
The proof is completed.
\endproof

\begin{lemma}\label{le-secondhieuu1u2} 
	Let $M\geq 1$ and 	$\chi, \tilde{\chi}\in\widetilde{\mathcal{W}}^-\cup\mathcal{W}^{+}_M$
	such that $\tilde{\chi}\leq\chi$ and $\chi\in \Cc^1(\R)$.
Let $u_1, u_2, ..., u_{n+2} \in \mathcal{E}(X, \theta, \phi)$ such that $u_1 \le u_2$ and
	$u_j-\phi$ is bounded ($j=1, 2, ..., n+2$), where $\phi$ is a negative $\theta$-psh function satisfying 
	$\varrho:=\vol(\theta_\phi)>0$. 
 Denote 
 $$T=\theta_{u_4}\wedge...\wedge\theta_{u_{n+2}}, \quad I= \left|\int_X\chi'(u_1-u_2)d(u_1-u_2)\wedge d^c(u_1-u_3)\wedge T\right|,$$
 and
 $$J=\int_X\chi'(u_1-u_2)d(u_1-u_2)\wedge d^c(u_1-u_2)\wedge T.$$
	Then there exists $C>0$ depending only on $n$ and $M$ such that
	$$I \leq C\varrho B(1-\tilde{\chi}(-1))
	Q(J/\varrho),$$
 where $B:=\sum_{j=1}^{n+2} \max\{E^0_{\tilde{\chi},\theta, \phi}(u_j),1 \}$
	and $Q$ is defined by \eqref{eq QB}. 
\end{lemma}

Clearly if $\chi \in \widetilde{\mathcal{W}}^-$, then the above constant $C$ does not depend on $M$. 

\proof 
In this proof, we use the symbols $\lesssim$ and $\gtrsim$ for inequalities modulo a  constant depending only on $n$ and $M$.
By Theorem \ref{th-integrabypart} and Lemma \ref{le-sosanhnangnluongintegrabig}, we have
$$I=\left|\int_X -\chi(u_1-u_2)dd^c(u_1-u_3)\wedge T\right|\lesssim \varrho B= \varrho B Q(1).$$ Therefore, without loss of generality,
we can assume that $J/\varrho<1$.
Approximating $u_3$ by $u_3-\delta$ with $\delta\searrow 0$, we can assume that $u_3<\phi$ on $X$.

For each $0<\epsilon<1/2$ we denote
\begin{center}
	$U(\epsilon)=\{u_1-u_2<\epsilon(u_1+u_3-2\phi)\}, V(\epsilon)=\{u_1-u_2>\epsilon(u_1+u_3-2\phi)\},$
\end{center}
	and $\Gamma(\epsilon)=\{u_1-u_2=\epsilon(u_1+u_3-2\phi)\}$.
Since $\Gamma(\epsilon_1)\cap\Gamma(\epsilon_2)=\emptyset$ for every $\epsilon_1\neq\epsilon_2$ (note $u_3<\phi$),
we have
\begin{equation}\label{eq epsilon lemforpr3.5}
	\int_{\Gamma(\epsilon)}d(u_1-u_3)\wedge d^c(u_1-u_3)\wedge T=0,
\end{equation}
for almost everywhere $\epsilon\in (0, 1/2)$. 

Let $0<\epsilon<1/2$ be a constant  satisfying \eqref{eq epsilon lemforpr3.5}. To simplify the notation, from now on, we write $U,V, \Gamma$ for $U(\epsilon), V(\epsilon), \Gamma(\epsilon)$ respectively.
Denote 
$$\tilde{u}_1=\dfrac{u_1+\epsilon u_3}{1+\epsilon},\qquad
\tilde{u}_2=\max\left\{\dfrac{u_2+\epsilon u_3}{1+\epsilon},\quad \dfrac{(1-\epsilon)u_1+2\epsilon\phi}{1+\epsilon}\right\} \quad\mbox{and}\quad \tilde{\varphi}=\tilde{u}_1-\tilde{u}_2.$$
Then
$\varphi:=(u_1-u_2)=(1+\epsilon)\tilde{\varphi}$ on $U$.
Hence
\begin{align*}
	I&=\left|\int_X-\chi(\varphi)dd^c(u_1-u_3)\wedge T\right|\\
	&\leq \left|\int_U-\chi(\varphi)dd^c(u_1-u_3)\wedge T\right|
	+\left|\int_{X\setminus U}-\chi(\varphi)dd^c(u_1-u_3)\wedge T\right|\\
	&\leq \left|\int_U-\chi((1+\epsilon)\tilde{\varphi})dd^c(u_1-u_3)\wedge T\right|
	+\left|\int_{X\setminus U}-\chi(\varphi)(\theta_{u_1}+\theta_{u_3})\wedge T\right|\\
	&\leq \left|\int_U-\chi((1+\epsilon)\tilde{\varphi})dd^c(u_1-u_3)\wedge T\right|
	+\left|\int_{X\setminus U}-\chi(\epsilon(u_1+u_3-2\phi))(\theta_{u_1}+\theta_{u_3})\wedge T\right|\\
	&:=I_1+I_2,
\end{align*}
where in the last inequality we used the fact that $\chi$ is increasing and $\varphi \ge \epsilon(u_1+ u_2- 2 \phi)$ on $X \backslash U$. By Lemma \ref{le-sosanhnangnluongintegrabig}, we have
$E^0_{\tilde{\chi},\theta, \phi}\left(\dfrac{u_1+u_3}{2}\right)\lesssim B$. Therefore, it follows from Lemma
\ref{le-uocluongchiepsilon} that
\begin{equation}\label{eqI2 lemforpr3.5}
	I_2\leq 2\int_{X}-\chi\left(2\epsilon\left(\dfrac{u_1+u_3}{2}-\phi\right)\right)\theta_{(u_1+u_3)/2}\wedge T \lesssim B\varrho(1-\tilde{\chi}(-1))Q_{0}(2\epsilon).
\end{equation}
In order to estimate $I_1$, we divide it into two terms
\begin{align*}
	I_1&\leq  \left|\int_X-\chi((1+\epsilon)\tilde{\varphi})dd^c(u_1-u_3)\wedge T\right|
	+ \left|\int_{X\setminus U}-\chi((1+\epsilon)\tilde{\varphi})dd^c(u_1-u_3)\wedge T\right|\\
	&:=I_3+I_4.
\end{align*}

Note that $\tilde{u}_1-\tilde{u}_2=\epsilon(u_1+u_3-2\phi)/(1+\epsilon)$ on $X\setminus U$. Hence
$$I_4\leq \int_{X\setminus U}-\chi((1+\epsilon)\tilde{\varphi})(\theta_{u_1}+\theta_{u_3})\wedge T
\leq \int_{X\setminus U}-\chi(\epsilon(u_1+u_2-2\phi))(\theta_{u_1}+\theta_{u_3})\wedge T.$$
Using Lemma \ref{le-uocluongchiepsilon} again, we get
\begin{equation}\label{eqI4 lemforpr3.5}
	I_4\lesssim B\varrho(1-\tilde{\chi}(-1))Q_{0}(2\epsilon).
\end{equation}
Using integration by parts, we have
	$$I_3=(1+\epsilon)\left|\int_X\chi'((1+\epsilon)\tilde{\varphi})d\tilde{\varphi}\wedge
	d^c(u_1-u_3)\wedge T\right|.$$
Moreover, by Cauchy-Schwarz inequality and by the choice of $\epsilon$ (see \eqref{eq epsilon lemforpr3.5}), we get
$$\int_{\Gamma}\chi'((1+\epsilon)\tilde{\varphi})d\tilde{\varphi}\wedge
	d^c(u_1-u_3)\wedge T=0.$$
Hence
\begin{equation}\label{eqI3 lemforpr3.5}
	I_3=(1+\epsilon)\left|\int_{U\cup V}\chi'((1+\epsilon)\tilde{\varphi})d\tilde{\varphi}\wedge
	d^c(u_1-u_3)\wedge T\right|\leq (1+\epsilon)(I_5I_6)^{1/2}
\end{equation}
where
$$I_5=\int_{U\cup V}\chi'((1+\epsilon)\tilde{\varphi})d(u_1-u_3)\wedge
d^c(u_1-u_3)\wedge T,$$
and
$$I_6=\int_{U\cup V}\chi'((1+\epsilon)\tilde{\varphi})d\tilde{\varphi}\wedge
d^c\tilde{\varphi}\wedge T.$$
Since $(1+\epsilon)\tilde{\varphi}\leq \epsilon(u_1+u_3-2\phi)$, if $\chi\in\widetilde{\mathcal{W}}^-$ (hence $\chi'$ is nonnegative and increasing on $\R_{ \le 0}$) then
\begin{align*}
	I_5&\leq \int_X\chi'(\epsilon(u_1+u_3-2\phi))d(u_1-u_3)\wedge
	d^c(u_1-u_3)\wedge T\\
	& \lesssim \int_X\chi'(\epsilon(u_1+u_3-2\phi))d(u_1-\phi)\wedge d^c(u_1-\phi)\wedge T\\
	&+\int_X\chi'(\epsilon(u_1+u_3-2\phi))d(u_3-\phi)\wedge d^c(u_3-\phi)\wedge T\\
	&\leq \int_X\chi'(\epsilon(u_1-\phi))d(u_1-\phi)\wedge	d^c(u_1-\phi)\wedge T
	+\int_X\chi'(\epsilon(u_3-\phi))d (u_3-\phi)\wedge d^c(u_3-\phi)\wedge T\\
	&=\epsilon^{-1}\int_X\chi(\epsilon(u_1-\phi))dd^c(u_1-\phi)\wedge T
	+\epsilon^{-1}\int_X\chi(\epsilon(u_3-\phi))dd^c(u_3-\phi)\wedge T\\
	&\lesssim B\varrho(1-\tilde{\chi}(-1))\epsilon^{-1}Q_{0}(\epsilon),
\end{align*}
where the last estimate holds due to Lemma \ref{le-uocluongchiepsilon}.

Denote $v_1:=(u_1+2u_3)/3$ and $v_2:=(2u_1+u_3)/3$.
Since 
$$(1+\epsilon)(\tilde{u}_1-\tilde{u}_2)\geq u_1+u_3-2\phi, \quad u_1- u_3= -(v_1- v_2)/3,$$
one sees that  if $\chi\in\mathcal{W}_M^+$ (hence $\chi'$ is nonnegative and decreasing in $\R_{\le 0}$) then
\begin{align*}
	I_5&\leq \int_X\chi'((u_1+u_3-2\phi))d(u_1-u_3)\wedge
	d^c(u_1-u_3)\wedge T\\
	&\lesssim \int_X\chi'((u_1+u_3-2\phi))\big(d(v_1-\phi)\wedge
	d^c(v_1-\phi)+ d(v_2-\phi)\wedge d^c(v_2-\phi)\big)\wedge T\\
	&\leq \int_X\chi'(3(v_1-\phi))d(v_1-\phi)\wedge
	d^c(v_1-\phi)\wedge T
	+ \int_X\chi'(3(v_2-\phi))d(v_2-\phi)\wedge
	d^c(v_2-\phi)\wedge T\\
	&= \dfrac{1}{3}\int_X-\chi(3(v_1-\phi))dd^c(v_1-\phi)\wedge T
	+ \dfrac{1}{3}\int_X-\chi(3(v_2-\phi))dd^c(v_2-\phi)\wedge T\\
	&\leq \int_X-\chi(3(v_1-\phi))(\theta_{v_1}+\theta_{\phi})\wedge T
	+ \int_X-\chi(3(v_2-\phi))(\theta_{v_1}+\theta_{\phi})\wedge T\\
	&\leq 3^{M}\int_X-\chi(v_1-\phi)(\theta_{v_1}+\theta_{\phi})\wedge T
	+ 3^{M}\int_X-\chi(v_2-\phi)(\theta_{v_1}+\theta_{\phi})\wedge T\\
	&\lesssim B\varrho,
\end{align*}
where the two last estimates hold due to Lemma \ref{le-sosanhnangnluongintegrabig} and the fact
$$\log(-\chi(3t))-\log(-\chi(t))=\int_t^{3t}\dfrac{\chi'(s)}{\chi(s)}ds\le \int_t^{3t}\dfrac{M}{s}ds
=M\log 3,$$
for every $\chi\in\mathcal{W}_M^+$ and $t \le 0$. Combining the estimates in both cases, we obtain
\begin{equation}\label{eqI5 lemforpr3.5}
	I_5\lesssim B\varrho (1-\tilde{\chi}(-1))\dfrac{Q(\epsilon)^2}{\epsilon},
\end{equation}
where we used the inequality $Q(\epsilon) \ge \epsilon^{1/2}$ if $\chi \in \widetilde{W}^-$.   Now, we estimate $I_6$. Since $U, V$ are open in the plurifine topology and 
$$(1+\epsilon)\tilde{\varphi}=\begin{cases}\varphi\quad\mbox{on}\quad U\\
\epsilon(u_1+u_3-2\varphi)\quad\mbox{on}\quad V\end{cases},$$ we have
\begin{align*}
	I_6&=\int_{U}\chi'((1+\epsilon)\tilde{\varphi})d\tilde{\varphi}\wedge
	d^c\tilde{\varphi}\wedge T+\int_{V}\chi'((1+\epsilon)\tilde{\varphi})d\tilde{\varphi}\wedge
	d^c\tilde{\varphi}\wedge T\\
	&=(1+\epsilon)^{-2}\int_{U}\chi'(\varphi)d\varphi\wedge
	d^c\varphi\wedge T\\
	&+\dfrac{\epsilon^2}{(1+\epsilon)^2}\int_{V}\chi'(\epsilon(u_1+u_3-2\varphi))d(u_1+u_3-2\varphi)\wedge
	d^c(u_1+u_3-2\varphi)\wedge T\\
		&\leq J+\epsilon^2\int_{X}\chi'(\epsilon(u_1+u_3-2\varphi))d(u_1+u_3-2\varphi)\wedge
	d^c(u_1+u_3-2\varphi)\wedge T\\
		&= J+\epsilon\int_{X}-\chi(\epsilon(u_1+u_3-2\varphi))dd^c(u_1+u_3-2\varphi)\wedge T.
\end{align*}
Therefore, it follows from Lemma \ref{le-uocluongchiepsilon} that
\begin{equation}\label{eqI6 lemforpr3.5}
	I_6\lesssim J+B\varrho (1-\tilde{\chi}(-1))\epsilon Q_{0}(2\epsilon).
\end{equation}
Combining \eqref{eqI2 lemforpr3.5}, \eqref{eqI3 lemforpr3.5}, \eqref{eqI4 lemforpr3.5}, \eqref{eqI5 lemforpr3.5}
and \eqref{eqI6 lemforpr3.5}, we get
\begin{align*}
	I\leq I_1+I_2&\leq I_3+I_4+I_2\\
	&\lesssim (I_5I_6)^{1/2}+I_4+I_2\\
	&\lesssim \left( B\varrho (1-\tilde{\chi}(-1))\epsilon^{-1}J\right)^{1/2} Q(\epsilon)+ B\varrho (1-\tilde{\chi}(-1))\epsilon Q(2\epsilon)^2.
\end{align*}
Letting $\epsilon\searrow J/(2\varrho)$ (and $\epsilon$ satisfies \eqref{eq epsilon lemforpr3.5}), we obtain
 $$I\lesssim B\varrho (1-\tilde{\chi}(-1)) Q(J/\varrho).$$
 The proof is completed.
\endproof

\begin{proposition} \label{pro-mainstabilitylowenergyconvex} Let $\chi, \tilde{\chi} \in\widetilde{\mathcal{W}}^-\cup\mathcal{W}^{+}_M$	such that $\tilde{\chi}\leq\chi$ and $\chi\in \Cc^1(\R)$.  Let $u_1, u_2, u_3 \in \mathcal{E}(X, \theta, \phi)$ such that $u_1 \le u_2$ and
	$u_j-\phi$ is bounded ($j=1, 2, 3$), where $\phi$ is a negative $\theta$-psh function satisfying 
	$\varrho:=\vol(\theta_\phi)>0$.  Then there exists a constant $C_n>0$ depending only on $n$ and $M$ such that 
	\begin{align} \label{ine-chigradientfixedtypelower}
		\int_X  \chi'(u_1-u_2) d(u_1- u_2) \wedge \dc (u_1- u_2) \wedge \theta_{u_3}^{n-1}  \le C_n\varrho B^2(1-\tilde{\chi}(-1))^2 Q^{\circ (n-1)}\big(I^0_\chi(u_1,u_2)\big),
	\end{align}
	where $B:=\sum_{j=1}^3 \max\{E^0_{\tilde{\chi},\theta, \phi}(u_j),1 \}$
	and $Q$ is defined by \eqref{eq QB}. 
\end{proposition}

\begin{proof} 
 Let 
$$\varphi:= u_1- u_2, \quad T:= \sum_{j=1}^{n-1} \theta_{u_1}^j \wedge \theta_{u_2}^{n-1-j} ,$$
and
$$T_{k,l}:= \theta_{u_1}^{k} \wedge \theta_{u_2}^l   \wedge  \theta_{u_3}^{n-k-l-1}, \quad L_{k,l}:= \int_X \chi'(\varphi) d \varphi \wedge \dc \varphi \wedge T_{k,l}.$$
Observe 
$$\theta_{u_2}^n- \theta_{u_1}^n=-  \ddc \varphi \wedge T$$
 and 
\begin{align}\label{ine-Tklntru1}
L_{k,n-1-k} \le \int_X \chi'(\varphi) d \varphi \wedge \dc \varphi \wedge T = \varrho I^0_\chi(u_1, u_2)
 \end{align}
 by integration by parts. We now prove by inverse induction on $m:= k+l$ that 
 \begin{align} \label{ine-chigradientfixedtypeLkl}
L_{k,l}  \le   C_{m,n}\varrho B^2(1-\tilde{\chi}(-1))^2Q^{\circ (n-1-k-l)}\big(I^0_\chi(u_1,u_2)\big),
\end{align}
for some  constant $C_{m,n}>1$ depending only on $m, n$ and $M$.  
The desired assertion (\ref{ine-chigradientfixedtypelower}) is the case where $k=l=0$. 
In what follows  we use the symbols $\lesssim$ and $\gtrsim$ for inequalities modulo a constant depending
only on $n$ and $M$. 
We have checked  (\ref{ine-chigradientfixedtypeLkl}) for $k+l=n-1$.  Suppose that (\ref{ine-chigradientfixedtypeLkl}) holds for $k+l=m$ with $0<m\leq n-1$.
 We will verify it for $L_{k-1,l}$, where
$k+l=m$ and $k>1$. The case $L_{k,l-1}$ is done similarly.  

Denote $S_{k-1, l}=\theta_{u_1}^{k-1} \wedge \theta_{u_2}^l   \wedge  \theta_{u_3}^{n-k-l-1}$. Then
$$L_{k-1, l}-L_{k, l}=\int_X \chi'(\varphi) d \varphi \wedge \dc \varphi\wedge dd^c(u_3-u_1) \wedge S_{k-1,l}.$$
Using integration by parts, we have
\begin{align*}
	L_{k-1, l}-L_{k, l}
	&=\int_X -\chi(\varphi)dd^c(\varphi)\wedge dd^c(u_3-u_1) \wedge S_{k-1,l}\\
	&=\int_X -\chi(\varphi) dd^c(u_3-u_1) \wedge T_{k,l}-\int_X -\chi(\varphi) dd^c(u_3-u_1) \wedge T_{k-1,l+1}\\
&=\int_X\chi'(\varphi) d\varphi\wedge d^c(u_3-u_1) \wedge T_{k,l}
-\int_X\chi'(\varphi) d\varphi\wedge d^c(u_3-u_1) \wedge T_{k-1,l+1}\\
\end{align*}
Therefore, it follows from Lemma \ref{le-secondhieuu1u2} that
$$L_{k-1, l}-L_{k, l}\lesssim \varrho B(1-\tilde{\chi}(-1))\left(Q(L_{k, l}/\varrho)+Q(L_{k-1, l+1}/\varrho)\right).$$
Hence, by using the inductive hypothesis, we get
\begin{align*}
	L_{k-1, l}&\lesssim \varrho B^2(1-\tilde{\chi}(-1))^2Q^{\circ (n-1-m)}\big(I^0_\chi(u_1,u_2)\big)\\
	&+\varrho B(1-\tilde{\chi}(-1))Q\left(C_{m,n}B^2(1-\tilde{\chi}(-1))^2Q^{\circ (n-1-m)}\big(I^0_\chi(u_1,u_2)\big)\right)\\
	&\lesssim \varrho B^2(1-\tilde{\chi}(-1))^2Q^{\circ (n-m)}\big(I^0_\chi(u_1,u_2)\big).
\end{align*}
Here we use the fact $Q(t_1)\leq (t_1/t_2)^{1/2}Q(t_2)$ for every $t_1>t_2>0$ (see Lemma \ref{lem qt1t2}).

Thus, \eqref{ine-chigradientfixedtypeLkl} holds for $L_{k-1, l}$. This finishes the proof.
\end{proof}

\begin{lemma}\label{lem qt1t2}
	The function  $h(t)=\dfrac{(Q(t))^2}{t}$ is non-increasing in $\R_{>0}$.
\end{lemma}

\begin{proof}
	If $\chi\in \mathcal{W}_M^+$ then we have
			$$h(t)=
	\begin{cases}
		\dfrac{1}{t} \quad\mbox{if}\quad t\ge 1,\\
		1\quad\mbox{if}\quad 0<t<1,
	\end{cases}$$
is a non-increasing function.

	 We consider the case $\chi\in\widetilde{\mathcal{W}}^-$. We have
	 $$h(t)=
	 \begin{cases}
	 	\dfrac{1}{t} \quad\mbox{if}\quad t\ge 1,\\
	 	\dfrac{Q_0(t)}{tQ_0(1)} \quad\mbox{if}\quad 0<t<1.
	 \end{cases}$$
 It is clear that $h$ is decreasing in $[1, \infty)$.
 We need to show that $h$ is non-increasing in $(0, 1)$. Since $\chi$ is convex, we have
 $$\dfrac{\chi (t_1s)}{t_1s}\leq\dfrac{\chi(t_2s)}{t_2s},$$
 for every $0<t_2<t_1<1$ and $s<0$. Dividing both sides of the last estimate by $\tilde{\chi}(s)/s$, we get
 $$\dfrac{\chi (t_1s)}{t_1\tilde{\chi}(s)}\leq\dfrac{\chi(t_2s)}{t_2\tilde{\chi}(s)}.$$
 Taking the supremum of both sides, we obtain 
 $$\dfrac{Q_0(t_1)}{t_1}=\sup_{s\leq -1}\dfrac{\chi (t_1s)}{t_1\tilde{\chi}(s)}\leq\sup_{s\leq-1}\dfrac{\chi(t_2s)}{t_2\tilde{\chi}(s)}=
 \dfrac{Q_0(t_2)}{t_2}.$$
 Then $h(t_1)\leq h(t_2)$. Hence, $h$ is non-increasing in $(0, 1)$. The proof is completed.
\end{proof}

\begin{proof}[End of the proof of Theorem \ref{th-lowerenergy}] By Lemma \ref{le-lientuccungkididwedgedc}
	and Lemma \ref{le-regularizedchi}, the problem is reduced to the case where
	$\chi \in \Cc^1(\R)$ and $u_j, \psi_j$ are of the same singularity type as $\phi$. 
 
  Let $L$ be the left-hand side of the desired inequality.  We have
 \begin{align*}
 	L&=\int_X-\chi(u_1-u_2)(\theta_{\psi_1}^n-\theta_{u_1}^n)
 	-\int_X-\chi(u_1-u_2)(\theta_{\psi_2}^n-\theta_{u_1}^n)\\
 	&=\int_X-\chi(u_1-u_2)dd^c(\psi_1-u_1)\wedge T_1-\int_X-\chi(u_1-u_2)dd^c(\psi_2-u_1)\wedge T_2\\
 	&=L_1-L_2,
 \end{align*}
 where $T_j=\sum_{l=0}^{n-1}\theta_{\psi_j}^l\wedge\theta_{u_1}^{n-l-1}$. Using integration by parts
 and Lemma \ref{le-secondhieuu1u2}, we get
 \begin{align*}
 	L_1&=\int_X\chi'(u_1-u_2)d(u_1-u_2)\wedge d^c(\psi_1-u_1)\wedge T_1\\
 	&\leq C_1\varrho B (1-\tilde{\chi}(-1))Q\left(\varrho^{-1}\int_X\chi'(u_1-u_2)d(u_1-u_2)\wedge d^c(u_1-u_2)\wedge T_1\right),
 \end{align*}
where $C_1>0$ depends only on $n$ and $M$. Moreover, it follows from Proposition \ref{pro-mainstabilitylowenergyconvex} that
 $$\varrho^{-1}\int_X\chi'(u_1-u_2)d(u_1-u_2)\wedge d^c(u_1-u_2)\wedge T_1\leq
 C_2 B^2(1-\tilde{\chi}(-1))^2Q^{\circ (n-1)}\left(I^0_{\chi}(u_1, u_2)\right),$$
 where $C_2>1$ depends only on $n$ and $M$.
 Then
 $$L_1\leq C_3\varrho B^2 (1-\tilde{\chi}(-1))^2Q^{\circ n}\left(I^0_{\chi}(u_1, u_2)\right),$$
 where $C_3>0$ depends only on $n$ and $M$. Here we use the fact $Q(t_1)\leq (t_1/t_2)^{1/2}Q(t_2)$ for every $t_1>t_2>0$.
 
  By the same arguments, we also have
 $$-L_2\leq C_4\varrho B^2 (1-\tilde{\chi}(-1))^2Q^{\circ n}\left(I^0_{\chi}(u_1, u_2)\right),$$
 where $C_4>0$ depends only on $n$ and $M$.
 
 Hence
 $$L=L_1-L_2\leq (C_3+C_4)\varrho B^2 (1-\tilde{\chi}(-1))^2Q^{\circ n}\left(I^0_{\chi}(u_1, u_2)\right).$$
   The proof is completed.  
\end{proof}

\subsection{Proof of Theorem \ref{th-lowerenergy-kocochuanhoa}}

Recall that for every Borel set $E$ in $X$, we define
$$\capK_{\theta, \phi}(E):= \sup \bigg\{\int_E \theta_h^n: \quad h \in \PSH(X,\theta), \quad   \phi-1 \le h \le \phi  \bigg\}.$$
The following is an improvement of results from \cite{Lu-Darvas-DiNezza-logconcave,Lu-Darvas-DiNezza-mono}  (see also \cite{BEGZ,Kolodziej_2003}).

\begin{theorem}\label{the P[u]-C<u}
	Let $A\geq 1$ be a constant and let $\theta$ be a closed smooth real $(1,1)$-form such that $\theta\leq A\omega$.
	Let $\phi\in \PSH(X, \theta)$
	and $0\leq f\in L^p(X)$  for some constant $p>1$
	such that $\phi=P[\phi]$ and $0<\int_X f \omega^n=\int_X\theta_{\phi}^n:=\varrho$.
	Assume $u\in\mathcal{E}(X, \theta, \phi)$ satisfies $\sup_X(u-\phi)=0$ and $\theta_u^n=fdV.$
	Then, there exists a constant $C \ge 1$ depending only on $X, \omega, n$ and $p$ such that
	\begin{equation}
		u\geq\phi-C\, A\left(\log\|f\,\mathrm{vol}_{\omega}(X)^q/\varrho\|_{L^p}+\log A+1\right),
	\end{equation}
 where $\vol_\omega(X):= \int_X\omega^n$ and $q=\dfrac{p}{p-1}$.
\end{theorem}

By H\"older inequalities, one sees that
$$1=\int_X\dfrac{f}{\varrho}\omega^n\leq \|f/\varrho\|_{L^p}\left(\mathrm{vol}_{\omega}(X)\right)^q,$$
and then $\log\|f \vol_{\omega}(X)^q/\varrho\|_{L^p}\geq 0$.

\begin{proof}
	Without loss of generality, we can assume that $\vol_{\omega}(X)=1$.
Recall that there exists a constant $\nu>0$ depending only on $X, \omega$ such that 		
$$\int_X \exp\left(-\psi/\nu\right)\omega^n\leq C_0^2,$$
	for every $\psi\in\PSH(X, \omega)$ with $\sup_X\psi=0$, where
	$C_0>0$ is a constant depending only on $X$ and $\omega$. Consequently, one gets
	$$\int_X \exp\left(-\psi/(A\nu)\right)\omega^n\leq C_0^2,$$
	for every $\psi\in\PSH(X, \theta)\subset\PSH(X, A\omega)$ with $\sup_X\psi=0$.
	By the same arguments as in the proof of
	\cite[Proposition 4.30]{Lu-Darvas-DiNezza-mono} (use  \cite[Lemma 3.9]{Lu-Darvas-DiNezza-logconcave}
	instead of \cite[Lemma 4.9]{Lu-Darvas-DiNezza-mono}), we have 
	$$\int_E\omega^n\leq C_0\exp\left(-\dfrac{1}{2A\nu}\left(\dfrac{\capK_{\theta, \phi}(E)}{\varrho}\right)^{-1/n}\right),$$
	for every Borel set $E\subset X$. Therefore, by the H\"older inequality and the fact
	$e^{-1/t}\leq m!t^m$ for every $m \in \N$ and every $t >0$, there exists $A_0>0$ depending only on 
	$X, \omega, n$ and $p$ such that
	\begin{equation}\label{eq0.1 proof the P[u]-C<u}
		\varrho^{-1}\int_E\theta_u^n= \int_E (f/\varrho) \omega^n \leq \|f/\varrho\|_{L^p}\left(\int_E \omega^n\right)^{1/q}\leq
		 A_0A^{2n}\|f/\varrho\|_{L^p}
		\dfrac{\capK_{\theta, \phi}(E)^2}{\varrho^2},
	\end{equation}
	for every Borel set $E\subset X$, where $1/p+1/q=1$. On the other hand, denoting
	$b=(A\nu q)^{-1}$ and $B_0=(C_0)^{1/q}$, we have
	\begin{equation}\label{eq0.2 proof the P[u]-C<u}
		\varrho^{-1}\int_Xe^{-b w}\theta_u^n\leq \|f/\varrho\|_{L^p}\left(\int_Xe^{-b q w}dV\right)^{1/q}
		\leq B_0\|f/\varrho\|_{L^p},
	\end{equation}
	for every $w\in \PSH(X, \theta)$ with $\sup_X w=0$. 
	
	For every $h\in \PSH(X, \theta)$ with $\phi-1\leq h\leq \phi$, for each $0\leq t\leq 1$ and $s>0$, we have
	\begin{align*}
		t^n\int\limits_{\{u<\phi-t-s\}}\theta_h^n\leq\int\limits_{\{u<(1-t)\phi+t h-s\}}\theta_{(1-t)\phi+t h}^n
		&\leq \int\limits_{\{u<(1-t)\phi+t h-s\}}\theta_{u}^n\\
		&\leq \int\limits_{\{u<\phi-s\}}\theta_{u}^n,
	\end{align*}
	where the third estimate holds due to the comparison principle \cite[Lemma 2.3]{Lu-Darvas-DiNezza-logconcave}.
	Then
	\begin{equation}\label{eq1 proof the P[u]-C<u}
		t^n \, \capK_{\phi}(u<\phi-t-s)\leq \int\limits_{\{u<\phi-s\}}\theta_{u}^n,
	\end{equation}
	for every $0\leq t\leq 1$, $s>0$. Therefore, it follows from \eqref{eq0.1 proof the P[u]-C<u} that
	\begin{equation*} 
		t^n \, \varrho^{-1}\capK_{\phi}(u<\phi-t-s)\leq A_1 \,  \varrho^{-2}\capK_{\phi}(u<\phi-s)^{2},
	\end{equation*}
	where $A_1=A_0A^{2n}\|f/\varrho\|_{L^p}$.
	Putting $g(s)=\varrho^{-1/n}\capK_{\phi}(u<\phi-s)^{1/n}$, the above inequality becomes
	\begin{equation*} 
		tg(t+s)\leq A_1^{1/n}g(s)^2.
	\end{equation*}
	Hence, it follows from \cite[Lemma 2.4 and Remark 2.5]{EGZ} that if $g(s_0)<1/(2A_1^{1/n})$
	then $g(s)=0$ for all $s\geq s_0+2$. Moreover, by \eqref{eq1 proof the P[u]-C<u}
	and the condition \eqref{eq0.2 proof the P[u]-C<u}, we have
	$$g(s+1)^n\leq \varrho^{-1}\int\limits_{\{u<\phi-s\}}\theta_{u}^n\leq \varrho^{-1}\int\limits_{X}e^{b(\phi-u-s)}\theta_{u}^n
	\leq B_1 \, e^{-bs},$$
	for every $s>0$, where $B_1=B_0\|f/\varrho\|_{L^p}$. Then $g(s+1)<1/(2A_1^{1/n})$ provided that
	$$s>\frac{n\log 2+\log A_1}{b}+\frac{\log B_1}{b} \cdot$$ 
	Hence $g(s)=0$ for every 
	$$s\geq \frac{n\log 2+\log A_1}{b}+\frac{\log B_1}{b}+4.$$
	Thus
	\begin{center}
		$u\geq\phi -\left(\dfrac{n\log 2+\log A_1}{b}+\dfrac{\log B_1}{b}+4\right)
		= \phi-C_1\log\|f/\varrho\|_{L^p}-C_2,$
	\end{center}
	where $C_1=\frac{2}{b}=2\nu q A$ and 
	\begin{align*}
	C_2 &=4+\frac{n\log 2+\log A_0+\log B_0+2n\log A}{b}\\
	&=4+2\nu q(n\log 2+\log A_0+\log B_0+2n\log A) A.
	\end{align*}
		The proof is finished.
\end{proof}

\begin{lemma}\label{lem vol estimate}
	There exists a constant $C>0$ depending only on $n, X$ and $\omega$ such that for every $u\in\PSH (X, \omega)$ satisfying $\sup_X u=0$ and for every constant $0<t \le 1$, one has
	\begin{equation}
		\int_{\{u>-t\}} \omega^n \geq C t^{2n}.
	\end{equation}
\end{lemma}

\begin{proof}
	Let $(U_j, \varphi_j)_{j=1}^m$ such that $U_j\subset X$ are open, $\varphi_j: 4\B \longrightarrow U_j$
	are biholomorphic and $\cup_{j=1}^m\varphi_j(\B)=X$ (where $\B$ is the open unit ball in $\C^n$), and  there is a smooth psh	function $\rho_j$ in $U_j$ such that $dd^c\rho_j=\omega$ for $1 \le j \le m$. Denote
	$$C_{\rho}=\sup\limits_{1\leq j\leq m}\sup\limits_{2\B}\|\nabla(\rho_j\circ\varphi_j)\|.$$
	Assume $u(z_0)=0$. Then there exists $1\leq j_0\leq m$ such that $z_0\in\varphi_{j_0}(\B)$.
	Denote $w_0=\varphi_{j_0}^{-1}(z_0)$, $\widehat{u}(w)=u\circ \varphi_{j_0}(w)$
	and $\widehat{\rho}(w)=\rho_{j_0}\circ\varphi_{j_0}(w)-\rho_{j_0}\circ\varphi_{j_0}(w_0)$. 
	By the plurisubharmonicity of $\widehat{u}+\widehat{\rho}$, for every $t>0$ and $0<r<1$,
	we have
	\begin{align*}
		0=(\widehat{u}+\widehat{\rho})(w_0)
		&\leq\dfrac{1}{\vol_{\C^n}(r\B)}\int_{r\B}(\widehat{u}+\widehat{\rho})dV_{2n}\\
		&\leq C_{\rho}r+\dfrac{1}{c_{2n}r^{2n}}\int_{r\B}\widehat{u} dV_{2n}\\
		&\leq C_{\rho}r-\dfrac{t}{c_{2n}r^{2n}}\int_{r\B \cap\{\widehat{u}\leq-t\}} dV_{2n}\\
		&\leq C_{\rho}r-t+\dfrac{t}{c_{2n}r^{2n}}\int_{r\B \cap\{\widehat{u}>-t\}}dV_{2n}\\
		&\leq C_{\rho}r-t+\dfrac{C_{\omega}t}{r^{2n}}\mathrm{vol}_\omega(\{u>-t\}),\\
	\end{align*}
	where $c_{2n}=\mathrm{vol}_{\C^n}(\B)$ and $C_{\omega}>0$ is a constant depending only on $n, X, \omega$.
	It follows that
	$$\mathrm{vol}_\omega(\{u>-t\})
	\geq \dfrac{r^{2n}}{C_{\omega}}\left(1-\dfrac{C_{\rho}r}{t}\right).$$
	Hence, for every $0<t<1$, by choosing $r=\frac{t}{1+C_{\rho}}$, we have
	$$\mathrm{vol}_\omega(\{u>-t\})\geq C t^{2n},$$
	where $C=\dfrac{1}{C_{\omega}(1+C_{\rho})^{2n+1}}$ depends only on $n, X$ and $\omega$.
\end{proof}

\begin{proof}[End of the proof of Theorem \ref{th-lowerenergy-kocochuanhoa}]
		Without loss of generality, we can assume that 
	$u_1\leq u_2$. Denote 
	$W_t=\{u_1>a_1-t\}$ for $0<t\leq 1$. We have
	\begin{equation}\label{eq1 proof GZ12}
		\int_{W_t}-\chi(u_1-u_2)\omega^n\leq\int_{W_t}-\chi(u_1-a_2)\omega^n\leq -b_t\chi(a_1-a_2-t),
	\end{equation} 
where $	b_t:=\vol (W_t)$.

	It follows from Lemma \ref{lem vol estimate} that $W_t\neq\emptyset$. Moreover,
	\begin{equation}\label{eq1.1 proof GZ12}
		b_t:=\int_{W_t}\omega^n\geq C_1\left(\dfrac{t}{A}\right)^{2n},
	\end{equation}
where $C_1>0$ is a constant depending only on $n, X$ and $\omega$. By \cite[Theorem A]{Lu-Darvas-DiNezza-logconcave}
(see also \cite[Theorem 3]{Vu_Do-MA}), there exists a unique $\varphi\in\mathcal{E}(X, \theta, \phi)$
with $\sup_X (\varphi-\phi)=0$ such that
$$\theta_{\varphi}^n=\dfrac{\varrho}{b_t}\mathbf{1}_{W_t}\omega^n.$$
It follows from Theorem \ref{the P[u]-C<u} that
	\begin{equation}\label{eq2 proof GZ12}
	\phi-C_2A\left(-\log t+\log A+1\right)\leq\varphi\leq 	\phi,
\end{equation}
for some constant $C_2 \ge 1$ depending only on $n, X$ and $\omega$. Thus, we have
	$$E^0_{\tilde{\chi}, \theta, \phi}(\varphi)\leq -\tilde{\chi}\big(C_2\, A\left(-\log t+\log A+1\right)\big)
	\leq -C_3\left(\log\dfrac{Ae}{t}\right)^M\tilde{\chi}(-A),$$
	where $C_3>0$ depends only on $n, X, \omega$ and $M$.
	
Hence, it follows from Theorem \ref{th-lowerenergy} that
\begin{equation}\label{eq3 proof GZ12}
	\int_X-\chi(u_1-u_2)(\theta_{\psi}^n-\theta_{\varphi}^n)\leq 
	C_4\varrho \left(\log\dfrac{Ae}{t}\right)^{2M}(B-\tilde{\chi}(-A))^2(1-\tilde{\chi}(-1))^2\lambda,
\end{equation}
where $\lambda=Q^{\circ (n)}(I^0_\chi(u_1,u_2))$ and $C_4>0$ depends only on $n, X, \omega$ and $M$.

Combining \eqref{eq1 proof GZ12} and \eqref{eq3 proof GZ12}, we get
$$\int_X-\chi(u_1-u_2)\theta_{\psi}^n\leq -\varrho\chi(a_1-a_2-t)
	+C_4\varrho \left(\log\dfrac{Ae}{t}\right)^{2M}(B-\tilde{\chi}(-A))^2(1-\tilde{\chi}(-1))^2 \lambda.$$
	Letting $t\rightarrow \lambda^{m}$, we get 
	\begin{align*}
		\int_X-\chi(u_1-u_2)\theta_{\psi}^n&\leq -\varrho\chi(a_1-a_2-\lambda^m)
		+C_4\varrho \left(\log\dfrac{Ae}{\lambda^m}\right)^{2M}(B-\tilde{\chi}(-A))^2(1-\tilde{\chi}(-1))^2 \lambda\\
		&\leq -\varrho\chi(a_1-a_2-\lambda^m)+
		 C_5\varrho \dfrac{A^{(1-\gamma)/m}}{\lambda^{1-\gamma}}(B-\tilde{\chi}(-A))^2(1-\tilde{\chi}(-1))^2 \lambda\\
			&\leq  -\varrho\chi(a_1-a_2-\lambda^m)+
			 C_5\varrho A^{(1-\gamma)/m}{\lambda^{1-\gamma}}(B-\tilde{\chi}(-A))^2(1-\tilde{\chi}(-1))^2 \lambda^{\gamma},
	\end{align*}
	where $C_5>0$ depends only on $n, X, \omega, M, m$ and $\gamma$.
	
		The proof is completed.
\end{proof}

\subsection{Applications} \label{subsec-appli}

\subsubsection{Quantitative version for the domination principle}

\begin{theorem}\label{the domination}
	Let $A \ge 1$ be a constant and let $\theta\leq A\omega$ be a closed smooth real $(1,1)$-form 
	 and $\phi$ be a  model $\theta$-psh function, and $\varrho:=\vol(\theta_\phi)>0$.  Let $B \ge 1$ be a constant, $\tilde{\chi} \in \widetilde{\mathcal{W}}^-$ and $u_1, u_2\in \mathcal{E}(X, \theta, \phi)$ such that $\tilde{\chi}(-1)=-1$ and
	$$E^0_{\tilde{\chi}, \theta, \phi}(u_1)+E^0_{\tilde{\chi}, \theta, \phi}(u_2) \le B.$$
	Assume that there exists a constant $0\leq c<1$ and a Radon measure $\mu$ on $X$ satisfying
	$\theta_{u_1}^n\leq c\theta_{u_2}^n+\varrho\mu$ on $\{u_1<u_2\}$ and 
	$c_{\mu}:=\int_{\{u_1<u_2\}}d\mu\leq 1$.
	Then there exists a constant $C>0$ depending only on $n, X$ and $\omega$   such that 
	$$\capK_{\omega}\{u_1<u_2-\epsilon\}\leq 
	\dfrac{C\vol(X) (A+B)^2}{\epsilon(1-c)h^{\circ n}(1/c_{\mu})},$$
	for every $0<\epsilon<1$,
	where  $h(s)=(-\tilde{\chi}(-s))^{1/2}$ for every $0 \le s \le \infty$.
	
	In particular, if $c_{\mu}=0$ then $\capK_{\omega}\{u_1<u_2-\epsilon\}=0$ for every $\epsilon>0$,
	and then $u_1\geq u_2$ on whole $X$.
\end{theorem}

The standard domination principle corresponds to the case where $c=0$ and $\mu:=0$. A non-quantitative version of this domination principle (\emph{i.e,} for $\mu=0$) in the non-K\"ahler setting was obtained in \cite{Guedj-Lu-3}. In order to prove Theorem \ref{the domination}, we need the following result which is an immediate consequence
of the Chern-Levine-Nirenberg inequality:

\begin{proposition}\label{pro-dominatedcapacitybigomega} Let $\theta\leq A\omega$ be a closed smooth real $(1,1)$-form (where $A\geq 1$ is a constant) and $\phi$ be a model $\theta$-psh function with $\varrho:=\int_X\theta_{\phi}^n>0$.
	Let $0 \le w \le 1$ is an $\omega$-psh function and $\psi$ is the unique solution to the problem
	\begin{equation}\label{eq propfordomination}
		\begin{cases}
			u\in\mathcal{E}(X, \theta, \phi),\\
			\theta_u^n=\dfrac{\varrho}{\vol (X)}(\ddc w+ \omega)^n,\\
			\sup_X u=0.
		\end{cases}
	\end{equation}
	Then there exists a constant $C>0$ depending only on $X$ and $\omega$ such that 
	$$\int_X|\psi|\theta_{\psi}^n\leq C\, A\varrho.$$
\end{proposition}

\begin{proof}[Proof of Theorem \ref{the domination}]
	Let $w$ be an arbitrary $\omega$-psh function satisfying $0\leq w\leq 1$ and 
	$\psi$ is the unique solution
	to \eqref{eq propfordomination}. Denote $v=\max\{u_1, u_2\}$ and 
	$\chi(t)=\max\{t, -1\}\geq \tilde{\chi}(t)$. By Theorem \ref{th-lowerenergy}
	and Proposition \ref{pro-dominatedcapacitybigomega}, there exists  a constant $C_1>0$ depending only on $n, X$ an 
	$\omega$ such that
	\begin{equation}\label{eq1 proof domination}
		I_1:=\int_X  -\chi (u_1-v)  (\theta_{\psi}^n- \theta_{u_1}^n) \le 
		C_1\varrho (A+B)^2 Q^{\circ (n)}(I^0_\chi(u_1, v)),
	\end{equation}
	and
	\begin{equation}\label{eq2 proof domination}
		I_2:=\int_X  -\chi (u_1-v)  (\theta_{u_2}^n- \theta_{u_1}^n) \le 
		C_1\varrho (A+B)^2 Q^{\circ (n)}(I^0_\chi(u_1, v)).
	\end{equation}
	Moreover,  by the fact $\theta_v^n=\theta_{u_2}^n$ on $\{u_1<u_2\}$ and
	by the assumption $\theta_{u_1}^n\leq c\theta_{u_2}^n+\varrho\mu$ on $\{u_1<u_2\}$
	, we have
	\begin{equation}\label{eq3 proof domination}
		I^0_\chi(u_1, v)
		=\varrho^{-1}\int_{\{u_1<u_2\}} -\chi (u_1-v)(\theta_{u_1}^n-\theta_{u_2}^n)
		\leq\varrho^{-1}\int_{\{u_1<u_2\}} -\chi (u_1-v)(\theta_{u_1}^n-c\theta_{u_2}^n)\leq c_{\mu}.
	\end{equation}
	Combining \eqref{eq1 proof domination}, \eqref{eq2 proof domination} and \eqref{eq3 proof domination}, we get
	\begin{align*}
		(1-c)	\int_X  -\chi (u_1-v)  \theta_{\psi}^n
		&=\int_X -\chi (u_1-v)(\theta_{u_1}^n-c\theta_{u_2}^n)+(1-c)I_1+cI_2\\
		&\leq \int_X -\chi (u_1-v)(\theta_{u_1}^n-c\theta_{u_2}^n)
		+ C_1\varrho (A+B)^2(1-\tilde{\chi}(-1))^2 Q^{\circ n}(c_{\mu})\\
		&\leq \varrho c_{\mu}+ C_1\varrho (A+B)^2 Q^{\circ n}(c_{\mu})\\
		&\leq C\varrho (A+B)^2 Q^{\circ n}(c_{\mu}),
	\end{align*}
	where $C=C_1+1$. Hence
	$$\int_{\{u_1<u_2-\epsilon\}}\theta_{w}^n=\dfrac{\vol (X)}{\varrho}\int_{\{u_1<u_2-\epsilon\}}\theta_{\psi}^n
	\leq \dfrac{C\vol(X) (A+B)^2 Q^{\circ n}(c_{\mu})}{(1-c)\epsilon},$$
	for every $0<\epsilon<1$. Since $w$ is arbitrary, it follows that
	\begin{equation}\label{eq4 proof domination}
		\capK_{\omega}\{u_1<u_2-\epsilon\}\leq 
		\dfrac{C\vol(X) (A+B)^2 Q^{\circ n}(c_{\mu})}{(1-c)\epsilon}.
	\end{equation}
	Moreover, by the definition of $\chi$ and the formula of $Q$, we have
	$$Q(s)=\dfrac{1}{(-\tilde{\chi}(-1/s))^{1/2}}=\dfrac{1}{h(1/s)},$$
	for every $0<s\leq 1$, and $Q(0)=0$. Then
	\begin{equation}\label{eq5 proof domination}
		Q^{\circ n}(s)=\dfrac{1}{h^{\circ n}(1/s)},
	\end{equation}
for every $0\leq s\leq 1$.
	The proof is completed.
\end{proof}

\subsubsection{Quantitative version of Dinew's uniqueness theorem}

\begin{theorem}\label{thequantitativeuniqueness}
	Let $A \ge 1$ be a constant. Let $\theta\leq A\omega$ be a closed smooth real $(1,1)$-form and let $\phi$ be a  model $\theta$-psh
	 function such that $\varrho:=\vol(\theta_\phi)>0$. 
	Let $B \ge 1$, $\tilde{\chi} \in \mathcal{W}^-$ and $u_1, u_2\in \mathcal{E}(X, \theta, \phi)$ such that $\tilde{\chi}(-1)=-1$ and
	$$E^0_{\tilde{\chi}, \theta, \phi}(u_1)+E^0_{\tilde{\chi}, \theta, \phi}(u_2) \le B.$$
	Then, for every $0<\gamma<1$,
	there exists $C>0$ depending only on $n, X, \omega$ and $\gamma$  such that 
	$$d_{\capK}(u_1, u_2)^2\leq C\, (A+|a_1-a_2|)\left( |a_1-a_2|
	+A (A+B)^2 \mathbf{\tau}^{\gamma}\right),$$
	where $a_j:= \sup_Xu_j$,
		$\mathbf{\tau}=\dfrac{1}{h^{\circ n}(\varrho/\|\theta_{u_1}^n-\theta_{u_2}^{n}\|)}$
	and $h(s)=(-\tilde{\chi}(-s))^{1/2}$.
\end{theorem}

Note that if $\chi(t)=\max\{t, -1\}$ then $I^0_{\chi}(u_1, u_2)\leq \varrho^{-1}\|\theta_{u_1}^n-\theta_{u_2}^{n}\|$. Therefore, Theorem 
\ref{thequantitativeuniqueness} is a consequence of the following:

\begin{theorem}\label{thequantitativeuniqueness2}
		Let $\theta\leq A\omega$ be a closed smooth real $(1,1)$-form ($A\geq 1$) 
		and let $\phi$ be a  model $\theta$-psh function such that $\vol(\theta_\phi)>0$. 
	Let $B \ge 1$, $\tilde{\chi} \in \widetilde{\mathcal{W}}^-$ and $u_1, u_2\in \mathcal{E}(X, \theta, \phi)$ such that $\tilde{\chi}(-1)=-1$ and
	$$E^0_{\tilde{\chi}, \theta, \phi}(u_1)+E^0_{\tilde{\chi}, \theta, \phi}(u_2) \le B.$$
	Denote $\chi(t)=\max\{t, -1\}$.
	Then, for every $0<\gamma<1$,
	there exists $C>0$ depending only on $n, X, \omega$ and $\gamma$  such that 
	\begin{equation}\label{eq0uniqueness2}
		d_{\capK}(u_1, u_2)^2\leq C\, (A+|a_1-a_2|) \left( |a_1-a_2|
		+A(A+B)^2 \lambda^{\gamma}\right),
	\end{equation}
	where $a_j:= \sup_Xu_j$,
	$\lambda=\dfrac{1}{h^{\circ n}(1/I_{\chi}^0(u_1, u_2))}$
	and $h(s)=(-\tilde{\chi}(-s))^{1/2}$.
\end{theorem}

\begin{proof}
	Suppose that $w$ is an arbitrary $\omega$-psh function satisfying $0\leq w\leq 1$ and 
	$\psi$ is the unique solution
	to  the problem
	\begin{equation}\label{eqpsithedinewunique}
		\begin{cases}
			u\in\mathcal{E}(X, \theta, \phi),\\
			\theta_u^n=\dfrac{\varrho}{\vol (X)}(\ddc w+ \omega)^n,\\
			\sup_X u=0.
		\end{cases}
	\end{equation}
	Recall that  $\lambda=Q_{\chi, \tilde{\chi}}^{\circ n}(I^0_\chi(u_1,u_2))$,  and one has $-\tilde{\chi}(-A)\leq A$ because $\tilde{\chi}(-1)=-1$.
	It follows from Theorem \ref{th-lowerenergy-kocochuanhoa} and Proposition \ref{pro-dominatedcapacitybigomega} that, 
	for every $0<\gamma<1$,
	there exists $C_1>0$ depending only on $n, X, \omega$ and $\gamma$  such that 
	\begin{equation}\label{eq1thedcap}
		I:=\int_X  -\chi\big(-|u_1- u_2|\big) \theta_\psi^n \le -\varrho \chi\left(-|a_1-a_2|-\lambda\right)
		+C_1\varrho A(A+B)^2\lambda^{\gamma}.
	\end{equation}
	Moreover
	\begin{align*}
		\dfrac{\varrho}{\vol (X)}\int_X|u_1-u_2|^{1/2}(\omega+dd^cw)^n
		&=\int_X|u_1-u_2|^{1/2}\theta_{\psi}^n\\
		&=\int_{\{|u_1-u_2|\leq 1\}}|u_1-u_2|^{1/2}\theta_{\psi}^n+\int_{\{|u_1-u_2|> 1\}}|u_1-u_2|^{1/2}\theta_{\psi}^n
		\end{align*}
		which is less than or equal to
		$$\leq I^{1/2}\left( \left(\int_{\{|u_1-u_2|\leq 1\}}\theta_{\psi}^n\right)^{1/2}+
		\left(\int_{\{|u_1-u_2|> 1\}}|u_1-u_2|\theta_{\psi}^n\right)^{1/2}\right),$$
	where the last estimate holds due to the Cauchy-Schwarz inequality. Moreover, it follows from 
	Chern-Levine-Nirenberg inequality (\cite{Kolodziej05}) that
	\begin{align*}
		\int_{X}|u_1-a_1-u_2+a_2|\theta_{\psi}^n
		&=\dfrac{\varrho}{\vol (X)}\int_{X}|u_1-a_1-u_2+a_2|(\ddc w+ \omega)^n\\
		&\leq C_2\varrho (\|u_1-a_1\|_{L^1(X)}+\|u_2-a_2\|_{L^1(X)})\\
		&\leq \varrho C_3A,
	\end{align*}
where $C_2, C_3>0$ depend only on $X$ and $\omega$. Here, the last estimate holds due to the 
compactness of $\{u\in\PSH(X, \omega): \sup_Xu=0\}$ in $L^1(X)$.

Hence, we have
	\begin{equation}\label{eq2thedcap}
		\dfrac{\varrho}{\vol (X)}\int_X|u_1-u_2|^{1/2}(\omega+dd^cw)^n\leq C_4 I^{1/2}\varrho^{1/2}(A+|a_1-a_2|)^{1/2},
	\end{equation}
	where $C_4>0$ depends only on $X$ and $\omega$.
	
	Combining \eqref{eq1thedcap} and \eqref{eq2thedcap}, we get
	\begin{align*}
		\left(\int_X|u_1-u_2|^{1/2}(\omega+dd^cw)^n\right)^2
		&\leq C_5(A+|a_1-a_2|)\left(-\chi\left(-|a_1-a_2|-\lambda\right)
		+A(A+B)^2\lambda^{\gamma}\right)\\
		&\leq  C_5(A+|a_1-a_2|)\left(|a_1-a_2|+\lambda
		+A(A+B)^2\lambda^{\gamma}\right)\\
		&\leq C_6(A+|a_1-a_2|)\left(|a_1-a_2|
		+A(A+B)^2\lambda^{\gamma}\right),
	\end{align*}
	where $C_5, C_6>0$ depend only on $n, X, \omega$ and $\gamma$. Since $w$ is arbitrary, we obtain desired inequality.	
	The proof is completed.
\end{proof}

\begin{remark}\label{rmkAB}
	If $B\geq A$ then the inequality \eqref{eq0uniqueness2} is equivalent to
		$$d_{\capK}(u_1, u_2)^2\leq \widetilde{C}\, (A+|a_1-a_2|) \left( |a_1-a_2|
		+A\, B^2 \lambda^{\gamma}\right),$$
	where $\tilde{C}>0$ depends only on $n, X, \omega$ and $\gamma$.
\end{remark}

\subsubsection{Relation to Darvas's metrics on the space of potentials of finite energy} \label{subsec-metricenergy}

Let $\chi \in \mathcal{W}^- \cup \mathcal{W}^+_M$. Let $\theta$ be a closed smooth real $(1,1)$-form in a big cohomology class. When $\theta$ is K\"ahler, it was proved in \cite{Darvas-finite-energy,Darvas-kahlerclass,Darvas-lower-energy} that there is a natural metric $d_\chi$ on $\mathcal{E}_\chi(X, \theta)$ which makes the last space to be a complete metric space. When $\chi(t)= t$, such metrics have a long history and play an important role in the study of complex Monge-Amp\`ere equations. We refer to these last references and  \cite{Berman-Boucksom-Jonsson-KE,Berman-Darvas-Lu} for more details.  
We now draw the connection between $I_\chi(u,v)$ and the metric on $\mathcal{E}_\chi(X,\theta)$. 
Let 
$$\tilde{I}_\chi(u,v)= \int_{\{u<v\}} -\chi(u-v) (\theta_v^n + \theta_u^n)+\int_{\{u>v\}} -\chi(v-u) (\theta_u^n + \theta_v^n) \ge I_\chi(u,v).$$
By \cite{Darvas-finite-energy,Darvas-kahlerclass,Darvas-lower-energy}, there exists a constant $C>0$ such that 
$$C^{-1} \tilde{I}_\chi(u,v) \le d_\chi(u,v) \le C \tilde{I}_\chi(u,v)$$
for every $u,v \in \mathcal{E}_\chi(X,\theta)$ and $\theta$ is K\"ahler. It was proved in \cite{Gupta} (and also \cite{Darvas-finite-energy,DDL-L1metric,Lu-DiNezza-Lpmetric,Trusiani-energy,Xia-energy}) that   $\tilde{I}_\chi(u,v)$ satisfies a quasi-triangle inequality, and  the convergence in  $\tilde{I}_\chi(u,v)$  implies the convergence in capacity by using the plurisubharmonic envelope. Such a method is not quantitative.  We present below quantitative version of this fact by using our approach. 


\begin{theorem}\label{the-capmetricdarvas1}
	Let $\theta\leq A\omega$ be a closed smooth real $(1,1)$-form ($A\geq 1$ is a constant) 
	and $\phi$ be a  model $\theta$-psh function with $\varrho:=\vol(\theta_\phi)>0$.  
		Let $B\geq 1$, $\tilde{\chi} \in \mathcal{W}^-$ and $u_1, u_2\in \mathcal{E}(X, \theta, \phi)$ such that $|\sup_Xu_1-\sup_Xu_2|\leq A$, $\tilde{\chi}(-1)=-1$ and
	$$E^0_{\tilde{\chi}, \theta, \phi}(u_1)+E^0_{\tilde{\chi}, \theta, \phi}(u_2) \le B.$$
	Then there exist  $C>0$ depending only on $n, X$ and $\omega$ such that
	$$d_{\capK}(u_1, u_2)^2\leq  \dfrac{C\,A\, (A+B)^2}{h^{\circ n}(\varrho/\tilde{I}_{\tilde{\chi}}(u_1, u_2))},$$
	where $h(s)=(-\tilde{\chi}(-s))^{1/2}$ for every $0 \le s \le \infty$.
\end{theorem}

One sees from the above estimate that if $\tilde{I}_{\tilde{\chi}}(u_1, u_2)$ is small, then so is $d_\capK(u_1,u_2)$ (uniformly in $u_1,u_2 \in \mathcal{E}(X, \theta,\phi)$ of $\tilde{\chi}$-energy  bounded by a fixed constant). 

\begin{proof}
	Let $\chi(t)=\max\{t, -1\}$.
	Suppose that $w$ is an arbitrary $\omega$-psh function satisfying $0\leq w\leq 1$. By the proof of Theorem \ref{thequantitativeuniqueness2} (see \eqref{eq2thedcap}), there exists $C_1>0$ depending only on $X$ and $\omega$ such that
\begin{equation}\label{eq1darvas1}
\left(\int_X|u_1-u_2|^{1/2}(\omega+dd^cw)^n\right)^2\leq C_1A \varrho^{-1}\int_X  -\chi\big(-|u_1- u_2|\big) \theta_{\psi}^n,
\end{equation}
where $\psi$ is defined by \eqref{eqpsithedinewunique}. Moreover, it follows from Theorem \ref{th-lowerenergy}
and Proposition \ref{pro-dominatedcapacitybigomega} that
$$\int_X  -\chi\big(-|u_1- u_2|\big) \theta_{\psi}^n\leq  \tilde{I}_{\chi}(u_1, u_2)
+C_2\varrho (A+B)^2Q_{\chi, \tilde{\chi}}^{\circ (n)}(I^0_\chi(u_1,u_2)),$$
where $C_2>0$ depends only on $n$. Therefore, by the facts 	$Q^{\circ(n)}(s)=\dfrac{1}{h^{\circ(n)}(1/s)}$ and 
$I_\chi(u_1,u_2)\leq \tilde{I}_{\chi}(u_1, u_2)\leq \tilde{I}_{\tilde{\chi}}(u_1, u_2))$, we obtain
\begin{equation}\label{eq2darvas1}
\int_X  -\chi\big(-|u_1- u_2|\big) \theta_{\psi}^n\leq
\dfrac{C_3\varrho (A+B)^2}{h^{\circ(n)}(\varrho/\tilde{I}_{\tilde{\chi}}(u_1, u_2))},
\end{equation}
	where $C_3>0$ depends only on $n, X$ and $\omega$.	Combining \eqref{eq1darvas1} and \eqref{eq2darvas1}, we get 
	$$\left(\int_X|u_1-u_2|^{1/2}(\omega+dd^cw)^n\right)^2\leq \dfrac{C\, A\, (A+B)^2}{h^{\circ(n)}(\varrho/\tilde{I}_{\tilde{\chi}}(u_1, u_2))},$$
	where $C>0$ depends only on $n, X$ and $\omega$. Since $w$ is arbitrary, we get the desired inequality.
		The proof is completed.
\end{proof}

When $\tilde{\chi} \in \mathcal{W}^+_M$, our estimate is more explicit.

\begin{theorem}\label{the-capmetricdarvas2}
	Let $\theta\leq A\omega$ be a closed smooth real $(1,1)$-form ($A\geq 1$) and $\phi$ be a  model $\theta$-psh function such that $\varrho:=\vol(\theta_\phi)>0$.  Let $B\ge 1$,  $\tilde{\chi} \in \mathcal{W}_M^+$ ($M\geq 1$)  and $u_1, u_2\in \mathcal{E}(X, \theta, \phi)$ such that $\tilde{\chi}(-1)=-1$ and
	$$E^0_{\tilde{\chi}, \theta, \phi}(u_1)+E^0_{\tilde{\chi}, \theta, \phi}(u_2) \le B.$$
	Then there exists  $C>0$ depending only on $n$ and $M$ such that
	\begin{align*}
	\int_X-\tilde{\chi}(-|u_1-u_2|)\theta_{\psi}^n
		\leq C\varrho B^2\left(\tilde{I}_{\tilde{\chi}}(u_1, u_2)/\varrho\right)^{2^{-n}},
	\end{align*}
for every $\psi\in\PSH (X, \theta)$ with  $\phi-1\leq \psi\leq\phi$. Moreover, if $\sup_X u_1=\sup_X u_2$ then 
there exists $C'>0$ depending on $n, X, \omega, A$ and $M$ such that
	\begin{align*}
	\tilde{I}_{\tilde{\chi}}(u_1, u_2)
	\le 
	C'\varrho A^{1/2}B^2\left(I^0_{\tilde{\chi}}(u_1, u_2)\right)^{2^{-n-1}}.
\end{align*}
\end{theorem}

\begin{proof}
	The case $I^0_{\tilde{\chi}}(u_1, u_2)\geq 1$ is trivial. It remains to consider the case 
	$I^0_{\tilde{\chi}}(u_1, u_2)<1$.
	
Denote $v=\max\{u_1, u_2\}$.
By Lemma \ref{le-sosanhnangnluongintegrabig}, we have 
$v\in \mathcal{E}(X, \theta, \phi)$ 
and $E^0_{\tilde{\chi}, \theta, \phi}(v)\leq C_1B,$ where $C_1>0$ depends only on $n$ and $M$. 
Taking $\chi=\tilde{\chi}$ and using Theorem \ref{th-lowerenergy}, we get
\begin{equation}\label{eq1 th2darvas}
		\int_X-\tilde{\chi}(u_j-v)\theta_{\psi}^n\leq \int_X-\tilde{\chi}(u_j-v)\theta_{u_j}^n
		+C_2\varrho B^2 \big(I^0_{\tilde{\chi}}(u_j, v)\big)^{2^{-n}},
\end{equation} 
for $j=1,2$, where $C_2>0$ depends on $n$ and $M$. Note that
$$\int_X-\tilde{\chi}(u_1-v)\theta_{u_1}^n+\int_X-\tilde{\chi}(u_2-v)\theta_{u_2}^n
\leq\int_X-\tilde{\chi}(-|u_1-u_2|)(\theta_{u_1}^n+\theta_{u_2}^n)=\tilde{I}_{\tilde{\chi}}(u_1, u_2),$$
and
$$I^0_{\tilde{\chi}}(u_1, v)+I^0_{\tilde{\chi}}(u_2, v)=I^0_{\tilde{\chi}}(u_1, u_2)\leq\varrho^{-1}\tilde{I}_{\tilde{\chi}}(u_1, u_2).$$
Hence, by \eqref{eq1 th2darvas}, we get
\begin{align*}
	\int_X-\tilde{\chi}(-|u_1-u_2|)\theta_{\psi}^n
	&=\int_X-\tilde{\chi}(u_1-v)\theta_{\psi}^n+\int_X-\tilde{\chi}(u_2-v)\theta_{\psi}^n\\
	&\leq \int_X-\tilde{\chi}(u_1-v)\theta_{u_1}^n+\int_X-\tilde{\chi}(u_2-v)\theta_{u_2}^n\\
	&+C_2\varrho B^2 \left(\big(I^0_{\tilde{\chi}}(u_1, v)\big)^{2^{-n}}+
	\big(I^0_{\tilde{\chi}}(u_2, v)\big)^{2^{-n}}\right)\\
	&\leq \tilde{I}_{\tilde{\chi}}(u_1, u_2)
	+2C_2\varrho B^2 \big(\tilde{I}_{\tilde{\chi}}(u_1, u_2)/\varrho\big)^{2^{-n}}\\
	&\leq C_3\varrho B^2 \big(\tilde{I}_{\tilde{\chi}}(u_1, u_2)/\varrho\big)^{2^{-n}},
\end{align*}
where $C_3>0$ depends on $n$ and $M$. Here, the last estimate holds due to the fact
$\tilde{I}_{\tilde{\chi}}(u_1, u_2)\leq \varrho B$.

Now, we consider the case $\sup_Xu_1=\sup_Xu_2$. By Theorem \ref{th-lowerenergy-kocochuanhoa}
(choose $m=1$ and $\gamma=1/2$),
there exists $C_4>0$ depending only on $n, X, \omega$ and $M$ such that
\begin{align}\label{eq2 th2darvas}
\tilde{I}_{\tilde{\chi}}(u_1, u_2) &\leq	\int_X-\tilde{\chi}(-|u_1-u_2|)(\theta_{u_1}^n+\theta_{u_2}^n)\\
\nonumber
&\leq
	-2\varrho\,\tilde{\chi}\left(-\big(I^0_{\tilde{\chi}}(u_1, u_2)\big)^{2^{-n}}\right)
	+C_4\varrho A^{1/2} B^2 \big(I^0_{\tilde{\chi}}(u_1, u_2)\big)^{2^{-n-1}}.
\end{align}
Moreover, since $\tilde{\chi}$ is concave, we have 
$$\dfrac{\tilde{\chi}(t)}{t}\leq \dfrac{\tilde{\chi}(-1)}{-1}=1,$$
for every $-1<t<0$. Hence, by \eqref{eq2 th2darvas}, we have
\begin{align*}
	\tilde{I}_{\tilde{\chi}}(u_1, u_2)
	&\leq 2\varrho\,\big(I^0_{\tilde{\chi}}(u_1, u_2)\big)^{2^{-n}}
	+C_4\varrho A^{1/2}B^2 \big(I^0_{\tilde{\chi}}(u_1, u_2)\big)^{2^{-n-1}}\\
	&\leq (2+C_4)\varrho A^{1/2}B^2 \big(I^0_{\tilde{\chi}}(u_1, u_2)\big)^{2^{-n-1}}.
\end{align*}
The proof is completed.
\end{proof}

\subsubsection{Comparison of capacities}

For every Borel subset $E$ in $X$ and for every $\varphi \in \PSH(X,\theta)$, one denotes
$$\capK_{\theta,\varphi}(E):= \sup\big\{\int_E \theta_\psi^n: \, \psi \in \PSH(X, \theta), \quad \varphi- 1 \le \psi \le \varphi\big\}.$$
In \cite{Lu-comparison-capacity}, Lu showed that if $\varphi_j$ ($j=1, 2$) is a  $\theta_j$-psh function with
 $\int_X(\theta_j+dd^c\varphi_j)^n>0$ then there exists a continuous function $f:\R_{ \le 0}\rightarrow\R_{ \ge 0}$
 with $f(0)=0$ such that $\capK_{\theta_1,\varphi_1}(E)\leq f(\capK_{\theta_2,\varphi_2}(E))$ for every
 Borel set $E\subset X$. As an application of Theorem \ref{th-lowerenergy-kocochuanhoa}, by giving a very explicit form for $f$, we will strengthen Lu's result for the case where $\varphi_j$ is a model $\theta_j$-psh function. 
 
 First, we need the following lemma:
\begin{lemma}\label{lem CLNfortheta}
	Let $A, B>0$ be constants.
	Let $\theta$ be a closed smooth real $(1,1)$-form representing a big cohomology class such that 
	$\theta\leq A\omega$. Assume that $u, v$ are $\theta$-psh functions satisfying $v\leq u\leq v+B$.
	Then, 
	$$\int_X(-\psi)\theta_u^n\leq\int_X(-\psi)\theta_v^n+nA^nB\int_X\omega^n,$$
	for every negative $A\omega$-psh function $\psi$.
\end{lemma}

\begin{proof}
	Using approximations, we can assume that $\psi$ is smooth. 
	Denote 
	$$T=\sum_{l=0}^{n-1}\theta_u^l\wedge\theta_v^{n-l-1}.$$
	We have $\theta_u^n-\theta_v^n=dd^c(u-v)\wedge T$.
	Moreover,  using integration by parts (Theorem \ref{th-integrabypart}), we get
	$$\int_X(-\psi)dd^c(u-v)\wedge T
	=\int_X (u-v)dd^c(-\psi)\wedge T\leq A\int_X (u-v)\omega\wedge T\leq nA^nB\int_X\omega^n.$$
	Hence
	$$\int_X(-\psi)\theta_u^n\leq\int_X(-\psi)\theta_v^n+nA^nB\int_X\omega^n.$$
\end{proof}

\begin{theorem}\label{the comparisoncap} (Comparison of capacities)
Assume that $\theta_1, \theta_2\leq A\omega$ are closed smooth real $(1,1)$-forms representing big cohomology classes	and, for $j=1,2$,  $\phi_j$ is a model $\theta_j$-psh function satisfying 
	$\int_X(\theta_j+dd^c\phi_j)^n=\varrho_j>0$. Then, for every $0<\gamma<1$, there exists $C>0$
	depending only on $n, X, \omega, A$ and $\gamma$ such that
	$$\frac{\capK_{\theta_1, \phi_1}(E)}{\varrho_1}\leq 
C	\bigg(\frac{\capK_{\theta_2, \phi_2}(E)}{\varrho_2}\bigg)^{2^{-n}\gamma},$$
	for every Borel set $E\subset X$.
\end{theorem}

\begin{proof}
By the inner regularity of capacities (see \cite[Lemma 4.2]{Lu-Darvas-DiNezza-mono}), we only need consider
the case where $E$ is compact. Since the case $\capK_{\theta_2, \phi_2}(E)=\varrho_2$ is trivial, we can also
assume that $\capK_{\theta_2, \phi_2}(E)<\varrho_2$. In particular,
 by \cite[Proposition 3.7]{Lu-Darvas-DiNezza-logconcave} and \cite[Lemma 2.7]{Darvas-Lu-DiNezza-singularity-metric},
 we have
 $$\sup_Xh_{E, \theta_2, \phi_2}^*=\sup_X(h_{E, \theta_2, \phi_2}^*-\phi_2)=0,$$
  where 
 $$h_{E, \theta_2, \phi_2}=\sup\left\{w\in\PSH(X, \theta_2): w|_E\leq\phi_2-1, w\leq\phi_2\right\}.$$
Set $\chi(t)=\tilde{\chi}(t)=t$. We will use Theorem \ref{th-lowerenergy-kocochuanhoa} for $u_1=(h_{E, \theta_2, \phi_2})^*$ and $u_2=\phi_2$.
It is clear that $E^0_{\tilde{\chi}, \theta_2, \phi_2}(u_2)=0$ and $u_1=u_2-1$ on $E\setminus N$, where
$N$ is a pluripolar set. Moreover, it follows from
\cite[Proposition 3.7]{Lu-Darvas-DiNezza-logconcave} that
$$I^0_{\chi}(u_1, u_2) \le E^0_{\tilde{\chi}, \theta_2, \phi_2}(u_1)=\varrho_2^{-1}\capK_{\theta_2, \phi_2}(E)\leq 1.$$
 By Theorem \ref{th-lowerenergy-kocochuanhoa}, for every $0<\gamma<1$ and $B\geq 1$, there exists $C>0$
depending only on $X, \omega, n, A$ and $\gamma$ such that
\begin{equation}\label{eq2 comparisoncap}
	\int_E\theta_{\psi}^n\leq\int_X\-\chi(-|u_1-u_2|)\theta_{\psi}^n\leq
	 C\varrho_2A(A+B)^2\left(\capK_{\theta_2, \phi_2}(E)/\varrho_2\right)^{2^{-n}\gamma},
\end{equation}
for every compact set $E$  and for each 
$\psi\in\mathcal{E}(X, \theta_2, \phi_2)$ with $E^0_{\tilde{\chi}, \theta_2, \phi_2}(\psi)\leq B$.
Let $\varphi\in \mathcal{E}(X, \theta_1, \phi_1)$ such that $\phi_1-1\leq\varphi\leq\phi_1$
and $\int_E(\theta_1+dd^c\varphi)^n\geq\dfrac{1}{2}\capK_{\theta_1, \phi_1}(E)$.
By \cite{Lu-Darvas-DiNezza-logconcave}, there exists a unique function 
$\psi_0\in\mathcal{E}(X, \theta_2, \phi_2)$ such that $\sup_X\psi_0=0$ and
$(\theta_2+dd^c\psi_0)^n=\dfrac{\varrho_2}{\varrho_1}(\theta_1+dd^c\varphi)^n$.
When $\psi=\psi_0$, we have
\begin{equation}\label{eq3 comparisoncap}
	\int_E\theta_{\psi}^n\geq \dfrac{\varrho_2}{2\varrho_1}\capK_{\theta_1, \phi_1}(E).
\end{equation}
Moreover, by using Lemma \ref{lem CLNfortheta} for $\varphi, \phi_1$ and using the fact that 
$(\theta_2+dd^c\phi_2)^n\leq\mathbf{1}_{\{\phi_2=0\}}\theta_2^n$ (see \cite[Theorem 3.8]{Lu-Darvas-DiNezza-mono}), we have
\begin{equation}\label{eq4 comparisoncap}
	\varrho_1 E^0_{\tilde{\chi}, \theta_2, \phi_2}(\psi_0)
	= \int_X(\phi_2-\psi_0)(\theta_1+dd^c\varphi)^n
	\leq\int_X(-\psi_0)(\theta_1+dd^c\phi_1)^n+nA^n\int_X\omega^n
	\leq B,
\end{equation}
where $B\geq 1$ depends only on $A, X, \omega, n$. Combining \eqref{eq2 comparisoncap}, \eqref{eq3 comparisoncap}
and \eqref{eq4 comparisoncap}, we get
\begin{align*}
	\capK_{\theta_1, \phi_1}(E)\leq \dfrac{2\varrho_1}{\varrho_2}\int_E\theta_{\psi_0}^n&\leq
	\dfrac{2\varrho_1}{\varrho_2}\int_X\-\chi(-|u_1-u_2|)\theta_{\psi_0}^n\\
	&\leq
	2C\varrho_1A(A+B)^2\left(\capK_{\theta_2, \phi_2}(E)/\varrho_2\right)^{2^{-n}\gamma}.
\end{align*}
The proof is completed.
\end{proof}

\section{Stability estimates for varied singularity type and cohomology class}  \label{sec-variedtype}  

\subsection{Pseudo-metric on the space of singularity types} 

We first  recall some facts about the pseudo-metric on the space of singularity types. Let $\alpha$ be a big cohomology class and $\theta$ a smooth closed $(1,1)$-form in $\alpha$. Let $\mathcal{S}(\theta)$ be the space of singularity types of $\theta$-psh functions and 
$$\mathcal{S}_\delta(\theta):= \{[u] \in \mathcal{S}(\theta): \int_X\theta_u^n \ge \delta\}.$$
The pseudo-distance $d_\mathcal{S}$ on $\mathcal{S}$ was introduced in \cite{Darvas-Lu-DiNezza-singularity-metric}, and it satisfies
\begin{equation}\label{eqdS}
	d_{\mathcal{S}(\theta)}([u],[v]) \le \sum_{j=0}^n \bigg( 2 \int_X \theta_{V_\theta}^j \wedge \theta_{\max\{u,v\}}^{n-j} - \int_X \theta_{V_\theta}^j \wedge \theta_u^{n-j}-\int_X \theta_{V_\theta}^j \wedge \theta_v^{n-j} \bigg)\\
	\le C d_{\mathcal{S}(\theta)}([u],[v]),
\end{equation} 
where $C>1$ depends only on $n$. Here $V_{\theta}$ is the upper envelope of all non-positive $\theta$-psh functions:
$$V_{\theta}:=\sup\{\varphi\in\PSH(X, \theta): \varphi\leq 0\quad\mbox{on}\quad X\}.$$

For all $\theta$-psh functions $u, v$, we put
$$d_\theta(u,v):= 2\int_X\theta_{\max\{u, v\}}^n-\int_X \theta_u^n - \int_X \theta_v^n .$$
In particular, if $u\leq v$ then $d_{\theta}(u, v)= \int_X \theta_v^n-\int_X\theta_u^n$.
By \eqref{eqdS}, we have $$d_{\theta}(u, v)\leq  C d_{\mathcal{S}(\theta)}([u],[v]),$$ where $C=C(n)>0$.
 Moreover, if $\theta=A\omega$ for some $A>0$ then, we have 
 $$d_{\mathcal{S}(A\omega)}([u],[v])\leq A^nd_{(A+1)\omega}(u, v),$$
for every $u, v\in\PSH(X, A\omega)$. In the sequel, we provide more properties of $d_{\theta}$.

\begin{lemma}\label{prop-dAomegadtheta}
	Let $u_1, u_2$ be $\theta$-psh functions. Let $\theta'$ be a smooth real closed $(1,1)$-form  such that $\theta' \ge \theta$. Then
	$$d_{\theta}(u_1, u_2)\leq d_{\theta'}(u_1, u_2).$$
\end{lemma}

\begin{proof}By the fact $d_{\eta}(u_1, u_2)=d_{\eta}(u_1, \max\{u_1, u_2\})+d_{\eta}(u_2, \max\{u_1, u_2\})$ for
	$\eta=\theta, \theta'$, the problem is reduced to the case  $u_1\leq u_2$. Then we have
	$$d_{\eta}(u_1, u_2)=\int_X(\eta+dd^cu_2)^n-\int_X(\eta+dd^cu_1)^n,$$
	for $\eta=\theta, \theta'$. Moreover,
	$$(\theta'+dd^cu_j)^n-(\theta+dd^cu_j)^n
	=(\theta'-\theta)\wedge\sum\limits_{l=0}^{n-1}(\theta'+dd^cu_j)^l\wedge(\theta+dd^cu_j)^{n-l-1},$$
	for $j=1, 2$. Hence
	$$d_{\theta'}(u_1, u_2)-d_{\theta}(u_1, u_2)
	=\int_X (\theta'-\theta)\wedge T_2-\int_X (\theta'-\theta)\wedge T_1,$$
	where
	$$T_j=\sum\limits_{l=0}^{n-1}(\theta'+dd^cu_j)^l\wedge(\theta+dd^cu_j)^{n-l-1}.$$
	Thus, by the monotonicity of non-pluripolar products \cite[Theorem 1.1]{Lu-Darvas-DiNezza-mono}, we obtain
	$$d_{\theta'}(u_1, u_2)-d_{\theta}(u_1, u_2)\geq 0.$$
	The proof is completed.
\end{proof}

\begin{lemma}\label{le-sosanhwedgeTduvnho} Let $\delta>0, A>0$ be constants. Let $u,v$ be $\theta$-psh functions such that $u \le v$ and  $\int_X \theta_u^n \ge \delta$. Let $\psi$ be an $\eta$-psh function, where $\eta$ is a closed smooth $(1,1)$-form.   Assume that $\theta \le A\omega, \eta \le A\omega$. Then there exists a constant $C$ depending only on  $n,\omega$ such that 
$$\bigg| \int_X \theta_u^m \wedge \eta_\psi^{n-m} - \int_X \theta_v^m \wedge \eta_\psi^{n-m} \bigg| \le C A^n   \left(\dfrac{d_{\theta}(u, v)}{\delta}\right)^{1/n}.$$
\end{lemma}


\proof
This is essentially the proof of \cite[Proposition 4.8]{Darvas-Lu-DiNezza-singularity-metric}. Note that by monotonicity we have
$$d_{\theta}(u,v)= \int_X \theta_v^n - \int_X \theta_u^n, \quad \int_X \theta_v^m  \wedge \eta_\psi^{n-m} \ge \int_X \theta_u^m \wedge \eta_\psi^{n-m}.$$
 Without loss of generality, we can assume $d_{\theta}(u,v) \le \delta/2^{n+2}$. If $d_{\theta}(u,v)=0$, then using $u \le v$ and \cite{Lu-Darvas-DiNezza-mono}, we get $P_\theta[u]=P_\theta[v]$. In this case the left-hand side of the desired inequality is also zero. Hence from now on we assume $d_{\theta}(u,v)>0$. 

  Let $b>2$ be a constant such that $\delta/ d_{\theta}(u,v)<2b^n< 2\delta/ d_{\theta}(u,v)$. We have 
$$b^n \int_X \theta_u^n \ge (b^n-1)\int_X \theta_v^n.$$
By  this and \cite[Lemma 4.3]{Darvas-Lu-DiNezza-singularity-metric}, we obtain $w_b:= P_{\theta}(bu+ (1-b)v) \in \PSH(X,\theta)$. 
Observe 
$$b^{-1} w_b+ (1- b^{-1}) v \le b^{-1}(b u + (1-b)v)+ (1-b^{-1})v  = u.$$
Combining this with monotonicity of non-pluripolar products gives
$$\int_X \theta_u^m \wedge \eta_\psi^{n-m}  \ge \int_X \theta_{b^{-1} w_b+ (1-b^{-1}) v}^m \wedge \eta_\psi^{n-m} \ge (1-b^{-1})^m \int_X \theta_v^m \wedge \eta_\psi^{n-m}.$$
It follows that  
$$\int_X \theta_u^m \wedge \eta_\psi^{n-m} - \int_X \theta_v^m \wedge \eta_\psi^{n-m} \ge -m b^{-1} \int_X \theta_v^m \wedge \eta_\psi^{n-m} \ge - n b^{-1} A^n \int_X \omega^n$$
by monotonicity. Hence
$$\bigg| \int_X \theta_u^m \wedge \eta_\psi^{n-m} - \int_X \theta_v^m \wedge \eta_\psi^{n-m} \bigg| \le 
 C b^{-1} \le 2^{1/n}C \left(\dfrac{d_{\theta}(u, v)}{\delta}\right)^{1/n},$$
where $C:= n A^n \int_X \omega^n$. This finishes the proof.
\endproof

By Lemma \ref{le-sosanhwedgeTduvnho}, we have 

\begin{proposition}\label{pro-equivdistance}  Let $\alpha, \theta$ be as above. Then there exists a constant $C>0$ such that 
$$C^{-1} \delta \, d_{\mathcal{S}(\theta)}([u],[v])^n \le d_\theta(u,v) \le C d_{\mathcal{S}(\theta)}([u],[v])$$
for every $[u],[v] \in \mathcal{S}_\delta(\alpha)$. Moreover if $\theta'$ is a smooth real closed $(1,1)$-form and $A$ is a positive constant such that 
$$\theta' \le A \omega, \quad \theta \le A \omega,$$
 for some constant $A>0$, then there exists a constant $C_1>0$ depending only on $A,\omega$ such that
$$\delta \big(d_{\theta'}(u,v)\big)^n \le C_1 d_\theta(u,v),$$
for every $u, v \in \mathcal{S}_\delta(\alpha)$.   
\end{proposition}

\proof The first desired assertion is clear from Lemma \ref{le-sosanhwedgeTduvnho}. Also by the same lemma, one gets 
$$\delta  \big(d_{A \omega}(u,v)\big)^n \le C_1 d_\theta(u,v),$$
for every $u, v \in \mathcal{S}_\delta(\alpha)$, and some constant $C_1$ independent of $u,v,\delta$. This coupled with Lemma \ref{prop-dAomegadtheta} gives the last desired inequality. The proof is complete.
\endproof

If $\theta'$ is another closed smooth form in $\alpha$, then $\mathcal{S}_\delta(\theta)$ and $\mathcal{S}_\delta(\theta')$ are isometric under the map $u \mapsto u+\varphi$, where $\varphi$ is a  smooth function such that $\ddc \varphi= \theta'- \theta$. Hence in general in order to study singularity types in $\alpha$,  it is enough to fix a smooth form in $\alpha$.

\subsection{The case of fixed cohomology}

In this subsection, we will study the stability question when solutions are in the same cohomology class.

\subsubsection{Proof of Theorem \ref{main1}}

Let $\theta, \eta$ be closed smooth real $(1,1)$-forms representing  big cohomology classes.
For every $\chi\in\widetilde{\mathcal{W}}^{-}$ and $u\in\PSH(X, \theta)$, we denote
\begin{equation}
	\tilde{E}_{\chi, \eta, \theta}(u)=\sup\left\{ \int_X-\chi(\psi)\theta_u^n:
	 \psi\in\PSH^-(X, \eta),\, \sup_X\psi=0\right\},
\end{equation}
where we recall that $\PSH^-(X,\eta)$ is the space of negative $\eta$-psh functions on $X$.  If $\chi$ is bounded then it is clear that
$\tilde{E}_{\chi, \eta, \theta}(u)<\infty$ for every $u\in\PSH(X, \theta)$.
Moreover, it follows from \cite[Proposition 3.2]{BEGZ} that for every $u\in\PSH(X, \theta)$,
there exists $\chi\in \mathcal{W}^{-}$ such that $\tilde{E}_{\chi, \eta, \theta}(u)<\infty$. 

For every constant $B>0$ and for every $\chi\in\widetilde{\mathcal{W}}^{-}$, we define
\begin{equation}
	\tilde{\mathcal{E}}_{\chi, \eta, B}(X, \theta)=\{u\in	\PSH^-(X, \theta):
	 \tilde{E}_{\chi, \eta, \theta}(u)\leq B\}.
\end{equation}
For the convenience, in the case $\eta=\theta$, we denote $\tilde{E}_{\chi, \theta}(u):=\tilde{E}_{\chi, \theta,
	 \theta}(u)$ and $\tilde{\mathcal{E}}_{\chi,  B}(X, \theta)=\tilde{\mathcal{E}}_{\chi, \theta, B}(X, \theta)$.
 
 If $u, v\in \tilde{\mathcal{E}}_{\chi, B}(X, \theta)$ then we also denote
\begin{equation}
	I_{\chi}(u, v)=\int_{\{u<v\}}-\chi(u-v)(\theta_u^n-\theta_v^n)
	+\int_{\{v<u\}}-\chi(v-u)(\theta_v^n-\theta_u^n).
\end{equation}
In general, $I_{\chi}(u, v)$ may be negative. However, by Lemma
\ref{cor-echikichik} below (observer that there always exists $\tilde{\chi} \in \mathcal{W}^-$
such that both $\tilde{E}_{\tilde{\chi},\theta}(u),\tilde{E}_{\tilde{\chi},\theta}(v)$ are finite), if $\inf \chi=-1$  then $I_{\chi}(u, v)$ is bounded from below by $-d_{\theta}(u, v)$.

\begin{lemma}\label{lemBk}
	Let $\chi \in \widetilde{\mathcal{W}}^-$.
	Assume that $u, \phi$ are negative $\theta$-psh functions satisfying $u\leq\phi$. Denote $u_{k}=\max\{u, \phi-k\}$ for every $k>0$. Then
	$$\int_X-\chi(u_{k}-\phi)\theta_{u_{k}}^n=
	\int_X-\chi(u-\phi)\theta_{u}^n-\chi(-k)d_{\theta}(u, \phi),$$
	for every $k>0$. 
\end{lemma}

\begin{proof}
	Since $\theta_{u_{k}}^n=\theta_{u}^n$ in $\{u>\phi-k\}$ and $u_k=\phi-k$ in $\{u\leq\phi-k\}$, we have
	\begin{equation}\label{eq1 lemBk}
		\int_X-\chi(u_{k}-\phi)\theta_{u_{k}}^n
		=\int_{\{u\leq\phi-k\}}-\chi(-k)\theta_{u_{k}}^n+\int_{\{u>\phi-k\}}-\chi(u-\phi)\theta_{u}^n.
	\end{equation}
Since $\int_X\theta_{\phi}^n=\int_X\theta_{u_{k}}^n$, we have
\begin{align}\label{eq2 lemBk}
	\int_{\{u\leq\phi-k\}}-\chi(-k)\theta_{u_{k}}^n
	&=\int_X-\chi(-k)\theta_{u_{k}}^n+\int_{\{u>\phi-k\}}\chi(-k)\theta_{u_{k}}^n\\
	\nonumber
	&=\int_X-\chi(-k)\theta_{\phi}^n+\int_{\{u>\phi-k\}}\chi(-k)\theta_{u}^n.
\end{align}
Combining \eqref{eq1 lemBk} and \eqref{eq2 lemBk}, we get
\begin{align*}
	\int_X-\chi(u_{k}-\phi)\theta_{u_{k}}^n
	&=\int_X-\chi(-k)\theta_{\phi}^n+\int_{\{u>\phi-k\}}\chi(-k)\theta_{u}^n
	+\int_{\{u>\phi-k\}}-\chi(u-\phi)\theta_{u}^n\\
	&=-\chi(-k)d_{\theta}(\phi, u)+\int_{\{u\leq\phi-k\}}-\chi(-k)\theta_{u}^n
	+\int_{\{u>\phi-k\}}-\chi(u-\phi)\theta_{u}^n\\
	&=-\chi(-k)d_{\theta}(\phi, u)	+\int_{X}-\chi(u-\phi)\theta_{u}^n.
\end{align*}
The proof is completed.
\end{proof}

\begin{lemma}\label{lem-estimateechik}
Let $\chi, \tilde{\chi} \in \widetilde{\mathcal{W}}^-$
such that $\inf_{\R_{<0}}\chi=-1$.
Assume that $u_1, u_2, u_3, \phi$ are negative $\theta$-psh functions satisfying $u_1\leq u_2\leq\phi$
and $u_1\leq u_3\leq \phi$. Denote $u_{j, k}=\max\{u_j, \phi-k\}$ for every $k>1$ and $j=1,2,3$. Then
$$\int_X-\chi(u_{1,k}-u_{2,k})\theta_{u_{3,k}}^n\leq 
\int_X-\chi(u_{1}-u_{2})\theta_{u_3}^n+d_{\theta}(u_3, \phi)+\dfrac{1}{\tilde{\chi}(-k)}
\int_X\tilde{\chi}(u_1-\phi)\theta_{u_3}^n,$$
for every $k>1$. In particular, if $\sup_Xu_1=0$ and 
$u_3\in \tilde{\mathcal{E}}_{\tilde{\chi}, B}(X, \theta)$ for some $B>0$ then
$$\int_X-\chi(u_{1,k}-u_{2,k})\theta_{u_{3,k}}^n\leq 
\int_X-\chi(u_{1}-u_{2})\theta_{u_3}^n+d_{\theta}(u_3, \phi)
-\dfrac{B}{\tilde{\chi}(-k)},$$
for every $k>1$.
\end{lemma}
\begin{proof}
	Denote 
		$$ I_k:=\int_X-\chi(u_{1,k}-u_{2,k})\theta_{u_{3,k}}^n.$$
Since $\theta_{u_{3, k}}^n=\theta_{u_3}^n$ in $\{u_1>\phi-k\}\subset\{u_3>\phi-k\}$, we have
\begin{align*}
	I_k&=\int_{\{u_1>\phi-k\}}-\chi(u_{1}-u_{2})\theta_{u_3}^n
	+\int_{\{u_1\leq \phi-k\}}-\chi(u_{1,k}-u_{2,k})\theta_{u_{3,k}}^n\\
	&\leq \int_{\{u_1>\phi-k\}}-\chi(u_{1}-u_{2})\theta_{u_3}^n
	+\int_{\{u_1\leq \phi-k\}}-\chi(-k)\theta_{u_{3,k}}^n\\
	&= \int_{\{u_1>\phi-k\}}-\chi(u_{1}-u_{2})\theta_{u_3}^n
	+\int_{\{u_1\leq \phi-k\}}\theta_{u_{3,k}}^n.
\end{align*}
Then, by the fact $\int_X\theta_{\phi}^n=\int_X\theta_{u_{3, k}}^n$, we get
\begin{align*}
	I_k&\leq \int_{\{u_1>\phi-k\}}-\chi(u_{1}-u_{2})\theta_{u_3}^n
	+\int_{X}\theta_{\phi}^n-\int_{\{u_1>\phi-k\}}\theta_{u_3}^n\\
	&=\int_{\{u_1>\phi-k\}}-\chi(u_{1}-u_{2})\theta_{u_3}^n
	+d_{\theta}(u_3, \phi)+\int_{\{u_1\leq\phi-k\}}\theta_{u_3}^n\\
	&\leq \int_X-\chi(u_{1}-u_{2})\theta_{u_3}^n+d_{\theta}(u_3, \phi)+\dfrac{1}{\tilde{\chi}(-k)}
	\int_X\tilde{\chi}(u_1-\phi)\theta_{u_3}^n.
\end{align*}
The proof is completed.
\end{proof}

\begin{lemma}\label{lem-estimateichik}
	Let $\chi, \tilde{\chi} \in \widetilde{\mathcal{W}}^-$
	such that $\inf_{\R_{<0}}\chi=-1$.
	Assume that $u_1, u_2, u_3, \phi$ are negative $\theta$-psh functions satisfying $u_1\leq u_2\leq\phi$
	and $u_1\leq u_3\leq \phi$. Denote $u_{j, k}=\max\{u_j, \phi-k\}$ for every $k>1$ and $j=1,2,3$. Then
	$$\int_X-\chi(u_{1,k}-u_{2,k})\theta_{u_{3,k}}^n\geq 
	\int_X-\chi(u_{1}-u_{2})\theta_{u_3}^n
	-\dfrac{1}{\tilde{\chi}(-k)}\int_X\tilde{\chi}(u_1-\phi)\theta_{u_3}^n,$$
	for every $k>1$. In particular, if $\sup_Xu_1=0$ and 
	$u_3\in \tilde{\mathcal{E}}_{\tilde{\chi}, B}(X, \theta)$ for some $B>0$ then
	$$\int_X-\chi(u_{1,k}-u_{2,k})\theta_{u_{3,k}}^n\geq 
\int_X-\chi(u_{1}-u_{2})\theta_{u_3}^n
+\dfrac{B}{\tilde{\chi}(-k)},$$
for every $k>1$.
\end{lemma}

\begin{proof}
	Since $\theta_{u_{3, k}}^n=\theta_{u_3}^n$ in $\{u_1>\phi-k\}$, we have
	\begin{align*}
\int_X-\chi(u_{1,k}-u_{2,k})\theta_{u_{3,k}}^n
&\geq \int_{\{u_1>\phi-k\}}-\chi(u_{1,k}-u_{2,k})\theta_{u_{3,k}}^n\\
&=  \int_{\{u_1>\phi-k\}}-\chi(u_1-u_2)\theta_{u_3}^n\\
&=\int_X-\chi(u_1-u_2)\theta_{u_3}^n-\int_{\{u_1\leq\phi-k\}}-\chi(u_1-u_2)\theta_{u_3}^n\\
&\geq\int_X-\chi(u_1-u_2)\theta_{u_3}^n-\int_{\{u_1\leq\phi-k\}}\theta_{u_3}^n\\
&\geq\int_X-\chi(u_{1}-u_{2})\theta_{u_3}^n
-\dfrac{1}{\tilde{\chi}(-k)}\int_X\tilde{\chi}(u_1-\phi)\theta_{u_3}^n.
	\end{align*}
This finishes the proof.
\end{proof}

\begin{lemma}\label{cor-echikichik}
		Let $\theta$ be a closed smooth real $(1,1)$-form representing  a big cohomology class. Let $\chi, \tilde{\chi} \in \widetilde{\mathcal{W}}^-$
	such that $\inf_{\R_{<0}}\chi=-1$.
	Assume that $B>0$ and
	 $u_1, u_2\in\tilde{\mathcal{E}}_{\tilde{\chi}, B}(X, \theta)$ with $\sup_Xu_1=\sup_Xu_2=0$. Denote $\phi=P_{\theta}[\max\{u_1, u_2\}]$ and $u_{j, k}=\max\{u_j, \phi-k\}$ for every $k>1$ and $j=1,2$. Then
	 \begin{equation}
	 	E_{\tilde{\chi}, \theta, \phi}(u_{j, k})\leq B-\tilde{\chi}(-k)d_{\theta}(u_j, \phi),
	 \end{equation}
	and
	\begin{equation}\label{eq0corechikichik}
		\dfrac{4B}{\tilde{\chi}(-k)}\leq I_{\chi}(u_{1, k}, u_{2, k})-I_{\chi}(u_1, u_2)
		\leq -\dfrac{4B}{\tilde{\chi}(-k)}+d_{\theta}(u_1, u_2),
	\end{equation}
for every $k>1$.
\end{lemma}

\begin{proof}
	The first inequality is obtained directly from Lemma \ref{lemBk}. It remains to prove
	\eqref{eq0corechikichik}. For $j=1,2$ and $k>1$, we denote
	$$I_{1, j}=\int_{\{u_1<u_2\}}-\chi(u_{1, k}-u_{2,k})\theta_{u_{j,k}}^n
	+\int_{\{u_1<u_2\}}\chi(u_1-u_2)\theta_{u_j}^n,$$
	and
	$$I_{2, j}=\int_{\{u_2<u_1\}}-\chi(u_{2, k}-u_{1,k})\theta_{u_{j,k}}^n
	+\int_{\{u_1<u_2\}}\chi(u_2-u_1)\theta_{u_j}^n.$$
	We have
	\begin{equation}\label{eq1corechikichik}
		I_{\chi}(u_{1, k}, u_{2, k})-I_{\chi}(u_1, u_2)=(I_{1, 1}-I_{1, 2})+(I_{2, 2}-I_{2, 1}):=I_1+I_2.
	\end{equation}
We will estimate $I_1$ and $I_2$. By Lemmas \ref{lem-estimateechik} and \ref{lem-estimateichik}
(replace $u_2, u_3$ and $\phi$, respectively, by $\max\{u_1, u_2\}$, $u_1$ and $\max\{u_1, u_2\}$), we have
\begin{equation}\label{eq2corechikichik}
	\dfrac{B}{\tilde{\chi}(-k)}\leq
	I_{1,1}\leq d_{\theta}(u_1, \max\{u_1, u_2\})-\dfrac{B}{\tilde{\chi}(-k)}.
\end{equation}
Using Lemmas \ref{lem-estimateechik} and \ref{lem-estimateichik} again
(replace $u_2, u_3$ and $\phi$ by  $\max\{u_1, u_2\}$), we get
\begin{equation}\label{eq3corechikichik}
	\dfrac{B}{\tilde{\chi}(-k)}\leq
	\int_{\{u_1<u_2\}}-\chi(u_{1, k}-u_{2,k})\theta_{u_{2,k}}^n
	+\int_{\{u_1<u_2\}}\chi(u_1-u_2)\theta_{u_2}^n\leq -\dfrac{B}{\tilde{\chi}(-k)}.
\end{equation}
Combining \eqref{eq2corechikichik} and \eqref{eq3corechikichik}, we obtain
\begin{equation}\label{eq4corechikichik}
\dfrac{2B}{\tilde{\chi}(-k)}\leq I_1\leq -\dfrac{2B}{\tilde{\chi}(-k)}+d_{\theta}(u_1, \max\{u_1, u_2\}).
\end{equation}
Similar, we have
\begin{equation}\label{eq5corechikichik}
	\dfrac{2B}{\tilde{\chi}(-k)}\leq I_2\leq 
	-\dfrac{2B}{\tilde{\chi}(-k)}+d_{\theta}(u_2, \max\{u_1, u_2\}).
\end{equation}
Combining \eqref{eq1corechikichik}, \eqref{eq4corechikichik} and \eqref{eq5corechikichik}, we have
$$\dfrac{4B}{\tilde{\chi}(-k)}\leq I_{\chi}(u_{1, k}, u_{2, k})-I_{\chi}(u_1, u_2)
\leq -\dfrac{4B}{\tilde{\chi}(-k)}+d_{\theta}(u_1, u_2).$$
The proof is completed.
\end{proof}

The following theorem is the key step to prove the main results in the case of fixed cohomology:

\begin{theorem}\label{the-dcapfixcohomology}
	Let $\theta\leq A\omega$ be a closed smooth real $(1,1)$-form  representing  a big cohomology class  ($A\geq 1$). Let $0<\delta<1$, $B\geq A$, $\tilde{\chi} \in\widetilde{\mathcal{W}}^-$ and
	$u_1, u_2\in\tilde{\mathcal{E}}_{\tilde{\chi}, B\delta}(X, \theta)$ such that $\inf\tilde{\chi}<\tilde{\chi}(-1)=-1$, $\sup_Xu_1=\sup_Xu_2=0$
	and $\int_X\theta_{u_1}^n+\int_X\theta_{u_2}^n\geq 2\delta$. 
	Let $\epsilon>0$ be a constant such that
	$$\inf\tilde{\chi}<\dfrac{-4B\delta}{\epsilon+d_{\theta}(u_1, u_2)}.$$
	Then, for every $0<\gamma<1$,
	there exists $C>0$ depending only on $n, X, \omega$ and $\gamma$  such that 
	$$d_{\capK}(u_1, u_2)^2\leq C(A B)^2 
	\left(h^{\circ n}\left(\dfrac{\delta}{|I_{\chi}(u_1, u_2)|+\epsilon
	+d_{\theta}(u_1, u_2)}\right)\right)^{-\gamma},$$
	where  $\chi(t)=\max\{t, -1\}$ and $h(s)=(-\tilde{\chi}(-s))^{1/2}$.
\end{theorem}

\begin{proof}
	Without loss of generality, we can assume that
	$$\dfrac{4B\delta}{\epsilon+d_{\theta}(u_1, u_2)}\geq 1.$$
	 Denote $\phi=P_{\theta}[\max\{u_1, u_2\}]$ and $u_{j, k}=\max\{u_j, \phi-k\}$ for every $k>1$ and $j=1,2$.
	 By Theorem \ref{thequantitativeuniqueness2}, Remark \ref{rmkAB} and Lemma \ref{cor-echikichik}, we get
	 	$$d_{\capK}(u_{1, k}, u_{2, k})^2\leq C_1A^2\left(B-\dfrac{\tilde{\chi}(-k)d_{\theta}(u_1, u_2)}{\delta}\right)^2
	 \left(h^{\circ n}\left(\dfrac{\delta}{I_{\chi}(u_1, u_2) -\dfrac{4B\delta}{\tilde{\chi}(-k)}
	 	+d_{\theta}(u_1, u_2)}\right)\right)^{-\gamma},$$
for every $k>1$, where $C_1>0$ depends only on $n, X, \omega$ and $\gamma$.

Let $k_0>1$ such that  $\tilde{\chi}(-k_0)=\dfrac{-4B\delta}{\epsilon+d_{\theta}(u_1, u_2)}$. We have
	\begin{equation}\label{eq1the-dcapfixcohomology}
	d_{\capK}(u_{1, k_0}, u_{2, k_0})^2\leq 25 C_1(AB)^2
	\left(h^{\circ(n)}\left(\dfrac{\delta}{I_{\chi}(u_1, u_2)+\epsilon
		+2d_{\theta}(u_1, u_2)}\right)\right)^{-\gamma}.
	\end{equation}
	On the other hand, for every $\varphi\in\PSH(X, \omega)$ with $0\leq \varphi\leq 1$, we have, 
	\begin{align*}
\left(\int_X|u_j-u_{j, k_0}|^{1/2}\omega_{\varphi}^n\right)^2
&=\left(\int_{\{u_j<\phi-k_0\}}|u_j-u_{j, k_0}|^{1/2}\omega_{\varphi}^n\right)^2\\
&\leq \dfrac{1}{k_0}\left(\int_{\{u_j<\phi-k_0\}}|u_j|\omega_{\varphi}^n\right)^2\\
&\leq \dfrac{C_2A^2}{k_0},
	\end{align*}
	for $j=1,2$, where $C_2>0$ depends only on $X$ and $\omega$. The last inequality holds due to the Chern-Levine-Nirenberg inequality.

Hence, by the facts $t\geq-\tilde{\chi}(-t)$ for every $t\geq 1$ and $s\leq h(s)$ for every $0<s<1$, we get
\begin{equation}\label{eq2the-dcapfixcohomology}
d_{\capK}(u_j, u_{j, k_0})^2\leq\dfrac{C_2A^2}{k_0}\leq\dfrac{C_2A^2(\epsilon+d_{\theta}(u_1, u_2))}{4B\delta}
\leq C_2A^2\left(h^{\circ n}\left(\dfrac{\delta}{\epsilon
	+d_{\theta}(u_1, u_2)}\right)\right)^{-1}.
\end{equation}
Combining \eqref{eq1the-dcapfixcohomology} and \eqref{eq2the-dcapfixcohomology}, we obtain
$$d_{\capK}(u_1, u_2)^2\leq C_3(A B)^2 
\left(h^{\circ n}\left(\dfrac{\delta}{|I_{\chi}(u_1, u_2)|+\epsilon
	+d_{\theta}(u_1, u_2)}\right)\right)^{-\gamma},$$
where $C_3>0$ depends only on $n, X, \omega$ and $\gamma$.

The proof is completed.
\end{proof}

Now, we prove Theorem \ref{main1} for the case of fixed cohomology:

\begin{theorem}\label{cor-massnorm-fixcohomology}
		Let $\theta\leq A\omega$ be a closed smooth real $(1,1)$-form representing  a big cohomology class ($A\geq 1$).
		 Let $u\in\PSH(X, \theta)$ such that  $\sup_Xu=0$ and $\int_X\theta_{u}^n:=\delta>0$. 
	Assume	$u\in\tilde{\mathcal{E}}_{\tilde{\chi}, B\delta}(X, \theta)$, where 
	 $B\geq A$ is a given constant and $\tilde{\chi} \in \mathcal{W}^-$ with $\tilde{\chi}(-1)=-1$.
	Then, for every $0<\gamma<1$,
	there exists $C>0$ depending only on $n, X, \omega$ and $\gamma$  such that 
	$$d_{\capK}(u, v)^2\leq C (A\, B)^2 
	\left(h^{\circ n}\left(\dfrac{\delta}{\|\theta_{u}^n-\theta_{v}^n\|
		+d_{\theta}(u, v)}\right)\right)^{-\gamma},$$
	for every $v\in\PSH(X, \theta)$ with $\sup_Xv=0$, where  $h(s)=(-\tilde{\chi}(-s))^{1/2}$.
\end{theorem}

\begin{proof}
Put
$$t_0=\|\theta_{u}^n-\theta_{v}^n\|+d_{\theta}(u, v) .$$
 Denote
	$$M=\dfrac{5B\delta}{t_0}, \quad \tilde{\chi}_M(s)=\max\{\tilde{\chi}(s), -M\}\quad\mbox{and}\quad
	h_M(s)=(-\tilde{\chi}_M(-s))^{1/2}.$$
	We have $v\in\tilde{E}_{\tilde{\chi}_M, B\delta+Mt_0}(X, \theta)$ and $Mt_0=5B\delta$. Since 
	$\inf\tilde{\chi}_M=-M<\dfrac{-4B\delta}{\|\theta_{u}^n-\theta_{v}^n\|+d_{\theta}(u, v)}$,
	it follows from Theorem \ref{the-dcapfixcohomology} that
	$$d_{\capK}(u, v)^2\leq C_1(A B)^2 
	\left(h_M^{\circ n}\left(\dfrac{\delta}{|I_{\chi}(u, v)|+\|\theta_{u}^n-\theta_{v}^n\|
		+d_{\theta}(u, v)}\right)\right)^{-\gamma},$$
	where  $\chi(s)=\max\{s, -1\}$ and $C_1>0$ depends only on $n, X, \omega$ and $\gamma$.
	Since $|I_{\chi}(u, v)|\leq \|\theta_{u}^n-\theta_{v}^n\|$ and $B\geq 1$, it follows that
	$$d_{\capK}(u, v)^2\leq C_2(A\, B)^2 
	\left(h_M^{\circ n}\left(\dfrac{\delta}{\|\theta_{u}^n-\theta_{v}^n\|
		+d_{\theta}(u, v)}\right)\right)^{-\gamma},$$
	where $C_2=4C_1$. By the fact $h_M(t)=h(t)\leq M$ for every $0<t\leq M$, we obtain
		$$d_{\capK}(u, v)^2\leq C_2(A\, B)^2 
	\left(h^{\circ n}\left(\dfrac{\delta}{\|\theta_{u}^n-\theta_{v}^n\|
		+d_{\theta}(u, v)}\right)\right)^{-\gamma}.$$
	The proof is completed.
\end{proof}

\subsubsection{Proof of Theorems \ref{main3} and \ref{main2}}

In order to prove the next main results, we need several auxiliary lemmas.

\begin{lemma}\label{lem estimate int p-pe}
	Let $u: B_{R+r}:=\{z\in\C^n: |z|<R+r\}\rightarrow [-\infty, 0]$ be a measurable function such that $A:=\int_{B_{R+r}}e^{-u}dV<\infty$, where $R>r>0$ and $dV$ denotes the Lebesgue measure on $\C^n$. Assume that
	$h$ is a non-negative smooth function on $\C^n$ satisfying $\int_{\C^n}hdV=1$ and $\supp(h)\subset B_{\epsilon}$ for some $\epsilon\in (0, r)$. Then, for every $0<a<1$, there exists $C>0$ depending
	only on $R, n, a$ and $A$ such that
	$$\bigg|\int_{B_R}(u\ast h-u) dV \bigg|\leq C \epsilon^{a}.$$
\end{lemma}
\begin{proof}
	We have
	\begin{align*}
		\bigg|\int_{B_R}(u\ast h-u)dV\bigg|&
		=\bigg|\int_{B_R}\int_{B_{\epsilon}}h(w)(u(z-w)-u(z))dV_wdV_z\bigg|\\
		&=\bigg|\int_{B_{\epsilon}}h(w)\int_{B_R}(u(z-w)-u(z))dV_wdV_z\bigg|\\
		&\leq \int_{B_{\epsilon}}h(w)|\hat{u}(-w)-\hat{u}(0)|dV_w,
	\end{align*}
	where $\hat{u}(w)=\int_{\{|\xi-w|<R\}}(-u)(\xi)dV_{\xi}$. Moreover, for every $w\in B_{\epsilon}$ and
	$k>0$, we have
	\begin{multline*}
		\bigg|\int_{B_R}\max\{u(z-w), k\}dV_z-\int_{B_R}\max\{u(z), k\}dV_z\bigg|
		\\=\bigg|\int_{B_R(-w)}\max\{u(z), k\}dV_z-\int_{B_R}\max\{u(z), k\}dV_z\bigg| \leq C_nR^{2n-1}k|w|,
	\end{multline*}
	where $C_n>0$ is a constant depending only on $n$, and
	$$\int_{B_R}|u(z-w)-\max\{u(z-w), -k\}|dV_z\leq\int_{\{u<-k\}}(-u-k)dV
	\leq \int_{B_{R+r}}e^{-u-k}dV=Ae^{-k}.$$
	Then
	$$\bigg|\int_{B_R}(u\ast h-u)dV\bigg|\leq\int_{B_{\epsilon}}h(w)|\hat{u}(-w)-\hat{u}(0)|dV_w\leq C_nR^{2n-1}k\epsilon+Ae^{-k}.$$
	Choosing $k=\epsilon^{a-1}$, we get the desired inequality.
\end{proof}

\begin{lemma}  \label{le-localglobal}
	Let $(X,\omega)$ be a compact K\"ahler manifold and let $\mu$ be a Randon measure
	on $X$ such that
	$$\int_X\min\{|u_1-u_2|, 1\}d\mu\leq H(\|u_1-u_2\|_{L^1(X)}),$$
	for all $u_1, u_2\in\PSH(X, \omega)$, where $H:\R_{\geq 0}\rightarrow\R_{\geq 0}$ is an
	increasing function. Let $\Omega\subset X$ such that there exists a biholomorphic mapping
	 $\varphi: 2\B\rightarrow\Omega$, where $\B$ is the unit ball in $\C^n$. Then there exists a constant $C\geq 1$ depending only on $X$, $\omega$, $\Omega$ and $\varphi$  such that 
	$$\int_{\varphi(\B)} |(u-v)\circ \varphi^{-1}| d \mu \le  CM\, H(\|(u-v)\circ \varphi^{-1}\|_{L^1(\varphi(2\B))}). $$
	for every $M \ge 1$ and for all psh functions $u, v\in\PSH (2\B)$ with $-M\leq u,v\leq 0$. 
\end{lemma}

\proof  Let $u':= \max\{u,  M|z|^2-2M\}$. We have $u= u'$ on $\B$ and $u' =  M|z|^2-2M$ outside $\sqrt{2}\B$.
Let $\phi\in \Cc_0^{\infty}(2 \B)$ such that $0\leq\phi\leq 1$ and $\phi\equiv 1$ on $\frac{3}{2}\B$. Then 
$\tilde{u}=(\phi u')\circ\varphi^{-1}$
is  a $(C M \omega)$-psh function on $X$, where $C\geq 1$ depends only on $X, \omega, \varphi$ and $\phi$.  We do similarly for $v$ to obtain $\tilde{v}$. By the assumption, we get 
$$\int_X \dfrac{|\tilde{u}-\tilde{v}|}{8C M} d \mu \le  H\left(\dfrac{\|\tilde{u} -\tilde{v}\|_{L^1(X)}}{8C M}\right).$$  
Since $\tilde{u}=\tilde{v}$ outside $\varphi (\sqrt{2}\B)$, we have
$$\|\tilde{u} -\tilde{v}\|_{L^1(X)}=\|\tilde{u}-\tilde{v}\|_{L^1(\varphi(\sqrt{2}\B))}\leq \|(u-v)\circ \varphi^{-1}\|_{L^1(\varphi(2\B))},$$
 Thus, we obtain
$$\int_{\varphi(\B)} |(u-v)\circ \varphi^{-1}| d \mu \le
\int_X |\tilde{u}-\tilde{v}|d \mu\le 8C M\, H(\|(u-v)\circ \varphi^{-1}\|_{L^1(\varphi(2\B))}). $$
The proof is completed.
\endproof

\begin{theorem} \label{th-localglobal} Let $(X,\omega)$ be a compact K\"ahler manifold and let $\mu$ be a Randon measure
	on $X$ such that
	$$\int_X\min\{|u_1-u_2|, 1\}d\mu\leq H(\|u_1-u_2\|_{L^1(X)}),$$
	for all $u_1, u_2\in\PSH(X, \omega)$, where $H:\R_{\geq 0}\rightarrow\R_{\geq 0}$ is an increasing function.
	Assume that $\Omega$ is an open subset of $X$ and $K$ is a compact subset of $\Omega$. Then, there
	exists a constant $C>0$ depending only on $n$, $X$, $\omega$, $K$ and $\Omega$  such that 
	$$\int_{K} \min\{|u-v|, 1\} d \mu \le C M\, H\left(\|u- v\|_{L^1(\Omega)}\right)
	+\int_{K\cap \{u<-M\}}d\mu	+ \int_{K\cap \{v<-M\} }d\mu,$$
	for all negative $\omega$-psh function $u, v$ in $\Omega$ and for every $M \ge 1$.
\end{theorem}


\begin{proof}
	By using a suitable open cover of $\Omega$, we can assume that there exists a biholomorphic mapping $\varphi:2\B\rightarrow\Omega$ such that $K\subset\varphi (\B)$.  Put $u_M:=\max\{u,-M\}$ and $v_M:= \max\{v, -M\}$. 
	We have
	\begin{align*}
	\min\{|u-v|, 1\} & \leq |u_M- v_M|+\min\{|u- u_M|, 1\}+\min\{|v- v_M|, 1\}\\
	& \leq |u_M- v_M|+|u_{M+1}- u_M|
	+|v_{M+1}- v_M|.
	\end{align*}
	Then
	$$\int_{K} \min\{|u-v|, 1\} d \mu  \le \int_{K} |u_M- v_M| d \mu + \int_{K} |u_{M+1}- u_M|d \mu
	+ \int_{K} |v_{M+1}- v_M| d \mu.$$
	By Lemma \ref{le-localglobal}, we have 
	$$\int_{K} |u_M- v_M|d \mu \le C M\, H\left(\|u- v\|_{L^1(\Omega)}\right),$$
		where $C\geq 1$ 	 depends only on $n$, $X$, $\omega$ and $\varphi$.
		
		Moreover
	$$\int_{K}  |u_{M+1}- u_M| d \mu\leq \int_{K\cap \{u<-M\}}d\mu,$$
	and
	$$\int_{K} |v_{M+1}- v_M| d \mu
	\leq \int_{K\cap \{v<-M\} }d\mu,$$

	Therefore
	$$\int_{K} \min\{|u-v|, 1\} d \mu \le C M\, H\left(\|u- v\|_{L^1(\Omega)}\right)
 +\int_{K\cap \{u<-M\}}d\mu	+ \int_{K\cap \{v<-M\} }d\mu,$$
	
	This finishes the proof.
\end{proof}

The following result can be considered as a generalization of Theorems \ref{main3} and \ref{main2} for
the case of fixed cohomology:

\begin{theorem}\label{the MAtocap}
	Let $\theta$ be a closed smooth real $(1,1)$-form such that $\theta\leq A\omega$ for a given constant $A\geq 1$. Let $0<\delta \le 1$, $B\geq A$, $\tilde{\chi} \in \mathcal{W}^-$ and
$u_1, u_2\in\tilde{\mathcal{E}}_{\tilde{\chi}, B\delta}(X, \theta)$ such that $\tilde{\chi}(-1)=-1$, $\sup_Xu_1=\sup_Xu_2=0$
and $\int_X\theta_{u_1}^n+\int_X\theta_{u_2}^n\geq 2\delta$. Assume that there exists a
concave increasing function $H:\R_{\geq 0}\rightarrow\R_{\geq 0}$ such that, for $j=1, 2$,
\begin{equation}\label{eq0-the MAtocap}
	\int_X\min\{|\psi_1-\psi_2|, 1\}\theta_{u_j}^n\leq H(\|\psi_1-\psi_2\|_{L^1(X)}),
\end{equation}
for every $\psi_1, \psi_2\in\PSH(X, \omega)$.
Then, for every $0<a, b, \gamma<1$ and $m>0$,
there exists $C>0$ depending only on $n, X, \omega, H(1), a, b, \gamma$ and $m$  such that 
$$d_{\capK}(u_1, u_2)^2\leq C (A\, B)^2 
\left(h^{\circ n}\left(\dfrac{\delta}{A\, G(x)
	+d_{\theta}(u_1, u_2)}\right)\right)^{-\gamma},$$
where  
$$x:=\dist_{-1}(\theta_{u_1}^n, \theta_{u_2}^n), \quad G(x)=A^2H^a(Ax^{a(1-b)})+H(H^m(x^{1-b}))+x^{ab},$$ 
 and $h(s)=(-\tilde{\chi}(-s))^{1/2}$.
\end{theorem}


\begin{proof}
Without loss of generality, we can assume that $0<x:=\dist_{-1}(\theta_{u_1}^n, \theta_{u_2}^n)<1$.
Denote $\chi(t)=\max\{t, -1\}$. Since $\inf\tilde{\chi}=-\infty$, it follows from Theorem \ref{the-dcapfixcohomology} that
\begin{equation}\label{eq0.1theMAtocap}
d_{\capK}(u_1, u_2)^2\leq C_0 (A\, B)^2 
	\left(h^{\circ n}\left(\dfrac{\delta}{|I_{\chi}(u_1, u_2)|
		+d_{\theta}(u_1, u_2)}\right)\right)^{-\gamma},
\end{equation}
 where $C_0>0$ depends only on $n, X, \omega$ and $\gamma$.
We will estimate $|I_{\chi}(u_1, u_2)|$.
	
	For  every $k>0$ and $j=1,2$,  it follows from the Skoda integrability theorem that
	\begin{equation}\label{eq1.1 MAtocap}
		\int_{\{u_j<-k\}}\omega^n\leq\int_X\exp\left(\dfrac{-c_0(u_j+k)}{A}\right)\omega^n
		\leq C_1\exp\left(\dfrac{-c_0k}{A}\right),
	\end{equation}
	where $c_0, C_1>0$ depend only on $X$ and $\omega$. 
	For every $k>0$ and for each $0<\epsilon<1$, 	by using the standard convolution and a partition of unit,
	we can find a smooth function $u_{j, k, \epsilon}$ such that 
	$ \max\{u_j, -k\}-\epsilon\leq u_{j, k, \epsilon}\leq 0$ and
	\begin{equation}\label{eq2 MAtocap}
		\|u_{j, k, \epsilon}\|_{C^1(X)}\leq \dfrac{C_2k}{\epsilon},
	\end{equation}  
	and
	\begin{equation}\label{eq3 MAtocap}
		\|u_{j, k, \epsilon}-\max\{u_j, -k\}\|_{L^1(X)}\leq C_3A\epsilon^{a},
	\end{equation} 
	where $0<a<1$, $C_2=C_2(X, \omega)>0$ and $C_3=C_3(X, \omega, a)>0$. Here, the last inequality holds due to
	Lemma \ref{lem estimate int p-pe}. Moreover, there exists a finite open cover $\{U_l\}$ of $X$ such that
	$u_{j, k, \epsilon}$ is $(A+1)\omega$-psh on $U_l$ for every $l$. It follows from Theorem \ref{th-localglobal} that 
	\begin{align*}
	\int_X \min\{|u_j-u_{j, k, \epsilon}|, 2\}d\mu &
	\leq 	\int_X \min\{|\max\{u_j, -k-3\}-u_{j, k, \epsilon}|, 2\}d\mu\\
	& \le C_4A M\, H\left(\|\max\{u_j, -k-3\}-u_{j, k, \epsilon}\|_{L^1(X)}\right)
	 +4\int_{\{u_j, u_{j, k, \epsilon}<-M\} }d\mu\\
	& \le C_4A M\, H\left(\|\max\{u_j, -k-3\}-u_{j, k, \epsilon}\|_{L^1(X)}\right)
	 +4\int_{\{u_j<-M+1\}}d\mu,
\end{align*}
	for every $k, M\geq 1$, $j=1, 2$ and $0<\epsilon<1$, where $\mu=\theta_{u_1}^n+\theta_{u_2}^n$
	and $C_4>0$ is a constant depending only on $n$, $X$ and $\omega$.
	Therefore, by \eqref{eq1.1 MAtocap} and \eqref{eq3 MAtocap}, we have
	\begin{equation}\label{eq3.1 MAtocap}
		\int_X \min\{|u_j-u_{j, k, \epsilon}|, 2\}d\mu\leq
		C_5A M\, H\left(A\epsilon^a+\exp\left(\dfrac{-c_0k}{A}\right)\right)
		+4\int_{\{u_j<-M+1\}}d\mu,
	\end{equation}
where $C_5>0$ depends on $n, X, \omega$ and $a$. Using the fact 
$\mathbf{1}_{\{u_j<-M+1\}}\leq |\max\{u_j, -M+1\}-\max\{u_j, -M+2\}|\leq \mathbf{1}_{\{u_j<-M+2\}}$, we have
	$$\int_{\{u_j<-M+1\}}d\mu\leq\int_X|\max\{u_j, -M+1\}-\max\{u_j, -M+2\}|d\mu
	\leq 2A\, H\left(\int_{\{u_j<-M+2\}}\omega^n\right).$$
	Then, by \eqref{eq1.1 MAtocap}, we get
	\begin{equation}\label{eq3.2 MAtocap}
		\int_{\{u_j<-M+1\}}d\mu\leq C_6A\, H\left(\exp\left(\dfrac{-c_0M}{A}\right)\right),
	\end{equation}
	where $C_6>0$ depends only on $X$ and $\omega$. Combining \eqref{eq3.1 MAtocap} and \eqref{eq3.2 MAtocap},
	 we get
	\begin{equation}\label{eq3.3 MAtocap}
			\int_X \min\{|u_j-u_{j, k, \epsilon}|, 2\}d\mu\leq
		C_5A M\, H\left(A\epsilon^a+\exp\left(\dfrac{-c_0k}{A}\right)\right)
		+C_6A\, H\left(\exp\left(\dfrac{-c_0M}{A}\right)\right)
	\end{equation}
	Recall that
	\begin{align*}
		I_{\chi}(u_1, u_2)&=\int_{\{u_1<u_2\}}\min\{|u_1-u_2|, 1\}(\theta_{u_1}^n-\theta_{u_2}^n)
		+\int_{\{u_2<u_1\}}\min\{|u_1-u_2|, 1\}(\theta_{u_2}^n-\theta_{u_1}^n)\\
		&=\int_X\max\{\min\{u_2-u_1, 1\}, -1\}(\theta_{u_1}^n-\theta_{u_2}^n).
	\end{align*}
	By the fact that
	$$\max\{t_1, t_3\}-\max\{t_2, t_3\}=\min\{-t_2, -t_3\}-\min\{-t_1, -t_3\}\leq \max\{t_1-t_2, 0\},$$
	we have
	$$|\max\{\min\{u_2-u_1, 1\}, -1\}-\max\{\min\{u_{2, k, \epsilon}-u_{1, k, \epsilon}, 1\}, -1\}|
	\leq |u_2-u_1-u_{2, k, \epsilon}+u_{1, k, \epsilon}|,$$
	for every $k>0$ and $0<\epsilon<1$.Since the LHS of the last inequality is bounded by 2, it follows that
	$$|\max\{\min\{u_2-u_1, 1\}, -1\}-\max\{\min\{u_{2, k, \epsilon}-u_{1, k, \epsilon}, 1\}, -1\}|
	\leq\Psi_{k, \epsilon},$$
	where 
	$$\Psi_{k, \epsilon}=\min\{|u_2-u_1-u_{2, k, \epsilon}+u_{1, k, \epsilon}|, 2\}.$$
	 Therefore
	\begin{equation}\label{eq4 MAtocap}
		|I_{\chi}(u_1, u_2)|\leq\left|\int_X\Phi_{k, \epsilon}(\theta_{u_1}^n-\theta_{u_2}^n)\right|
		+\int_X \Psi_{k, \epsilon}(\theta_{u_1}^n+\theta_{u_2}^n),
	\end{equation}
	where $\Phi_{k, \epsilon}=\max\{\min\{u_{2, k, \epsilon}-u_{1, k, \epsilon}, 1\}, -1\}$.
	By \eqref{eq2 MAtocap}, we have
	\begin{equation}\label{eq5 MAtocap}
		\left|\int_X\Phi_{k, \epsilon}(\theta_{u_1}^n-\theta_{u_2}^n)\right|\leq \dfrac{C_2k}{\epsilon}\dist_{-1}(\theta_{u_1}^n, \theta_{u_2}^n).
	\end{equation}
 Combining \eqref{eq3.3 MAtocap}, \eqref{eq4 MAtocap},
\eqref{eq5 MAtocap}, and choosing $\epsilon=x^{1-b}$, $M=-\dfrac{Am\log (H(\epsilon)/H(1))}{c_0}$, $k=-\dfrac{A\log\epsilon}{c_0}$, we get
	\begin{equation}\label{eq7theMAtocap}
		|I_{\chi}(u_1, u_2)|\leq C_7A\left(A^2H^a(Ax^{a(1-b)})+H(H^m(x^{1-b}))+x^{ab}\right),
	\end{equation}
	where $C_7>1$ depends only on $n, X, \omega, a, b, m$ and $H(1)$. Combining \eqref{eq0.1theMAtocap} and 
	\eqref{eq7theMAtocap}, we obtain the desired inequality.
	
	The proof is completed.
\end{proof}

\subsection{Application to the space of singularity types} \label{subsec-appli-singu}

In this part we apply quantitative stability theorems in the previous subsection to deduce
some properties of the pseudometric space of singularity
types in a big cohomology class. 
\begin{proposition}\label{cormodel}
	Let $\theta\leq A\omega$ be a closed smooth real $(1,1)$-form  representing  a big cohomology class  ($A\geq 1$). Assume that $u_1$ and $u_2$ are model $\theta$-psh functions such that 
 $\int_X\theta_{u_1}^n+\int_X\theta_{u_2}^n\geq 2\delta>0$, where $\delta>0$ is a constant. 
	Then, for every $0<\gamma<1$,
	there exists $C>0$ depending only on $n, X, \omega$ and $\gamma$  such that 
	$$d_{\capK}(u_1, u_2)^2\leq C\dfrac{A^{2n+4}}{\delta^2} 
	\left(\dfrac{d_{\theta}(u_1, u_2)}{\delta}\right)^{2^{-n}\gamma}.$$
\end{proposition}

The above result implies in particular that for model potentials,  the convergence in $d_\mathcal{S}$ is stronger than that in capacity. The last non-quantitative fact follows also from \cite[Theorem 5.6]{Darvas-Lu-DiNezza-singularity-metric}.

\begin{proof}
	By \cite[Theorem 3.8]{Lu-Darvas-DiNezza-mono}, we have
	$$\theta_{u_j}^n\leq \mathbf{1}_{\{u_j=0\}}\theta^n\leq A^n\omega^n,$$
	for $j=1, 2$. Therefore, there exists $C_1>0$ depending only on $X$ and $\omega$ such that
	$$\int_X(-\psi)\theta_{u_j}^n\leq C_1A^{n+1},$$
	for every $\psi\in\PSH(X, \theta)\subset\PSH (X, A\omega)$ with $\sup_X\psi=0$. Using Theorem 
	\ref{the-dcapfixcohomology} for $\tilde{\chi}(t)=t$, we get
	\begin{equation}\label{eq1cormodel}
		d_{\capK}(u_1, u_2)^2\leq C_2\left(A\dfrac{C_1A^{n+1}}{\delta}\right)^2 
		\left(\dfrac{|I_{\chi}(u_1, u_2)|
			+d_{\theta}(u_1, u_2)}{\delta}\right)^{2^{-n}\gamma},
	\end{equation}
where $\chi(t)=\max\{t, -1\}$ and $C_2>0$ is a constant depending only on $n, X, \omega$ and $\gamma$. Since 
$\theta_{u_j}^n\leq \mathbf{1}_{\{u_j=0\}}\theta^n$, we have 
$$\int_{\{u_1<u_2\}}-\chi(u_1-u_2)\theta_{u_1}^n=\int_{\{u_2<u_1\}}-\chi(u_2-u_1)\theta_{u_2}^n=0.$$
Therefore
\begin{equation}\label{eq2cormodel}
	I_{\chi}(u_1, u_2)\leq 0.
\end{equation}
Moreover, it follows from Lemma \ref{cor-echikichik} that
\begin{equation}\label{eq3cormodel}
	I_{\chi}(u_1, u_2)\geq -d_{\theta}(u_1, u_2).
\end{equation}
Combining \eqref{eq1cormodel}, \eqref{eq2cormodel} and \eqref{eq3cormodel}, we obtain
	$$d_{\capK}(u_1, u_2)^2\leq C_3\dfrac{A^{2n+4}}{\delta^2} 
\left(\dfrac{d_{\theta}(u_1, u_2)}{\delta}\right)^{2^{-n}\gamma},$$
where $C_3>0$ depends only on $n, X, \omega$ and $\gamma$.
The proof is completed.
\end{proof}

By using Proposition \ref{cormodel}, we recover the following result  which is obtained in \cite{Darvas-Lu-DiNezza-singularity-metric} (with a different proof). 

\begin{proposition} \label{cor-DDLcomplete} Let $\delta>0$ be a constant.  Let $\mathcal{S}_{\delta}(\theta)$ be the subset of $\mathcal{S}(\theta)$ consisting of $[u] \in \mathcal{S}_\theta$ such that $\int_X \theta_u^n  \ge \delta$. Then $(\mathcal{S}_{\delta}(\theta), d_\mathcal{S})$ is a complete (pseudo)-metric space.
\end{proposition}
\begin{proof}
	Let $([u_j])_j$ be a Cauchy sequence in $\mathcal{S}_{\delta}(\theta)$ (recall $[u_j]$ denotes the singularity type of a $\theta$-psh function $u_j$ with $\sup_X u_j= 0$), \emph{i.e,} for every constant $\epsilon>0$, there exists $k_\epsilon \in \N$ such that $d_\theta(u_j, u_k) \le \epsilon$ for every $j \ge k_\epsilon,$ and $k \ge k_\epsilon$. We need to prove that there exists a class $[u_\infty] \in \mathcal{S}_{\delta}(\theta)$ so that $d_\theta(u_j, u_\infty) \to 0$ as $j \to \infty$. By using contradiction, it suffices to prove it for some subsequence of $(u_j)_j$. Hence we can assume safely that 
	$$d_{\theta}(u_j, u_{j+1})\leq 4^{-n2^{n+1}},$$
	because one can always extract a subsequence of $(u_j)_j$ with that property.
	
	Since $d_{\theta}\big(u,P_\theta[u]\big)=0$ for every $u \in \PSH(X, \theta)$, without loss of generality, we can assume that $u_j= P_\theta[u_j]$ for every $j \in \N$, in other words, $u_j$'s are  model $\theta$-psh functions.
	Consequently, by Proposition \ref{cormodel} (with $\gamma=1/2$), we get
	$$d_{\capK}(u_j, u_{j+1})\leq 2^{-n}C_1,$$
	for every $j$, where $C_1>0$ is a constant depending only on $n, X, \omega, \theta$ and $\delta$. Therefore, there
	exists a $\theta$-psh function $u_{\infty}$ such that $u_j$ converges to $u_{\infty}$ in capacity as $j\to\infty$.

	Moreover, it follows from  \cite[Theorem 3.8]{Lu-Darvas-DiNezza-mono} that
	\begin{align}\label{ine-chantrencuamodelmeasure}
		\theta_{u_j}^n \le \bold{1}_{\{u_j=0\}}\theta^n \le C_2 \omega^n,
	\end{align}
	for some constant $C_2>0$ independent of $j$. This coupled with Lemma \ref{le-hoituMAchanmeasure} below yields that 
	\begin{align}\label{limit-MAujuCKX}
		\theta_{u_j}^n \to \theta_{u_{\infty}}^n,
	\end{align}
	as $j \to \infty$. It is clear that $\int_X\theta_{u_{\infty}}^n\geq\delta$. It remains to show that
	$d_{\theta}(u_j, u_{\infty})\rightarrow 0$ as $j\to\infty$.
	
	Since
	\begin{align*}
		d_\theta(u_j, u_k)= 2 \int_X \theta_{\max\{u_j, u_k\}}^n - \int_X \theta_{u_j}^n - \int_X \theta_{u_k}^n,
	\end{align*}
	using (\ref{limit-MAujuCKX}) and the fact that $\max\{u_j, u_k\} \to \max\{u_j, u_\infty\}$ in capacity as $k \to \infty$ (for $u_k \to u_\infty$ in capacity), one gets 
	$$\liminf_{k \to \infty} d_\theta(u_j, u_k) \ge 2 \int_X \theta_{\max\{u_j, u_\infty\}}^n - \int_X \theta_{u_j}^n - \int_X \theta_{u_\infty}^n= d_\theta(u_j, u_\infty).$$
	It follows that $d_\theta(u_j, u_\infty) \to 0$ as $j \to \infty$. In other words, $[u_j] \to [u_\infty]$ in the topology induced by the pseudo-metric $d_\mathcal{S}$ (we note that $[u_\infty]$ might not be unique, but  the singularity type of its envelope $P_\theta[u_\infty]$ is unique). 
\end{proof}

The following lemma is probably known. We present a proof for readers' convenience.

\begin{lemma} \label{le-hoituMAchanmeasure} Let $\Omega$ be an open subset in $\C^n$. Let $(u_j)_j$ be a sequence of psh functions converging to a psh function $u_\infty$ in capacity in $\Omega$. Assume that the non-pluripolar product $(\ddc u_j)^n$ is well-defined for $1 \le j \le \infty$, and  there exists a non-pluripolar Radon measure $\mu$ on $\Omega$ such that $(\ddc u_j)^n \le \mu$ for every $j$. Then $(\ddc u_j)^n$ converges
	weakly to $(\ddc u_{\infty})^n$ as $j \to \infty$. 
\end{lemma}

\proof  Let $\nu$ be a limit measure of the sequence $\big((\ddc u_j)^n\big)_j$  as $j \to \infty$. We need to check that $\nu= (\ddc u_\infty)^n$. Observe that $\nu \ge (\ddc u_\infty)^n$ because $u_j \to u_\infty$ in capacity. It remains to verify the converse inequality.

Let $g \ge 0$ be a smooth function with compact support in $\Omega$. Put $u_{jk}:= \max\{u_j,-k\}$ for $1 \le j \le \infty$.  We get $u_{jk} \to u_{\infty k}$ in the capacity as $j \to \infty$. It follows that $(\ddc u_{jk})^n \to (\ddc u_{\infty k})^n$ as $j \to \infty$, and  
\begin{align*}
\limsup_{j \to \infty} \int_\Omega  g \bold{1}_{\{u_j \ge -k+1\}} (\ddc u_{jk})^n \le \int_\Omega  g  \bold{1}_{\{u_\infty \ge -k+1\}} (\ddc u_{\infty k})^n. 
\end{align*} 
By this and the equality $\bold{1}_{\{u_j > -k\}} (\ddc u_{jk})^n= (\ddc u_j)^n$, we get 
 \begin{align}\label{ine-chantrenlimsupthetauj}
\limsup_{j \to \infty} \int_\Omega  g \bold{1}_{\{u_j > -k+1\}} (\ddc u_{j})^n \le \int_\Omega  g  \bold{1}_{\{u_\infty > -k\}} (\ddc u_{\infty k})^n= \int_\Omega  g  \bold{1}_{\{u_\infty > -k\}} (\ddc u_{\infty})^n
\end{align} 
which converges to $0$ as $k \to \infty$.  On the other hand, by hypothesis, one gets
\begin{align*}
\bold{1}_{\{u_{j} < -k\}} (\ddc u_{j})^n &\le \bold{1}_{\{u_{j} < -k\}} \mu\le \bold{1}_{\{u < -k\}} \mu+ \bold{1}_{\{u_{j}- u \le 1\}} \mu.
\end{align*}
Combining this with \cite[Lemma 4.5]{GZ-weighted} (we use here the fact that $\mu$ is non-pluripolar) gives
$$\limsup_{j \to \infty}\int_\Omega  g  \bold{1}_{\{u_{j} < -k\}} (\ddc u_{j})^n  \le \int_\Omega  g  \bold{1}_{\{u < -k\}} \mu$$
 Letting $k \to \infty$, we obtain
$$\limsup_{j \to \infty}\int_\Omega  g \bold{1}_{\{u_{j} < -k\}} (\ddc u_{j})^n  \to 0$$
as $k \to \infty$. Combining the last inequality with (\ref{ine-chantrenlimsupthetauj}) yields
$$ \limsup_{j \to \infty}\int_\Omega g (\ddc u_{j})^n \le \int_\Omega g (\ddc u_\infty)^n.$$
Hence $(\ddc u_{j})^n \to \ddc u_\infty$ as $j \to \infty$. This finishes the proof.
\endproof

\subsection{The case of varied cohomology} \label{subsec-stability-estimates-varied}

We now present the proof of Theorem \ref{main1}.

\begin{proof}[End of the proof of Theorem \ref{main1}]
		Put $\epsilon=\|\theta-\eta\|_{\Cc^0}$. Then, there exists $C_1\geq 1$ depending only on $X$ and $\omega$
	such that
	\begin{equation}\label{eq1mainvaried}
		\theta\leq\eta+C_1\epsilon\omega\leq\theta+2C_1\epsilon\omega.
	\end{equation}
Note that, by Chern-Levine-Nirenberg inequality
and by the compactness of $\{w\in\PSH(X, \omega): \sup_Xw=0\}$ in $L^1(X)$,
 there exists $C_{\omega}>0$ depending only on $X$ and $\omega$ such that
$$d_{\capK}(u, v)^2\leq C_{\omega}A.$$
Therefore, if $0<\dfrac{\delta}{2C_1}\leq\epsilon$ then the desired inequality \eqref{eq0main1} holds.
	Hence, without loss of generality, we can assume that
	$$\epsilon<\dfrac{\delta}{2C_1},$$
	and, as a consequence, we have $\tilde{\theta}:=\theta+C_1\epsilon\omega\leq (A+1)\omega$.
	
	 It follows from
	 \cite[Theorem 4.7]{Lu-Darvas-DiNezza-logconcave} that there exists a unique
	  $\tilde{u}\in\mathcal{E}(X, \tilde{\theta}, P_{\tilde{\theta}}[u])$ such that
	  \begin{equation}
	  	\begin{cases}
	  		\tilde{\theta}_{\tilde{u}}^n=c\theta_u^n,\\
	  		\sup_X\tilde{u}=0,
	  	\end{cases}
	  \end{equation}
  where $c=\dfrac{\int_X\tilde{\theta}_u^n}{\int_X\theta_u^n}\geq 1$. Observe that 
  $\int_X\tilde{\theta}_{\tilde{u}}^n\geq c\delta$ and 
  $\tilde{u}\in\tilde{E}_{\tilde{\chi}, Bc\delta}(X, \tilde{\theta})$. It follows from Theorem
  \ref{cor-massnorm-fixcohomology} that
  \begin{equation}\label{eq2mainvaried}
  	d_{\capK}(\tilde{u}, u)^2\leq C_2(A+1)^2B^2 
  	\left(h^{\circ n}\left(\dfrac{\delta}{\|\tilde{\theta}_{\tilde{u}}^n-\tilde{\theta}_{u}^n\|
  		+d_{\tilde{\theta}}(\tilde{u}, u)}\right)\right)^{-\gamma},
  \end{equation} 
and
\begin{equation}\label{eq3mainvaried}
	d_{\capK}(\tilde{u}, v)^2\leq C_2(A+1)^2B^2 
	\left(h^{\circ n}\left(\dfrac{\delta}{\|\tilde{\theta}_{\tilde{u}}^n-\tilde{\theta}_{v}^n\|
		+d_{\tilde{\theta}}(\tilde{u}, v)}\right)\right)^{-\gamma},
\end{equation}
where $C_2>0$ depends only on $n, X, \omega$ and $\gamma$. Since 
$P_{\tilde{\theta}}[u]=P_{\tilde{\theta}}[\tilde{u}]$, we have
\begin{equation}\label{eq4mainvaried}
	d_{\tilde{\theta}}(\tilde{u}, u)=0\quad\mbox{and}\quad
	d_{\tilde{\theta}}(\tilde{u}, v)=d_{\tilde{\theta}}(u, v).
\end{equation}
Combining \eqref{eq2mainvaried}, \eqref{eq3mainvaried} and \eqref{eq4mainvaried}, we get
\begin{equation}\label{eq5mainvaried}
		d_{\capK}(u, v)^2\leq C_3(A\, B)^2
	\left(h^{\circ n}\left(\dfrac{\delta}{\|\tilde{\theta}_{\tilde{u}}^n-\tilde{\theta}_{u}^n\|
		+\|\tilde{\theta}_{\tilde{u}}^n-\tilde{\theta}_{v}^n\|
		+d_{\tilde{\theta}}(u, v)}\right)\right)^{-\gamma},
\end{equation}
where $C_3>0$ depends only on $n, X, \omega$ and $\gamma$.
By \eqref{eq1mainvaried}, we have
	$$\theta_u^n\leq \tilde{\theta}_{u}^n\leq \theta_u^n+C_4\epsilon(\theta_u+\omega)^n,$$
	and
$$\eta_v^n\leq \tilde{\theta}_{v}^n\leq (\eta_v+2C_1\epsilon\omega)^n\leq
\eta_v^n+C_4\epsilon(\eta_v+\omega)^n,$$
where $C_4>0$ depends only on $X$ and $\omega$.
	Therefore
	\begin{equation}\label{eq6mainvaried}
		\|\theta_u^n-\tilde{\theta}_{u}^n\|+\|\eta_v^n-\tilde{\theta}_{v}^n\|
		\leq C_5(A+1)^n\vol (X)\epsilon,
	\end{equation}
where $C_5>0$ depends only on $X$ and $\omega$. Moreover,
\begin{equation}\label{eq7mainvaried}
	\|\theta_u^n-\tilde{\theta}_{\tilde{u}}^n\|=(c-1)\int_X\theta_u^n=\int_X(\tilde{\theta}_{u}^n-\theta_u^n)
	\leq \|\theta_u^n-\tilde{\theta}_{u}^n\|.
\end{equation}
Combining \eqref{eq6mainvaried} and \eqref{eq7mainvaried}, we get
\begin{align*}
	\|\tilde{\theta}_{\tilde{u}}^n-\tilde{\theta}_{u}^n\|+\|\tilde{\theta}_{\tilde{u}}^n-\tilde{\theta}_{v}^n\|
	&\leq \|\theta_u^n-\tilde{\theta}_{u}^n\|+2\|\theta_u^n-\tilde{\theta}_{\tilde{u}}^n\|
+\|\theta_u^n-\eta_v^n\|+\|\eta_v^n-\tilde{\theta}_{v}^n\|\\
	&\leq 3\|\theta_u^n-\tilde{\theta}_{u}^n\|+\|\eta_v^n-\tilde{\theta}_{v}^n\|+\|\theta_u^n-\eta_v^n\|\\
	&\leq 3C_5(A+1)^n\vol (X)\epsilon+\|\theta_u^n-\eta_v^n\|.
\end{align*}
Hence, by \eqref{eq5mainvaried}, we obtain
\begin{align*}
		d_{\capK}(u, v)^2&\leq C_3(A\, B)^2
	\left(h^{\circ n}\left(\dfrac{\delta}{3C_5(A+1)^n\vol (X)\epsilon+\|\theta_u^n-\eta_v^n\|
		+d_{\tilde{\theta}}(u, v)}\right)\right)^{-\gamma}\\
	&\leq C_6(A\, B)^2
	\left(h^{\circ n}\left(\dfrac{\delta}{A^n\epsilon+\|\theta_u^n-\eta_v^n\|
		+d_{(A+1)\omega}(u, v)}\right)\right)^{-\gamma},
\end{align*}
where $C_6>0$ depends only on $n, X, \omega$ and $\gamma$. Here we use the facts 
$d_{\tilde{\theta}}(u, v)\leq d_{(A+1)\omega}(u, v)$ (see Lemma \ref{prop-dAomegadtheta})
and $h(t)\leq h(Mt)\leq M\, h(t)$ for every $M\geq 1$ and $t>0$.
The proof is completed.
\end{proof}

We give here a consequence of Theorem \ref{main1}.

\begin{theorem}\label{the-hoitucapacitymassnorm} Let $(\theta_j)_{j\in \N \cup \{\infty\}}$ be a sequence of closed smooth real $(1,1)$-forms in $X$ such that $\theta_j \to \theta_\infty$ in $\cali{C}^0$ topology as $j \to\infty$. Let $\phi_j$ be a model $\theta_j$-psh function for $j \in \N \cup\{\infty\}$ such that 
$$d_{\mathcal{S}}(\phi_j, \phi_\infty)\rightarrow 0$$
as $j\to\infty$. Let $\mu_j$  be a non-pluripolar measure on $X$ such that 
$$\mu_j(X)= \int_X (\ddc \phi_j+ \theta_j)^n$$
for every $j$ and  $\mu_j \to \mu_\infty$ in the mass norm.   Let $u_j$ be the $\theta$-psh function satisfying  
$$(\ddc u_j+ \theta_j)^n=\mu_j, \quad \sup_X (u_j-\phi_j) =0$$ for $j \in \N\cup \{\infty\}$.  Then  $u_j \to u_\infty$ in capacity as $j \to \infty$. 
\end{theorem}

Here by $d_S(\phi_j, \phi_\infty)$ we mean the pseudodistance $d_{\mathcal{S}(A \omega)}$ between the $(A\omega)$-singularity types of $\phi_j$ and $\phi_\infty$, where $A>0$ is a big enough constant such that $\theta_j \le A\omega$ for every $j$.  The property $d_{\mathcal{S}}(\phi_j, \phi_\infty)\to 0$ is independent of the choice of $A$. Moreover, as mentioned above  when $\theta_j$ is equal to a fixed $\theta$, the pseudometric $d_{\mathcal{S}(A \omega)}$ is comparable with $d_{\mathcal{S}(\theta)}$. 

\begin{proof}
Since $\theta_j \to \theta_\infty$ in $\Cc^0$-norm, there exists a constant $A \ge 1$ so that $\theta_j \le A \omega$ for every $j \in \N \cup \{\infty\}$.   By \cite[Proposition 3.2]{BEGZ}, there exists $\tilde{\chi} \in \mathcal{W}^-$ such that 
$$\sup_{\psi \in \PSH(X, (A+1)\omega): \sup_X \psi =0}\int_X - \tilde{\chi}(\psi) d \mu_\infty < \infty.$$
By considering $\tilde{\chi}/|\tilde{\chi}(-1)|$ instead of $\tilde{\chi}$, we can assume that $\tilde{\chi}(-1)=-1$. 
This allows us to apply Theorem \ref{main1} to $u:= u_\infty$, $v:= u_j$, $\theta:= \theta_\infty$, and $\eta:= \theta_j$, and we note that 
$$d_{(A+1)\omega}(u,v)= d_{(A+1)\omega}(u_j, u_\infty)= d_{(A+1)\omega}(\phi_j, \phi_\infty) \to 0 $$
as $j \to \infty$ by the hypothesis. We thus obtain $d_{\capK}(u_j, u_\infty) \to 0$ as $j \to \infty$. The desired convergence hence follows. The proof is finished.
\end{proof}

Theorem \ref{the-hoitucapacitymassnorm} considerably extends \cite[Proposition A]{Guedj-Zeriahi-big-stability} (which treats the case  where the cohomology class is fixed,  $(\phi_j)_j$ is  constant and of minimal singularity types, and only the convergence in $L^1$ was obtained) and  \cite[Theorem 1.4]{Darvas-Lu-DiNezza-singularity-metric} which treats the case where  again the cohomology class is fixed, and $\mu_j$ has $L^p$ density with respect to $\omega^n$; see also  \cite[Theorem 2.14]{DiNezzLu-entropy} for a particular version of  Theorem \ref{the-hoitucapacitymassnorm}. The assumption that $\mu_j$ has $L^p$ density is crucial in the plurisubharmonic envelope approach in \cite[Theorem 1.4]{Darvas-Lu-DiNezza-singularity-metric}. We note furthermore that in  \cite{Trusiani2,Trusiani3} the convergence of solutions were also obtained under more restrictive conditions that the sequence of measures $(\mu_j)_j$ are of uniformly bounded $E^1$ energy (\emph{i.e,} the corresponding solutions $(u_j)_j$ are of uniformly bounded $\chi$-energy for $\chi(t)=t$, see below for the definition) and the sequence of prescribed singularities is totally ordered.

It is well-known that it is not possible to have $u_j \to u_\infty$ in $L^1$ if  $\mu_j$ only converges weakly to $\mu_\infty$ in general  (see \cite{Cegrell-Kolodziej,Guedj-Zeriahi-big-stability} and references therein for examples).

In the sequel, we will proceed to prove Theorems \ref{main3} and  \ref{main2}.

\begin{theorem}\label{main2.varied}
	Let $\theta_1, \theta_2\leq A\omega$ be closed smooth real $(1,1)$-forms ($A\geq 1$). Let $0<\delta \le 1$, $B\geq 1$, $\tilde{\chi} \in \mathcal{W}^-$ and
	$u_j\in\tilde{\mathcal{E}}_{\tilde{\chi}, (A+1)\omega, B\delta}(X, \theta_j)$ ($j=1, 2$) such that $\tilde{\chi}(-1)=-1$, $\sup_Xu_j=0$
	and $\int_X\theta_{u_j}^n\geq \delta$. Assume that there exists a
	concave increasing function $H:\R_{\geq 0}\rightarrow\R_{\geq 0}$ such that, for $j=1, 2$,
	\begin{equation}\label{eq0main2.varied}
		\int_X\min\{|\psi_1-\psi_2|, 1\}(\theta_j+dd^cu_j)^n\leq H(\|\psi_1-\psi_2\|_{L^1(X)}),
	\end{equation}
	for every $\psi_1, \psi_2\in\PSH(X, \omega)$.
	Then, for every $0<a, b, \gamma<1$ and $m>0$,
	there exists $C>0$ depending only on $n, X, \omega, H(1), a, b, \gamma$ and $m$  such that 
	$$d_{\capK}(u_1, u_2)^2\leq C (A\, B)^2 
	\left(h^{\circ n}\left(\dfrac{\delta}{A\, G(\tau)+A^n\|\theta_1-\theta_2\|_{\Cc^0}
		+d_{(A+1)\omega}(u_1, u_2)}\right)\right)^{-\gamma},$$
	where  $\tau=\dist_{-1}((\theta_1+dd^cu_1)^n, (\theta_2+dd^cu_2)^n)$ , 
	$G(\tau)=A^2H^a(A\tau^{a(1-b)})+H(H^m(\tau^{1-b}))+\tau^{ab}$
	and $h(s)=(-\tilde{\chi}(-s))^{1/2}$.
\end{theorem}

\begin{proof}
Without loss of generality, we can assume that $\int_X(\theta_2+dd^cu_2)^n\geq\int_X(\theta_1+dd^cu_1)^n$.
Denote $\mu_1=(\theta_1+dd^cu_1)^n$, $\mu_2=(\theta_2+dd^cu_2)^n$ and $c=\dfrac{\mu_1(X)}{\mu_2(X)}\leq 1$.
It follows from
\cite[Theorem 4.7]{Lu-Darvas-DiNezza-logconcave} that there exists a unique
$u_3\in\mathcal{E}(X, \theta_1, P_{\theta_1}[u_1])$ such that
\begin{equation}
	\begin{cases}
		(\theta_1+dd^cu_3)^n=c\mu_2,\\
		\sup_Xu_3=0.
	\end{cases}
\end{equation}
By Theorem \ref{the MAtocap}, we have
\begin{equation}\label{eq1main2varied}
d_{\capK}(u_1, u_3)^2\leq C_1 (A\, B)^2 
	\left(h^{\circ n}\left(\dfrac{\delta}{A\, G(x)
		+d_{\theta_1}(u_1, u_3)}\right)\right)^{-\gamma},
\end{equation}
where  $x:=\dist_{-1}(\mu_1, c\mu_2)$ and $C_1>0$ depends only on $n, X, \omega, H(1), a, b, \gamma$ and $m$.

By Theorem \ref{main1}, we have
\begin{equation}\label{eq2main2varied}
d_{\capK}(u_2, u_3)^2\leq C_2 (A\, B)^2 
\left(h^{\circ n}\left(\dfrac{\delta}{(1-c)\|\mu_2\|+A^n\|\theta_1-\theta_2\|_{\Cc^0}
	+d_{(A+1)\omega}(u_2, u_3)}\right)\right)^{-\gamma}
\end{equation}
where  $C_2>0$ depends only on $n, X, \omega$ and $\gamma$.

Combining \eqref{eq1main2varied}, \eqref{eq2main2varied} and using the fact 
$d_{\theta_1}(u_1, u_3)=d_{(A+1)\omega}(u_1, u_3)=0$, we get
\begin{equation}\label{eq3main2varied}
	d_{\capK}(u_1, u_2)^2\leq C_3 (A\, B)^2 
	\left(h^{\circ n}\left(\dfrac{\delta}{A\, G(x)+(1-c)\|\mu_2\|
		+R}\right)\right)^{-\gamma},
\end{equation}
where  $R=A^n\|\theta_1-\theta_2\|_{\Cc^0}
+d_{(A+1)\omega}(u_1, u_2)$ and $C_2>0$ is a constant depending only on $n, X, \omega, H(1), a, b, \gamma$ and $m$.

Note that
\begin{equation}\label{eq4main2varied}
	(1-c)\|\mu_2\|=\int_Xd\mu_2-\int_Xd\mu_1\leq\dist_{-1}(\mu_1, \mu_2)=\tau.
\end{equation}
Then
\begin{equation}\label{eq5main2varied}
	x=\dist_{-1}(\mu_1, c\mu_2)\leq \dist_{-1}(\mu_1, \mu_2)+(1-c)\|\mu_2\|\leq 2\dist_{-1}(\mu_1, \mu_2)=2\tau.
\end{equation}
Combining \eqref{eq3main2varied}, \eqref{eq4main2varied} and \eqref{eq5main2varied}, we get
\begin{align*}
		d_{\capK}(u_1, u_2)^2&\leq C_3 (A\, B)^2 
	\left(h^{\circ n}\left(\dfrac{\delta}{A\, G(\tau)+A^n\|\theta_1-\theta_2\|_{\Cc^0}
		+d_{(A+1)\omega}(u_1, u_2)}\right)\right)^{-\gamma},
\end{align*}
where  $C_3>0$ depends only on $n, X, \omega, H(1), a, b, \gamma$ and $m$.

This finishes the proof.
\end{proof}

\begin{proof}[End of the proof of Theorem \ref{main3}]
By the assumption and by \cite[Lemma 3.3]{DinhVietanhMongeampere}, we have 
$\mu_j:=(\theta_j+dd^cu_j)^n$ satisfies \eqref{eq0main2.varied} for
 $H(t)=\tilde{M}\delta t^{\beta}$ and $j=1, 2$, where $\tilde{M}>0$ is a constant depending only on $X, \omega$
 and $M$.
  Moreover, it follows from \cite[Proposition 4.4]{DinhVietanhMongeampere} that,
for every $\psi\in\PSH (X, \omega)$ with $\sup_X\psi=0$, 
$$\int_X-\psi \mu_j\leq B\delta,$$
where $B>0$ depends on $X, \omega, M$ and $\beta$. Hence, by using Theorem \ref{main2.varied} (choose $a=\gamma^{1/2}$, $b=\dfrac{a\beta}{a\beta+1}$ and $m=\dfrac{ab}{(1-b)\beta^2}$), we have
$$d_{\capK}(u_1, u_2)^2\leq C
\left(\dfrac{\tau^{\gamma\beta/(\beta+1)}+\|\theta_1-\theta_2\|_{\Cc^0}
	+d_{(A+1)\omega}(u_1, u_2)}{\delta}\right)^{2^{-n}\gamma},$$
where  $\tau=\dist_{-1}(\mu_1, \mu_2)$ and $C>0$ is a constant depending only on $n, X, \omega, A, M, \gamma$
 and $\beta$.

The proof is completed.
\end{proof}

In order to prove Theorem \ref{main2}, we need again several auxiliary lemmas. 

\begin{lemma}\label{lem key for new1}
	Let $\mu$ be a Radon measure on $X$ vanishing on every pluripolar set. Assume that $u_j$, $j\in\N\cup\{\infty\}$, are negative $\theta$-psh functions satisfying $u_j\rightarrow u_{\infty}$
	in $L^1(X)$ as $j\rightarrow\infty$. Then
	$$\int_X\min\{|u_{j}-u_{\infty}|, 1\}d\mu\rightarrow 0,$$
	as $j\rightarrow\infty$.
\end{lemma}

\begin{proof}
	Denote $B=\sup_j \|u_j\|_{L^1}$. By Chern-Levine-Nirenberg inequality, there exists $C>0$ such that
	$$\capK\{u_j<-k\}\leq \dfrac{B\,C}{k},$$
	for every $j\in\N\cup\{\infty\}$ and $k>0$. Since $\mu$ vanishes on pluripolar sets,  by \cite[Lemma 4.5]{GZ-weighted}, there exists $w \in \PSH(X, \omega)\cap L^{\infty}(X)$  such that $\mu= f \omega_w^n$ for some nonnegative function $f \in L^1(\omega_w^n)$. Let $M>0$ be a big enough constant such that 
	$$\int_{\{f> M\}}d \mu < \epsilon/6.$$	We have 
	\begin{align*}
		\mu(\{u_j<-k\}) &=\int_{\{f> M\} \cap \{u_j< -k\}}d \mu+ \int_{\{f> M\} \cap \{u_j <-k\}}d \mu \\
		&\le \int_{\{f \le  M\} \cap \{u_j< -k\}}d \mu+ \int_{\{f> M\}}d \mu\\
		& \le  M (\sup_Xw-\inf_Xw)\capK\{u_j<-k\}+ \epsilon/6.
	\end{align*}
	It follows that
	for each $\epsilon>0$, there exists $k_0\geq 1$ such that 
	\begin{align}\label{eq1 proof lem key1}
		\mu(\{u_j<-k\}) \le \epsilon/3 
	\end{align}
	for every $j\in\N\cup\{\infty\}$ and $k\geq k_0$. Denote $u_{j, k}=\max\{u_j, -k\}$ and $v_{j,k}=\max\{u_{j,k}, u_{\infty,k}\}$. Thus for every $k$, we have $u_{j,k}\rightarrow u_{\infty, k}$ in $L^1(X)$ and 
	$v_{j,k}\rightarrow u_{\infty, k}$ in capacity as $j\to\infty$. It follows from
	\cite[Lemma 11.5]{GZbook} that
	\begin{equation*}
		\int_X\max\{u_{j,k}-u_{\infty,k}, 0\}d\mu=\int_X(v_{j, k}-u_{\infty, k})d\mu\rightarrow 0,
	\end{equation*}
	and
	\begin{equation*}
		\int_X(u_{j, k}-u_{\infty, k})d\mu\rightarrow 0,
	\end{equation*}
	as $j\to\infty$. Combining the last two convergences gives
	\begin{equation*}
		\int_X|u_{j, k}-u_{\infty, k}|d\mu\rightarrow 0,
	\end{equation*}
	as $j\to\infty$. Choose $j_0$ such that 
	\begin{equation*}
		\int_X|u_{j, k_0}-u_{\infty, k_0}|d\mu<\frac{\epsilon}{3},
	\end{equation*}
	for every $j>j_0$.
	Using the last inequality and \eqref{eq1 proof lem key1}, we have
	\begin{align*}
		\int_X\min\{|u_{j}-u_{\infty}|, 1\}d\mu&\leq\int_{\{u_{j}, u_{\infty}\geq -k_0\}}|u_{j}-u_{\infty}|d\mu
		+\mu(\{u_j<-k\})+\mu(\{u_{\infty}<-k\})\\
		&\leq \int_{\{u_{j}, u_{\infty}\geq -k_0\}}|u_{j, k_0}-u_{\infty, k_0}|d\mu+\dfrac{2\epsilon}{3} \leq\epsilon,
	\end{align*}
	for every $j>j_0$. Thus $\int_X\min\{|u_{a_j}-u_{\infty}|, 1\}d\mu\rightarrow 0$
	as $j\rightarrow\infty$.
\end{proof}

\begin{lemma} \label{le-mychanboiHcap}
	Let $\mu$ be a Radon measure on $X$ vanishing on every pluripolar set. Then, there exists
	a concave, non-decreasing function $H:\R_{\geq 0}\rightarrow\R_{\geq 0}$
	with $H(0)=0$ such that 
		$$\int_X\min\{|u-v|, 1\}d\mu\leq H\left(\|u-v\|_{L^1(X)}\right),$$
	for every $u, v\in\PSH(X, \omega)$.
\end{lemma}

\begin{proof}  For every $t>0$, we denote
	$$h(t)=\sup\left\{\int_X\min\{|u-v|, 1\}d\mu: u,v\in\PSH(X, \omega), \|u-v\|_{L^1(X)}\leq t\right\}.$$
	Then $h$ is non-decreasing. We will show that
	\begin{equation}\label{eq2 lemmucap}
		\lim\limits_{t\to 0^+}h(t)=0.
	\end{equation}
	Indeed, 
	if $\lim\limits_{t\to 0^+}h(t)=2\epsilon>0$ then there exist sequences $u_j, v_j\in\PSH(X, \omega)$
	such that	$\|u_j-v_j\|_{L^1(X)}\rightarrow 0$ as $j\rightarrow\infty$ and 
	\begin{equation}\label{eq3 lemmucap}
		\int_X\min\{|u_j-v_j|, 1\}d\mu\geq\epsilon,
	\end{equation}
	for every $j$. Without loss of generality, we can assume that $c_j:=\sup_Xv_j\leq\sup_X u_j=0$
	
	By the compactness of $\PSH_{\sup}(X, \omega):=\{\varphi\in\PSH(X, \omega): \sup_X\varphi=0\}$ in $L^1(X)$,
	 we can assume that $u_j, v_j-c_j\rightarrow w\in\PSH_{\sup}(X, \omega)$ as $j\rightarrow\infty$.
	 In particular,
	 \begin{equation}\label{eq4 lemmucap}
	 	c_j\int_X\omega^n\leq \|u_j-v_j\|_{L^1(X)}+\|u_j-w\|_{L^1(X)}+\|v_j-c_j-w\|_{L^1(X)}
	 	\stackrel{j\to\infty}{\longrightarrow}0.
	 \end{equation}
	  Moreover, it follows from Lemma \ref{lem key for new1}
	that
	$$\lim\limits_{j\to\infty}\int_X\min\{|u_j-w|, 1\}d\mu=
	\lim\limits_{j\to\infty}\int_X\min\{|v_j-c_j-w|, 1\}d\mu=0,$$
	and it follows that
	\begin{equation}\label{eq5 lemmucap}
		\lim\limits_{j\to\infty}\int_X\min\{|u_j-v_j-c_j|, 1\}d\mu=0.
	\end{equation}
Combining \eqref{eq4 lemmucap} and \eqref{eq5 lemmucap}, we get
$$\lim\limits_{j\to\infty}\int_X\min\{|u_j-v_j|, 1\}d\mu=0.$$
	This contradicts with \eqref{eq3 lemmucap}. Hence, \eqref{eq2 lemmucap} is true.
	
	Now, we put
	$$\tilde{h}(t)=\begin{cases}
		h(t)\quad\mbox{if}\quad 0<t<1,\\
		\int_X\mu\quad\mbox{if}\quad t\geq 1.
	\end{cases}$$
 For every $m>1$, we also define
	$$k_m=\sup\left\{\dfrac{\tilde{h}(s)}{s}: \dfrac{1}{m}\leq t\right\}\qquad\mbox{and}
	\qquad H_m(t)=k_mt+h(1/m).$$
	Then $H_m(t)\geq h(t)$ for every $t\geq 0$ and $\lim_{t\to 0^{+}}H_m(t)=h(1/m)$. Set
	$H(t)=\inf_{m>1}H_m(t).$ We have $H$ is
	a concave, non-decreasing function 
	satisfying $H(0)=0$ and $H\geq h$. In particular,
	$$\int_X\min\{|u-v|, 1\}d\mu\leq H\left(\|u-v\|_{L^1(X)}\right),$$
	for every $u, v\in\PSH(X, \omega)$.
	
	The proof is completed.
\end{proof}

\begin{proof}[End of the proof of Theorem \ref{main2}] 
	By Lemma \ref{le-mychanboiHcap}, there exists
	a concave, non-decreasing function $H:\R_{\geq 0}\rightarrow\R_{\geq 0}$ depending only on $\mu$, $X$ and
	 $\omega$ such that $H(0)=0$ and
	$$\int_X\min\{|\psi_1-\psi_2|, 1\}d\mu\leq H\left(\|\psi_1-\psi_2\|_{L^1(X)}\right),$$
	for every $\psi_1, \psi_2\in\PSH(X, \omega)$.
	
	Moreover, it follows from \cite[Proposition 3.2]{BEGZ} that there exist a constant $B>0$ and a function
	$\tilde{\chi}\in\mathcal{W}^-$ depending on $X, \omega$ and $\mu$ such that
	$$\int-\tilde{\chi}(\psi)d\mu\leq B,$$
	for every $\psi\in\PSH(X, \omega)$ with $\sup_X\psi=0$. In particular,
	 $u_j\in\tilde{\mathcal{E}}_{\tilde{\chi}, (A+1)\omega, (A+1)B}(X, \theta_j)$ for $j=1, 2$. 
	 Hence, by Theorem \ref{main2.varied}, 
	 there exists $C>0$ depending only on $n, X, \omega$ and $H(1)$  such that 
	 $$d_{\capK}(u_1, u_2)^2\leq \dfrac{C(A+1)^2B^2}{\delta^2} 
	 \left(h^{\circ n}\left(\dfrac{\delta}{A\, G(\tau)+A^n\|\theta_1-\theta_2\|_{\Cc^0}
	 	+d_{(A+1)\omega}(u_1, u_2)}\right)\right)^{-1/2},$$
	 where  $\tau=\dist_{-1}(\mu_1, \mu_2)$, $G(\tau)=A^2H^{1/2}(A\tau^{1/4})+H(H(\tau^{1/2}))+\tau^{1/4}$
	 and $h(s)=(-\tilde{\chi}(-s))^{1/2}$.
	 Denote 
	 $$f_{\mu}(t)=\dfrac{C(A+1)^2B^2}{\delta^2} 
	 \left(h^{\circ n}\left(\dfrac{\delta}{A\, G(t)+A^nt}\right)\right)^{-1/2}.$$
	 We obtain
	 $$d_{\capK}(u_1, u_2)^2\leq f_{\mu}\left(\dist_{-1}(\mu_1, \mu_2)+\|\theta_1-\theta_2\|_{\Cc^0}
	 +d_{(A+1)\omega}(u_1, u_2)\right).$$
	 The proof is completed.
\end{proof}

\begin{remark} \label{re-suyratinhday} We explain how to prove Proposition \ref{cor-DDLcomplete} using either Theorem \ref{main2} or \ref{main3} (in place of Proposition \ref{cormodel}). This is almost identical to the proof of Proposition \ref{cor-DDLcomplete} presented above: the new point is to show that there is a subsequence of  $(u_j)_j$ which is convergent in capacity. To this end, we can assume $u_j$'s are model as we did in the above proof of Proposition \ref{cor-DDLcomplete}. Next we extract a subsequence $(u_{j_s})_{s}$ of $(u_j)_j$ such that $u_{j_s}$ converges to some $u$ in $L^1$, and $\theta_{u_{j_s}}^n$ is convergent. Now applying either  Theorem \ref{main2} or \ref{main3} (thanks to (\ref{ine-chantrencuamodelmeasure})), one sees that the sequence $(u_{j_s})_s$ is convergent in capacity. Consequently $u_{j_s} \to u$ in capacity (see, e.g., \cite[Lemma 2.2]{DinhMarinescuVu}).
\end{remark}

\bibliography{biblio_family_MA,biblio_Viet_papers,bib-kahlerRicci-flow}

\bibliographystyle{siam}

\bigskip

\noindent
\Addresses
\end{document}